\documentclass[a4paper,reqno,11pt]{amsart}
\usepackage{amsmath}       
\usepackage[margin=2.5cm]{geometry}
\usepackage[foot]{amsaddr}
\usepackage{amsfonts}      
\usepackage{amssymb}       
\usepackage{amsthm}        
\usepackage{bm}            
\usepackage{enumitem}
\usepackage{graphicx}      
\usepackage{mathtools}     
\usepackage{xcolor}        
\usepackage{dsfont}
\usepackage{tikz}
\usetikzlibrary{positioning,arrows.meta,calc}
\usepackage{caption}

\definecolor{refgreen}{RGB}{0,100,0}
\definecolor{refdarkblue}{RGB}{0, 0, 139}
\definecolor{citegold}{RGB}{160,120,0}
\definecolor{urlblue}{RGB}{0,90,160}

\usepackage[
    colorlinks=true,
    linkcolor=refdarkblue,
    citecolor=citegold,
    urlcolor=urlblue
]{hyperref} 

\usepackage[capitalize]{cleveref}

\usepackage{glossaries}
\expandafter\def\csname ver@etex.sty\endcsname{3000/12/31}

\usepackage{autonum}

\newcommand{\BV}{\mathcal{{BV}}}
\newcommand{\R}{\mathbb{R}}
\newcommand{\D}{\mathrm{D}}
\newcommand{\cD}{\mathcal{D}}
\newcommand{\N}{\mathbb{N}}
\newcommand{\Q}{\mathbb{Q}}
\renewcommand{\P}{\mathcal{P}}
\newcommand{\M}{\mathcal{M}}
\newcommand{\nnu}{{\boldsymbol{\nu}}}
\newcommand{\vv}{\boldsymbol{v}}
\newcommand{\xxi}{{\boldsymbol{\xi}}}

\renewcommand{\L}{\mathcal{L}}
\newcommand{\Ha}{\mathcal{H}}

\newcommand{\spt}{\mathrm{supp}}
\newcommand{\Geo}{\mathrm{Geo}}

\newcommand{\Div}{\mathrm{div}\,}
\newcommand{\lip}{\mathrm{lip}}
\newcommand{\Lip}{\mathrm{Lip}}
\newcommand{\Lipb}{\mathrm{Lip}_{\mathrm{b}}}
\newcommand{\Var}{\mathrm{Var}}
\newcommand{\essVar}{\mathrm{ess\, Var}}
\newcommand{\pr}{\mathrm{Pr}}
\newcommand{\X}{X}
\newcommand{\wsconver}{\overset{*}{\rightharpoonup}}

\makeatletter
\newcommand{\precbar}{\mathrel{\mathpalette\@precbar\relax}}
\newcommand{\@precbar}[2]{%
  \sbox\z@{$\m@th#1\prec$}%
  \dimen@=.054\wd\z@                       
  \dimen@ii=\ifx#1\scriptscriptstyle .148\wd\z@
            \else\ifx#1\scriptstyle .121\wd\z@
            \else .107\wd\z@\fi\fi           
  \rlap{\usebox\z@}%
  \kern\dimexpr\wd\z@-\dimen@ii-\dimen@\relax
  \vrule width\dimen@ height\dimexpr\ht\z@-\dimen@/2\relax
                      depth\dimexpr\dp\z@-\dimen@/2\relax
  \kern\dimen@ii}
\makeatother

\usepackage{xpatch} 
\xpatchcmd{\proof} 
  {\itshape}
  {\bfseries}
  {}
  {}

\usepackage[square,numbers]{natbib}

\numberwithin{equation}{section}
\newtheorem{theorem}{Theorem}[section]
\newtheorem{lemma}[theorem]{Lemma}
\newtheorem{corollary}[theorem]{Corollary}
\newtheorem{proposition}[theorem]{Proposition}

\theoremstyle{definition}
\newtheorem{definition}[theorem]{Definition}

\theoremstyle{definition}
\newtheorem{remark}[theorem]{Remark}

\newtheorem{example}[theorem]{Example}

\renewcommand{\d}{{\mathrm{\,d}}}
\DeclareMathAlphabet{\mathbbmsl}{U}{bbm}{m}{sl}

\makeatletter
\renewcommand\subsubsection{\@startsection{subsubsection}{3}{\z@}%
                                     {.5\linespacing\@plus.7\linespacing}{-.5em}
                                     {\normalfont\bfseries}}

\title[Continuity equation via measure-valued derivations and BV-Wasserstein curves]{Continuity equation on metric spaces via measure-valued derivations and BV-Wasserstein curves}

\author{Ehsan Abedi$^\dagger$}
\author{Zhenhao Li$^\ddag$}
\author{Timo Schultz$^\ast$}

\address{$^\dagger$Institute of Mathematics \\
        TU Berlin\\
        Str. des 17. Juni 136 \\
        10623 Berlin \\
        Germany.}
\address{$^\ddag$Faculty of Mathematics \\
         Bielefeld University \\
         Postfach 10 01 31 \\
         33501 Bielefeld \\
         Germany.}
\address{$^\ast$Department of Mathematics and Statistics\\
    University of Jyv\"askyl\"a \\
    P.O.Box 35\\
    FI-40014\\
    Finland.}
        
\email{ehsan.abedi@tu-berlin.de}
\email{zhenhao.li@math.uni-bielefeld.de}
\email{timo.m.schultz@jyu.fi}

\thanks{}



\begin{document}

\begin{abstract}
    We introduce a notion of continuity equation on metric spaces that is capable of describing curves of probability measures which are absolutely continuous or, more generally, of bounded variation (BV) with respect to the 1-Wasserstein distance. This continuity equation is based on a notion of measure-valued derivations, whose basic theory is also developed in this paper. On $\mathbb{R}^n$, our formulation is consistent with the continuity equation with singular flux introduced by Almi--Rossi--Savar\'e (arXiv:2506.15333), including the corresponding notion of minimal solutions when the norm is strictly convex. In this work, we characterize BV-curves in the space of probability measures equipped with the (extended) 1-Wasserstein distance as those curves satisfying the continuity equation with a measure-valued derivation of finite mass. To this aim, we extend our previous work (Calc.~Var.~(2024)~63:16) on probabilistic representations on BV-curves and construct from them measure-valued derivations (resp. flux measures) on geodesic metric spaces (resp. on $\mathbb{R}^n$).
\end{abstract}

\maketitle
\tableofcontents

\section{Introduction}\label{sec:introduction}
\subsubsection*{Characterization of $p$-absolutely continuous $p$-Wasserstein curves ($p>1$): existing results}
The characterization of curves of probability measures in Wasserstein spaces over metric spaces $(X,d)$ plays an important role in the study of evolution equations, their gradient-flow interpretation and probabilistic representations, as well as in the geometry and analysis of metric measure spaces.
We first recall an existing result on $X = \R^n$, $n \in \N$, with the distance induced by a norm $|\,{\cdot}\,|$.
We denote by $\P (\R^n)$ the space of Borel probability measures on $\R^n$, and by $\P_p(\R^n)$ its subspace of measures with finite $p$-th moment, endowed with the $p$--(Kantorovitch--Rubinstein--)Wasserstein distance $W_p$, where $p > 1$. Let $(\mu_t) \coloneqq (\mu_t)_{t \in I} \subset \P_p(\R^n)$ be a narrowly continuous curve of probability measures indexed by
$t$ in a time interval $I\coloneqq (0,T) \subset \R$. A result of Ambrosio--Gigli--Savar\'e \cite[Chapter 8]{AGS2008GFs} states that the following are equivalent:
\begin{enumerate}[label=(\Alph*), font=\normalfont,itemsep=0.25em]
	\item\label{itm:intro_A} $(\mu_t) \in AC^p(I;\P_{p}(\mathbb{R}^n))$;
	\item\label{itm:intro_B} there exists a path measure $\pi \in \P(C(I;\R^n) )$ concentrated on $AC^p(I;\R^n)$ such that the $p$-energy condition $ \int \int_{I} |\dot{\gamma}_t|^p \d t \d \pi (\gamma) < + \infty $ holds and  $(e_t)_{\#} \pi  = \mu_t$ for all $t \in I$ (where $e_t: \gamma \in C(I;\R^n) \mapsto \gamma_t \in \R^n$ is the evaluation map); 
	\end{enumerate}
    \begin{enumerate}[label=(\Alph*$_{\mathrm{v}}$), font=\normalfont,itemsep=0.25em,start=3]
	\item\label{itm:intro_C} there exists a vector field $\vv \in L^p (I \times \R^n, \d \mu_t \d t  ; \R^n)$ such that the pair $((\mu_t), \vv)$ solves the continuity equation
	\begin{equation}\label{eq:CE_classical}
		\partial_t \mu_t  + \Div(\vv \mu_t ) = 0 \quad \qquad \text{in } I \times \mathbb{R}^{n}, 
	\end{equation}
	in the distributional sense, i.e., for every $\varphi \in C_c^1 (I \times \R^n)$,
	\begin{equation}\label{eq:CE_classical_distributional_solution}
		\int_{I} \int_{\R^n}\partial_t\varphi(t,x)\d \mu_t(x) \d t+	\int_{I} \int_{\R^n} \nabla  \varphi (t,x) \cdot \vv(t,x) \d \mu_t(x) \d t= 0. 
	\end{equation}
\end{enumerate}
In the above, $AC^p (I;Y)$ denotes the set of $p$-absolutely continuous curves in a general metric space $(Y,d_Y)$, and $|\dot{\gamma}_t|$ denotes the metric derivative of such a curve $(\gamma_t)$. 
Item~\ref{itm:intro_B} is usually referred to as a \emph{Lagrangian description}, in which $(\mu_t)$ is represented by the superposition of sample paths $(\gamma_t)$.
This, in particular, provides a probabilistic representation of $(\mu_t)$. Item~\ref{itm:intro_C} is usually referred to as an \emph{Eulerian description}, in which the evolution of $(\mu_t)$ is encoded by a vector field $\vv$ through the continuity equation (CE). 

The extension of Lagrangian characterization \ref{itm:intro_B} to the metric setting $(X,d)$ was established by Lisini \cite{Lisini2007}; see also \cite{Lisini2016,AmbrosioErbarSavare2016}. The extension of the Eulerian formulation \ref{itm:intro_C} is more delicate, as one needs to give meaning to the notions of a gradient and a vector field, or at least to their pairing. Observe that, in \eqref{eq:CE_classical_distributional_solution}, for each fixed time, $\vv(t,x)\cdot\nabla\varphi(t,x)$ can be interpreted as the action of the vector field $\vv(t,\cdot)$ on the spatial test function $\varphi(t,\cdot)$. This action has three important properties: it is linear, satisfies the Leibniz rule, and satisfies the pointwise estimate $|\vv\cdot\nabla\varphi| \leq |\vv| |\nabla\varphi|_*$. In the metric setting, a common replacement for test functions is elements of $\Lipb(X)$ (the set of bounded Lipschitz functions on $X$) multiplied by functions of time.
Furthermore, operators on Lipschitz functions that are linear and satisfy the Leibniz rule are \emph{derivations}, studied e.g. by \cite{Weaver00-JFA, Dimarino2014}.

The perspective above was used by Stepanov--Trevisan \cite{Stepanov-Trevisan2017} to formulate a CE on metric spaces, where the aforementioned action is replaced by the action of a suitable derivation. 
More precisely, fix an algebra $\mathcal A\subseteq \Lipb(X)$, and define the space-time measure $\mu(\mathrm{d} t,\mathrm{d} x)\coloneqq \mu_t(\mathrm d x) \d t$.
A \emph{derivation} $V$ (with reference measure $\mu$) is a linear map
    \begin{equation}\label{eq:derivation_ST}
    	V:\mathcal A\to L^1(I\times X, \mu ),
    \end{equation}
that satisfies the \emph{Leibniz rule} i.e., for every $f,g\in\mathcal A$
    \begin{equation}\label{eq:Leibniz_ST_intro}
	V(fg)=f\,Vg+g\,Vf\,.
    \end{equation}
In addition, we say that it satisfies the \emph{pointwise mass bound} if for every $f\in \mathcal{A}$,
    \begin{equation}\label{eq:mass_bound_ST_intro}
        |V f(t,x)| \leq \Lambda (t,x) \lip_{\rm a} (f) (x), \qquad\text{$\mu$-a.e. } (t,x)\in I\times X,
    \end{equation}
for some (fixed) Borel function $\Lambda: I \times X \to [0,\infty]$ with $\Lambda \in L^1 (I \times X, \mu)$, where $\lip_{\rm a} (f) (x)$ denotes the asymptotic Lipschitz constant of $f$ at $x$; see \eqref{eq:def_lipfx}.
\\
In this setting, the metric analogue of \ref{itm:intro_C} reads as follows:
\begin{enumerate}[label=(\Alph*$_{\mathrm{V}}$), font=\normalfont,itemsep=0.25em,start=3]
	\item\label{itm:intro_C_prime} there exists a derivation $V:\mathcal A\to L^1(I\times X, \mu )$ in the sense of \eqref{eq:derivation_ST}-\eqref{eq:Leibniz_ST_intro} with mass bound \eqref{eq:mass_bound_ST_intro} for some Borel function $\Lambda$ satisfying the $p$-integrability $\int_I \int_X |\Lambda (t,x)|^p \d \mu_t \d t < + \infty$, and the pair $( (\mu_t), V)$ solves the continuity equation, formally written as
    \begin{equation}\label{eq:CE_ST_intro}
        \partial_t \mu_t + \Div (V \mu_t ) =0 \qquad \textrm{in }  I\times X,
    \end{equation}
    meaning that, for every
    $\xi  \in C^1_c(I)$ and $f \in \mathcal{A} \subseteq \Lipb(X)$, 
    \begin{equation}\label{eq:metric_CE}
        \int_I\int_X \frac{\d}{\d t}\xi(t) f(x) \d\mu_t(x) \d t + \int_I\int_X \xi(t) (V f)(t,x) \d\mu_t(x) \d t =0.
    \end{equation}
\end{enumerate}
In particular, \cite[Theorem 4.1]{Stepanov-Trevisan2017} proves that for every $(\mu_t)\in AC^p(I;\P_p(X))$ on a complete metric space $X$, the Eulerian description \ref{itm:intro_C_prime} holds with $\mathcal A=\Lipb(X)$. Their result also gives a Lagrangian representation as in \ref{itm:intro_B}, with additional information encoded through the characteristic ODEs in observables.

When $X=\R^n$, starting from a velocity field $\vv$ and defining $(Vf)(t,x)\coloneqq \vv(t,x)\cdot\nabla f(x)$ on $C_c^1(\R^n)$, the weak formulation \eqref{eq:metric_CE} reduces to the classical case \eqref{eq:CE_classical_distributional_solution} for test functions of product form. Note that, in the above, $V_t$ need not be defined pointwise in time, just as in the classical CE, where the vector field is relevant only up to negligible sets with respect to $\mu$.

Let us finally mention that other formulations of the CE on metric measure spaces were developed by Ambrosio--Trevisan \cite{AmbrosioTrevisan2014} (see also the lecture notes \cite{AmbrosioTrevisan2017}) and by Gigli--Han \cite{GigliHan2015}. These approaches rely on a fixed reference measure on the space; in particular, the latter assumes bounded compression of $(\mu_t)$ with respect to this reference measure.

\subsubsection*{Characterization of BV Wasserstein curves ($p=1$): main result} The situation changes drastically in the case $p=1$.  For the Lagrangian formulation \ref{itm:intro_B}, it is not always possible to represent $1$-absolutely continuous curves in $1$-Wasserstein spaces by path measures on continuous paths $C(I;X)$. This was addressed in the previous work of the authors \cite{AbediLiSchultz2024} by relaxing the path space, namely, by taking a larger path space of c\`adl\`ag curves $D(I;X)$, and by extending Lisini's superposition principle to the case $p=1$. This then allows one to characterize not only $1$-absolutely continuous curves but also curves of bounded variation, shortly BV-curves.

A relaxation is also needed on the Eulerian side \ref{itm:intro_C}. This has been recently developed by Almi--Rossi--Savar\'{e} \cite{AlmiRossiSavare2025} on $\R^n$. 
Observe that in the distributional formulation of classical continuity equation \eqref{eq:CE_classical_distributional_solution}, if we introduce the nonnegative measure $\mu$ on space-time $I\times \R^n$ by $\mu(\mathrm{d} t,\mathrm{d} x)\coloneqq \mu_t(\mathrm{d} x) \mathrm{d} t$, and similarly the vector-valued measure  $\nnu$ on space-time  $I\times \R^n$ by $\boldsymbol{\nu}(\mathrm{d}t,\mathrm{d}x)=\vv(t,x) \mu_t(\mathrm{d} x) \d t$, then clearly $\nnu \ll \mu$, and the Radon--Nikodym derivative $\frac{\d \nnu }{\d \mu} (t,x)  =\vv (t,x)$ is precisely the space-time vector field. The Eulerian characterization of $1$-absolutely continuous curves, and more generally BV-curves, in 1-Wasserstein spaces, requires relaxing the absolute continuity condition $\nnu \ll \mu$.
Instead, one has to allow $\nnu$ to be a general measure, possibly with a singular part with respect to $\mu$ (see \cite[Example 7.1]{AlmiRossiSavare2025} and Example~\ref{ex:jump_between_two_points} in this paper). This leads to the following formulation of CE, studied by \cite{AlmiRossiSavare2025}, where the vector-valued measure $\nnu$ is called a \emph{flux measure}.

    \begin{definition}[CE via flux measures]\label{def:ce}
		A pair $(\mu,\nnu)\in \M^+(I\times \R^n)\times \M(I\times \R^n;\R^n)$ is called a (distributional) solution to the continuity equation,
		\begin{equation}\label{eq:CE}\tag{$\mathrm{CE}_{\nu}$}
			\partial_t \mu+ \Div (\nnu) =0 \qquad \textrm{in }  I\times \R^n,
			\end{equation} 
     if for every $\varphi\in C^1_c(I\times \R^n)$, we have 
    \begin{equation}
    	\int_{I \times \R^n}\partial_t\varphi(t,x)\d \mu(t,x)+	\int_{I\times\R^n} \nabla  \varphi (t,x) \cdot \d \nnu(t,x)= 0. 
    \end{equation}
	\end{definition}
    In the above, $\M^+(I\times\R^n)$ and $\M(I\times\R^n;\R^n)$ denote, respectively, the spaces of nonnegative finite Borel measures and $\R^n$-valued Borel measures with finite total variation on $I\times\R^n$.

    One of the main contributions of this paper is to introduce a metric extension of the CE above, which in particular is used to characterize BV-Wasserstein curves. In analogy with \cite{Stepanov-Trevisan2017}, we replace the flux by a suitable derivation, and, in analogy with the relaxation in \cite{AlmiRossiSavare2025}, we allow the reference measure of this derivation to be an arbitrary measure $\lambda$, not necessarily $\mu$. More precisely, given $\lambda\in\M^+(I\times X)$, we define a \emph{$\lambda$-derivation} $V$ to be a linear map
	\begin{equation}\label{eq:lambda_derivation_intro_1}
		V:\Lipb(X)\to L^\infty(I\times X,\lambda)
	\end{equation}
	such that it satisfies the \emph{Leibniz rule}, i.e., for every $f,g\in\Lipb(X)$,
	\begin{equation}\label{eq:lambda_derivation_intro_2}
		V(fg)=f\,Vg+g\,Vf
		\qquad \lambda\text{-a.e. on } I\times X.
	\end{equation}
    In addition, we say it satisfies the \emph{pointwise mass bound} if for every $f\in \Lipb(X)$, 
    \begin{equation}\label{eq:massbound_intro}
    	|Vf(t,x)|\leq  \lip_{\rm a} (f)(x),\quad  \text{$\lambda$-a.e. on $I\times X$}.
    \end{equation}

    In comparison with the $L^1$-valued derivations satisfying \eqref{eq:mass_bound_ST_intro} discussed earlier, having $L^\infty$-valued derivations satisfying \eqref{eq:massbound_intro} is just a normalization choice, not a restriction\footnote{Indeed, for any derivation $V:\Lipb(X) \to L^1(I\times X,\mu)$ satisfying \eqref{eq:mass_bound_ST_intro} with bound $\Lambda\in L^1(I\times X,\mu)$, one can define a new derivation $\widetilde{V}:\Lipb(X) \to L^\infty(I\times X,\lambda )$ with reference measure $\lambda \coloneqq \Lambda\mu \in\M^+(I\times X)$ satisfying \eqref{eq:massbound_intro} by setting $\widetilde V f\coloneqq (Vf)/\Lambda$ on $\{\Lambda>0\}$ and 0 on $\{\Lambda=0\}$. 
    Then $(Vf)\mu = (\widetilde V f)\lambda$ for every $f\in\Lipb(X)$ and thus the action remains unchanged on the level of CE.}.  Also, imposing \eqref{eq:massbound_intro} with constant $1$ is not a restriction. This condition is in fact equivalent to the norm boundedness of $V$; see Section~\ref{sec:contiDerivation}.
    Following Weaver \cite{Weaver00-JFA,Weaver_book}, we adopt the $L^\infty$-framework, which has the additional advantage of carrying a natural weak$^*$-topology.

    When $X=\R^n$, starting from a flux measure $\nnu \in \M(I\times\R^n;\R^n)$ and writing its polar decomposition $\nnu=\boldsymbol w |\nnu|$, where $|\boldsymbol w|=1$ $|\nnu|$-a.e., one may take $\lambda\coloneqq|\nnu|$ and define the map $(Vf)(t,x)\coloneqq \boldsymbol w(t,x)\cdot\nabla f(x)$ initially for $f\in C_c^1(\R^n)$. This choice gives $Vf\in L^\infty(I\times\R^n,\lambda)$ and $|Vf| \leq |\nabla f|_*$ for $|\nnu|$-a.e. $(t,x)$. This map can then be extended to a derivation on the whole $\Lipb(\R^n)$ when $\nnu$ has divergence; see Proposition~\ref{prop:measuretoderivation} and the definition preceding it.
    The normalization choice above thus corresponds to the choice that the size of the flux is carried by the measure $\lambda$, while the derivation $V$ records its normalized action on test functions. Using the derivation above, we formulate a CE in the metric setting:
    \begin{definition}[CE via derivations]\label{def:CE_mu_lambda_V_intro}
    	We say that the triple $(\mu,\lambda,V)$, consisting of $\mu , \lambda \in \M^+ (I \times X)$ and a $\lambda$-derivation $V\colon \Lipb(X)\to L^\infty(I\times X,\lambda)$ in the sense \eqref{eq:lambda_derivation_intro_1}-\eqref{eq:lambda_derivation_intro_2}, is a (distributional) solution to the continuity equation, formally written as
        \begin{equation}\label{eq:CE_derivation_intro}\tag{$\mathrm{CE}_{V}$} 
            \partial_t \mu + \Div (V \lambda) =0 \qquad \textrm{in }  I\times X,
        \end{equation}
        if for every $\xi  \in C^1_c(I)$ and $f\in \Lipb(X)$, we have
        \begin{equation}\label{eq:CE_derivation_weak_formula_intro}
        \int_{I \times X} \frac{\d}{\d t}\xi(t) f(x) \d \mu(t,x) + \int_{I \times X} \xi(t) (Vf) (t,x) \d \lambda (t,x) = 0.
    \end{equation}
    \end{definition} 

    In the equation above, the pair $(V,\lambda)$ appears only through the measure $(Vf)\lambda$. Hence the choice of $(V,\lambda)$ has some freedom: for instance, if $(\mu,\lambda,V)$ is a solution to the CE, then $(\mu,c \lambda, V /c)$ is also a solution for any positive constant $c$. This suggests encoding the action in a single object, directly at the level of measures, which motivates introducing the following notion. We define a \emph{measure-valued derivation} $\mathcal{D}$ to be a linear map
    \begin{equation}\label{eq:MeasuredD_intro}
    \mathcal{D}\colon \Lipb(X) \to \M(I \times X; \R)
    \end{equation}
    such that it satisfies the \emph{Leibniz rule}, i.e.,  for every $f,g\in \Lipb(X)$,
    \begin{equation}\label{eq:LeibnitzD_intro}
    		\mathcal{D}(fg) = f \mathcal{D}(g) + g \mathcal{D}(f).
    \end{equation}
    In addition, we say it has \emph{finite mass} if for all $f\in\Lipb(X)$,
    \begin{equation}\label{eq:finiteD_intro}
        |\mathcal{D}(f)|\leq \lip_{\rm a} (f)\, \lambda \,,
    \end{equation}
    for some (fixed) measure $\lambda \in\M^+(I\times X)$. 
   The unique minimal measure satisfying \eqref{eq:finiteD_intro} is called the \emph{mass measure} of $\cD$ and is denoted by $|\cD|\in\M^+(I\times X)$; see~\cref{lemma:finitemassD}.
    \\
    Given a $(V,\lambda)$, one can always associate with it a measure-valued derivation $\mathcal D$ by 
    \begin{equation}\label{eq:measured-Derivation_intro}
        \mathcal D f \coloneqq (Vf)  \lambda,
        \qquad  f\in\Lipb(X).
    \end{equation}
    Conversely, every derivation $\mathcal D$ with finite mass admits such a representation. Generally, we say that $(V,\lambda)$ is a \emph{representation} of $\mathcal D$ if \eqref{eq:measured-Derivation_intro} holds. This representation can be regarded as an analogue of the polar decomposition of vector measures on Euclidean spaces. 
    \\
    We now formulate a CE using measure-valued derivations as above, where the action is encoded in a single object. This formulation can be regarded as a direct metric extension of Definition~\ref{def:ce}.

    \begin{definition}[CE via measure-valued derivations]\label{def:CE_mu_D_intro}
        We say that a pair $(\mu,\mathcal{D})$, consisting of $\mu\in \M^+ (I \times X)$ and a measure-valued  derivation $\mathcal{D}\colon \Lipb(X) \to \M(I \times X; \R)$ in the sense of \eqref{eq:MeasuredD_intro}-\eqref{eq:LeibnitzD_intro}, is a (distributional) solution to the continuity equation, formally written as 
        \begin{equation}\label{eq:CE_mD_intro}\tag{$\mathrm{CE}_{\mathcal{D}}$} 
            \partial_t \mu + \Div (\mathcal{D}) =0 \qquad \textrm{in }  I\times X,
        \end{equation}
        if for every $\xi  \in C^1_c(I)$ and $f\in \Lipb(X)$,
    \begin{equation}\label{eq:CE_Mderivation_intro} 
        \int_{I \times X} \frac{\d}{\d t} \xi(t) f(x) \d \mu(t,x) + \int_{I \times X} \xi(t) \d\mathcal{D}f (t,x)= 0.
    \end{equation}
    \end{definition}
    Whenever $(V,\lambda)$ is a representation of $\cD$ in the sense of \eqref{eq:measured-Derivation_intro}, the triple $(\mu,\lambda,V)$ solves \eqref{eq:CE_derivation_intro} if and only if the pair $(\mu,\mathcal D)$ solves \eqref{eq:CE_mD_intro}. With this formulation of CE in the metric setting, we provide a summary of the characterization of BV-Wasserstein curves.

    \begin{theorem}[Characterization of BV-Wasserstein curves]\label{thm:characterization_BV_W1_curves}
        Let $(\X,d)$ be a complete separable metric space, and $I \coloneqq (0,T) \subset \R$. 
        Let $(\mu_t)_{t\in I} \subset \P(\X)$ be a Borel family, and define $\mu \in \M^+(I \times X)$ via $\mu(\mathrm{d}t,\mathrm{d}x) \coloneqq \mu_t (\mathrm{d}x) \d t $. Then the following are equivalent:
        \begin{enumerate}[label=(\Alph*), font=\normalfont,itemsep=0.25em]
        \item\label{itm:BV_characterization_A}  $(\mu_t)\in BV\big(I;(\P(\X),W_1)\big)$;
        \item\label{itm:BV_characterization_B}  there exists $\pi \in \P(D(I;X))$ concentrated on $BV(I;\X)$ such that $\int |\D\gamma| (I)\d\pi(\gamma)< + \infty$ and $(e_t)_{\#} \pi  = \mu_t$ for a.e. $t \in I$ (where $e_t: \gamma \in D(I;X) \mapsto \gamma_t \in X$ is the evaluation map);
        \end{enumerate}
        if, in addition, $(\X,d)$ is a geodesic space, these are also equivalent to:
        \begin{enumerate}[label=(\Alph*$_{\mathcal{D}}$), font=\normalfont,itemsep=0.25em,start=3]
        \item\label{itm:BV_characterization_C1} there exists a finite-mass derivation $\mathcal{D}\colon \Lipb(X) \to \M(I \times X; \R)$ in the sense of \eqref{eq:MeasuredD_intro}-\eqref{eq:finiteD_intro} such that $(\mu,\mathcal{D})$ solves the continuity equation \eqref{eq:CE_mD_intro};
        \end{enumerate}
        and if $\X = \R^n$ endowed with the distance induced by a norm $|\,{\cdot}\,|$, these are also equivalent to:
        \begin{enumerate}[label=(\Alph*$_{\nu}$), font=\normalfont,itemsep=0.25em,start=3]
        \item\label{itm:BV_characterization_C2}  there exists $\nnu\in \M(I\times \R^n;\R^n)$ such that $(\mu,\nnu)$ solves the continuity equation \eqref{eq:CE}.
        \end{enumerate}
        Furthermore, under the corresponding assumptions above on $(X,d)$, the variation measure $|\D\mu|$ of the BV-curve $(\mu_t)$ is characterized by:
        \begin{equation}\label{eq:thm_Dmu}
        |\D \mu| \, = \min_{\pi \text{ as in \normalfont \ref{itm:BV_characterization_B}}} \int |\D \gamma | \d \pi (\gamma) = \min_{\mathcal{D} \text{ as in \normalfont \ref{itm:BV_characterization_C1}}} {\Pr}^I_{\#} |\mathcal{D}| = \min_{\nnu \text{ as in \normalfont \ref{itm:BV_characterization_C2}}} {\Pr}^I_{\#} |\nnu|,
        \end{equation}
        where the minimum is taken with respect to the usual order on measures, $|\cD|$ stands for the mass measure of $\cD$, and $\Pr^I:I \times X\to I$ is the projection onto the time variable.
    \end{theorem}

    In the above, $BV(I;Y)$ denotes the set of curves in $L^0(I;Y)$ with bounded essential variation in an extended metric space $(Y,d_Y)$, allowing $d_Y=\infty$, and $|\D\gamma|$ denotes the variation measure of such a curve $\gamma$ (see \cref{thm:BVequiv} for a set of equivalent definitions of BV-curves in the extended setting, generalizing our previous result in the metric setting \cite[Theorem~2.17]{AbediLiSchultz2024}). In \eqref{eq:thm_Dmu}, we call the minimizers \emph{optimal}, to distinguish the terminology from the notion of \emph{minimal solutions} to CEs discussed later.
    
    The theorem above is proved in this paper through the implications described below. These implications are also shown in Fig.~\ref{fig:diagram}, together with some additional results. 
    \begin{itemize}
    	\item[\textbf{-}] \ref{itm:BV_characterization_A}$\Leftrightarrow$\ref{itm:BV_characterization_B} (Section~\ref{subsec:lift_to_mu}). This is based on our previous work \cite{AbediLiSchultz2024}, where the result was established for BV-curves in the metric space $(\P_1(X),W_1)$. Here we extend the result to BV-curves in the extended metric space $(\P(X),W_1)$, using a truncation argument.
    	
    	\item[\textbf{-}] \ref{itm:BV_characterization_B}$\Rightarrow$\ref{itm:BV_characterization_C2} and \ref{itm:BV_characterization_B}$\Rightarrow$\ref{itm:BV_characterization_C1} (Sections~\ref{subsec:lifts_to_flux measures} and~\ref{subsec:lifts_to_derivation}). This is the key construction in the proof. Starting from a path measure $\pi$ concentrated on BV-curves, we construct a flux measure and a measure-valued derivation solving the corresponding CEs, thereby rigorously verifying the expected hierarchy between these descriptions. This step involves several measurability subtleties, which we address in Appendix~\ref{App:Measurability}.
    	
    	\item[\textbf{-}] \ref{itm:BV_characterization_C1}$\Rightarrow$\ref{itm:BV_characterization_A} and \ref{itm:BV_characterization_C2}$\Rightarrow$\ref{itm:BV_characterization_A} (Sections~\ref{subsec:CE_derivations} and~\ref{sec:consistencyCE}).
        These follow from a priori estimates derived from the corresponding CEs.
        We first prove the first implication. We then establish the relation between the two formulations of the CE. The second implication then follows as a corollary.
    \end{itemize}

When $(\mu_t) \subset \P_1(\R^n)$, the proof above provides an alternative proof of \cite[Theorem~3.4]{AlmiRossiSavare2025}, where \ref{itm:BV_characterization_A}$\Leftrightarrow$\ref{itm:BV_characterization_C2} is shown directly. The aforementioned paper also establishes a probabilistic representation corresponding to solutions of the CE, yielding \ref{itm:BV_characterization_C2}$\Rightarrow$\ref{itm:BV_characterization_B}. Indeed, in \cite[Theorem~6.5 and Corollary~6.6]{AlmiRossiSavare2025}, starting from a solution $(\mu,\nnu)$ to the CE, a probability measure $\widehat{\eta}$ on augmented BV-curves is constructed, whose pushforward through the so-called BV skeleton map yields a probability measure $\pi$ on right-continuous BV-curves with the desired marginals, i.e., \ref{itm:BV_characterization_B}. This augmented representation, however, is finer, since it also records the trajectories at jump times induced by the flux measure and thereby retains its full information.


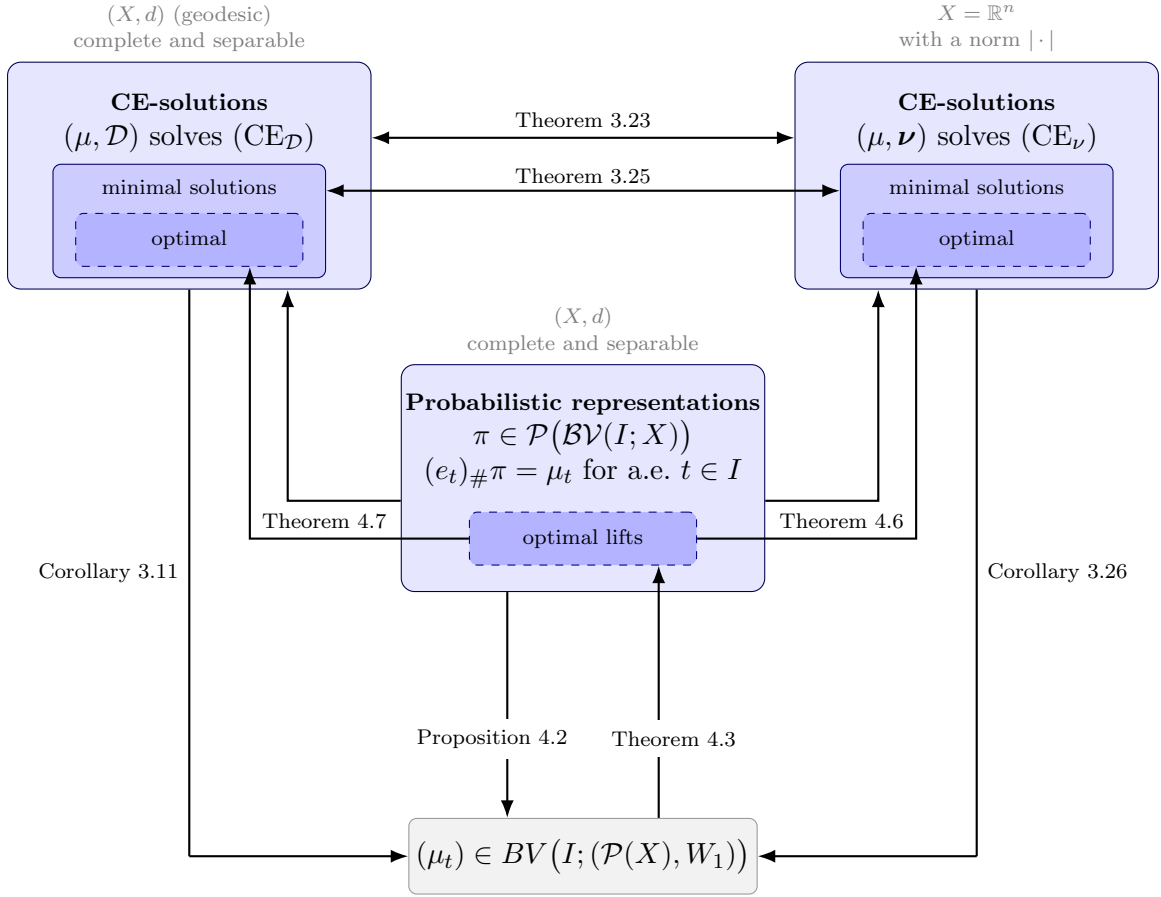
\begin{figure}
\centering

\begin{tikzpicture}[
	>=Latex,
	outer/.style={
		draw=blue!30!black,
		fill=blue!10,
		rounded corners=5pt,
		minimum width=48mm,
		minimum height=30mm
	},
	minset/.style={
		draw=blue!40!black,
		fill=blue!20,
		rounded corners=3pt,
		minimum width=36mm,
		minimum height=15mm
	},
	optset/.style={
		draw=blue!50!black,
		fill=blue!30,
		dashed,
		rounded corners=3pt,
		minimum width=30mm,
		minimum height=7mm
	},
	curvebox/.style={
		draw=gray!80,
		fill=gray!10,
		rounded corners=3pt,
		minimum width=46mm,
		minimum height=10mm
	},
	thmlab/.style={
		font=\scriptsize,
		fill=white,
		inner sep=4pt,
		align=center
	},
	toplab/.style={
		font=\scriptsize,
		text=gray
	}
	]
	
	\coordinate (LiftC) at (0,-4);
	\coordinate (CEDC)  at (-5.2,0);
	\coordinate (CEC)   at ( 5.2,0);
	\coordinate (MuC)   at (0,-9.0);
	
	\node[outer] (lift) at (LiftC) {};
	\node[outer] (ce)   at (CEC)   {};
	\node[outer] (ced)  at (CEDC)  {};
	
	\node[align=center] at ([yshift=5mm]lift.center)
	{\footnotesize\textbf{Probabilistic representations} \\
    
		$\pi \in \P\big(\BV(I;X)\big)$\\
		$(e_t)_{\#}\pi=\mu_t$ for a.e.\ $t\in I$};
	
	\node[align=center] at ([yshift=7mm]ce.center)
	{\footnotesize\textbf{CE-solutions}\\
		$(\mu,\nnu)$ solves \eqref{eq:CE}};
	
	\node[align=center] at ([yshift=7mm]ced.center)
	{\footnotesize\textbf{CE-solutions}\\
		$(\mu,\mathcal{D})$ solves \eqref{eq:CE_mD_intro}};
	
	
	\node[optset] (liftopt) at ([yshift=-8mm]lift.center) {};
	\node[font=\scriptsize, align=center] at (liftopt.center)
	{optimal lifts};
	
	\node[minset] (cemin) at ([yshift=-6mm]ce.center) {};
	\node[font=\scriptsize, align=center] at ([yshift=-3mm]cemin.north) {minimal solutions};
	
	\node[optset] (ceopt) at ([yshift=-2.5mm]cemin.center) {};
	\node[font=\scriptsize] at (ceopt.center) {optimal};
	
	\node[minset] (cedmin) at ([yshift=-6mm]ced.center) {};
	\node[font=\scriptsize, align=center] at ([yshift=-3mm]cedmin.north) {minimal solutions};
	
	\node[optset] (cedopt) at ([yshift=-2.5mm]cedmin.center) {};
	\node[font=\scriptsize] at (cedopt.center) {optimal};
	
	\node[curvebox] (mucurve) at (MuC) {};
	\node[align=center] at (mucurve.center)
	{$(\mu_t)\in BV\big(I;(\P(\X),W_1)\big)$};
	
	\node[toplab, above=0mm of lift, align=center] {$(X,d)$\\complete and separable};
	\node[toplab, above=0mm of ce, align=center] {$X = \R^n$ \\ with a norm $|\,{\cdot}\,|$};
	\node[toplab, above=0mm of ced, align=center] {$(X,d)$ (geodesic)\\complete and separable};
	
	
	\draw[->, thick]
	([yshift=-3mm]lift.west) -| ([xshift=13mm]ced.south);
	
	\draw[->, thick]
	(liftopt.west) -| ([xshift=8mm]cedopt.south)
	node[pos=0.33, above] {\scriptsize Theorem \ref{thm:lifttoderivation}};
	
	
	\draw[->, thick]
	([yshift=-3mm]lift.east) -| ([xshift=-13mm]ce.south);
	
	\draw[->, thick]
	(liftopt.east) -| ([xshift=-8mm]ceopt.south)
	node[pos=0.33, above] {\scriptsize Theorem \ref{thm:lifttoflux}};

	\draw[<->, thick]
	([yshift=5mm]ced.east) -- ([yshift=5mm]ce.west)
	node[midway, above] {\scriptsize Theorem \ref{thm:nu_to_V}};
	
	\draw[<->, thick]
	([yshift=4mm]cedmin.east) -- ([yshift=4mm]cemin.west)
	node[midway, yshift=2mm] {\scriptsize Theorem \ref{thm:correspondence_minimal_solutions}};
	
	
	
	\draw[->, thick]
	([xshift=-10mm]lift.south) -- ([xshift=-10mm]mucurve.north);
	
	\draw[->, thick]
	([xshift=10mm]mucurve.north) -- ([xshift=10mm]liftopt.south);
	
	\path
	([xshift=-10mm]lift.south) -- ([xshift=-10mm]mucurve.north)
	node[pos=0.65, thmlab] {Proposition \ref{prop:lift_to_mut} \,\,\,\,\,\,};
	
	\path
	([xshift=10mm]lift.south) -- ([xshift=10mm]mucurve.north)
	node[pos=0.65, thmlab] {\,\,\,\,\,\, Theorem \ref{thm:optimal_lift}};
	
	
	\coordinate (cedA) at (ced.south);
	\coordinate (cedC) at (mucurve.west);
	\coordinate (cedB) at (cedA |- cedC);
	
	\draw[thick] (cedA) -- (cedB)
	node[midway, left] {\scriptsize Corollary \ref{cor:derivation_CE}};
	\draw[->, thick] (cedB) -- (cedC);
	
	\coordinate (ceA) at (ce.south);
	\coordinate (ceC) at (mucurve.east);
	\coordinate (ceB) at (ceA |- ceC);
	
	\draw[thick] (ceA) -- (ceB)
	node[midway, right] {\scriptsize Corollary \ref{cor:BVflux}};
	\draw[->, thick] (ceB) -- (ceC);
	
\end{tikzpicture}
\captionsetup{font=footnotesize}
\caption{Main objects in the paper and the discussed implications between them, which, in particular, proves the characterization of BV-Wasserstein curves in Theorem~\ref{thm:characterization_BV_W1_curves}.}
\label{fig:diagram}

\end{figure}

\subsubsection*{Further results on continuity equation via measure-valued derivations}
We now highlight some further results concerning the continuity equation. On $\R^n$, \cite{AlmiRossiSavare2025} introduced a notion of minimal solutions for \eqref{eq:CE}, for which they established a superposition principle. In this paper, we introduce an analogous notion of minimal solutions for \eqref{eq:CE_mD_intro}. More specifically, we provide the following:
\begin{itemize}
    \item \textbf{Definition~\ref{def:minimalsolution_M}} (Minimal solutions to \eqref{eq:CE_mD_intro})\textbf{.}
	First, in Definition~\ref{def:subderivation}, we introduce a partial order $\prec$ on finite-mass measure-valued derivations: we call $\cD_1$ a \emph{subderivation} of $\cD$ and write $\cD_1\prec\cD$, if the mass of $\cD$ splits into the mass of $\cD_1$ and the mass of the remainder $\cD-\cD_1$ without cancellation. 
    In Definition~\ref{def:minimalderivation},
    we call a finite-mass derivation \emph{minimal} if it has no nontrivial divergence-free subderivation. This leads to Definition~\ref{def:minimalsolution_M}: a solution $(\mu,\cD)$ to \eqref{eq:CE_D} is called a \emph{minimal solution} if the singular part $\cD^\perp$ in the Lebesgue decomposition of $\cD$ with respect to $\mu$ is minimal. The decomposition $\cD=\cD^a+\cD^\perp$ means that, for every $f\in\Lipb(X)$, the measure $\cD f$ is decomposed into its absolutely continuous and singular parts with respect to $\mu$.
    \item \textbf{\cref{thm:consistency}} (Consistency I, minimality)\textbf{.}
    In Definition~\ref{def:submeasure}, we recall two partial orders on $\R^n$-valued measures: the \emph{submeasure} order $\precbar$ used in \cite{AlmiRossiSavare2025}, and what we call the \emph{metric submeasure} order $\prec$. The submeasure relation is stronger than the metric submeasure relation, and the two coincide when the norm is strictly convex \cite[Lemma 2.2]{AlmiRossiSavare2025}. We have adapted the metric notion to define the subderivations discussed above. In Theorem~\ref{thm:consistency}, we show that a weak$^*$-type continuous derivation $\cD_\nnu$ induced by a flux measure $\nnu$ is minimal if and only if $\nnu$ is minimal with respect to $\prec$.
    \item \textbf{\cref{thm:nu_to_V}} (Consistency II, solutions to CEs)\textbf{.}
	The two formulations of the continuity equation are equivalent on $\R^n$. Every solution $(\mu,\nnu)$ to \eqref{eq:CE} induces a unique (weak$^*$-type) continuous measure-valued derivation $\cD_{\nnu}$ such that $(\mu,\cD_{\nnu})$ solves \eqref{eq:CE_D}.
    Conversely, for every solution $(\mu,\cD)$ to \eqref{eq:CE_D} with a finite-mass derivation $\cD$, the derivation $\cD$ comes from a flux $\nnu$ such that $(\mu,\nnu)$ solves \eqref{eq:CE}.
    \item \textbf{\cref{thm:correspondence_minimal_solutions} }(Consistency III, minimal solutions to CEs)\textbf{.}
	On $\R^n$, let $(\mu,\cD)$ be a solution to \eqref{eq:CE_D} with a finite-mass derivation $\cD$, and $\nnu$ be the associated flux measure under the preceding correspondence. Then $(\mu,\cD)$ is a minimal solution to \eqref{eq:CE_D} if and only if $(\mu,\nnu)$ is a minimal solution to \eqref{eq:CE} with respect to $\prec$. When the norm is strictly convex, the latter notion agrees with that of \cite{AlmiRossiSavare2025}.
\end{itemize}
To illustrate the notion of minimality, we next discuss a simple example.

\subsubsection*{An illustrative example}

As discussed in Sections~\ref{subsec:lifts_to_flux measures}
and~\ref{subsec:lifts_to_derivation}, our construction of fluxes and derivations from a path measure $\pi$ is first performed pathwise and then superposed. Here we present two examples. The first example illustrates this construction for a single curve, which may be viewed as the basic building block of the general construction, and also clarifies the notions of minimality and optimality. The second example illustrates the superposition step.

\begin{example}[Jump between two points]\label{ex:jump_between_two_points}
	Let $y,z\in \R^n$ with $y\neq z$, and consider the curve that stays in $y$ and jumps to $z$ at a fixed time $\alpha\in(0,1)$, i.e.,
	\begin{equation}\label{eq:jump_curve}
		 \gamma_t \coloneqq y \mathds{1}_{(0,\alpha)}(t) + z \mathds{1}_{[\alpha,1)}(t), \qquad t\in I\coloneqq (0,1),
	\end{equation}
	and set $\mu_t \coloneqq \delta_{\gamma_t}$.
	Then $(\mu_t)$ is a right-continuous BV-curve in $(\P(\R^n),W_1)$ with
	\begin{equation}
		|\D\mu| = |z-y|\delta_\alpha .
	\end{equation}
	We define $\mu \in \M^+(I \times \R^n)$ via $\mu(\mathrm{d}t,\mathrm{d}x) \coloneqq \mu_t (\mathrm{d}x) \d t $. Then for every $\varphi\in C_c^1(I\times\R^n)$, we have
	\begin{align}
		\int_{I\times\R^n}
		\partial_t\varphi(t,x)\d\mu(t,x)
		=
		\int_0^\alpha \partial_t\varphi(t,y)\d t
		+
		\int_\alpha^1 \partial_t\varphi(t,z)\d t
		=
		\varphi(\alpha,y)-\varphi(\alpha,z).
	\end{align}
    
    \smallskip
	\noindent
	$\bullet$ \textit{Measure-valued fluxes/derivations satisfying CE.} A vector-valued measure
	$\nnu$ satisfies
	\eqref{eq:CE} if 
	\begin{equation}\label{eq:jump_example_flux_cond}
		\int_{I\times\R^n}
		\nabla\varphi(t,x)\cdot\d\nnu(t,x)
		=
		\varphi(\alpha,z)-\varphi(\alpha,y),
		\qquad
		\forall\varphi\in C_c^1(I\times\R^n).
	\end{equation}
	Obviously, this equation does not determine $\nnu$ uniquely; it only determines its spatial divergence. Similarly, a measure-valued derivation $\cD$ satisfies \eqref{eq:CE_mD_intro} if 
	\begin{equation}\label{eq:jump_example_derivation_cond}
		\int_{I\times\R^n}\xi(t)\d\cD f(t,x)
		=
		\xi(\alpha)\bigl(f(z)-f(y)\bigr),
		\qquad
		\forall \xi\in C_c^1(I),\quad
		\forall f\in\Lipb(\R^n).
	\end{equation}
    
    \smallskip
	\noindent
	$\bullet$ \textit{A family of minimal fluxes.}
	Let $\sigma\in C^1([0,1];\R^n)$ be any curve with end points	$\sigma(0)=y$ and $\sigma(1)=z$, and assume that it is injective and is parametrized
	with constant speed, i.e., $|\dot\sigma(r)|=\ell(\sigma),$ for all $r\in[0,1]$, where $\ell(\sigma)$ denotes the length of $\sigma$ (as will be seen, the construction is invariant under increasing reparametrizations). The right-hand side of \eqref{eq:jump_example_flux_cond} can be written as 
		\begin{align}
		\varphi(\alpha,z)-\varphi(\alpha,y)
		=
		\int_0^1
		\nabla\varphi(\alpha,\sigma(r))
		\cdot\dot\sigma(r)\d r
		=
		\int_{\sigma([0,1])}
		\nabla\varphi(\alpha,x)
		\cdot\frac{\dot\sigma\circ\sigma^{-1}(x)}{\ell(\sigma)}\d\Ha^1(x),
	\end{align}
	 where $\mathcal{H}^1$ is the 1-dimensional Hausdorff measure.
     Hence
	\begin{equation}\label{eq:flux_sigma_jump_example}
		\nnu^\sigma\coloneqq \delta_\alpha \otimes \left(
		\frac{\dot\sigma \circ \sigma^{-1}}{\ell (\sigma)} \,\Ha^1|_{\sigma([0,1])}
		\right)
	\end{equation}
	satisfies \eqref{eq:jump_example_flux_cond}. 
    Since $\sigma$ is
	injective, $\nnu^\sigma$ contains no
	non-trivial divergence-free metric submeasure and thus is minimal in the sense of
	Definition~\ref{def:minimalmeasure}.
    For the flux above, we have the estimate
    \begin{equation}\label{eq:mass_flux_sigma_jump_example}
		\pr^I_\#|\nnu^\sigma|
		=
		\ell(\sigma)\delta_\alpha
		\geq
		|z-y|\delta_\alpha
		=
		|\D\mu|.
	\end{equation}
	In particular, if we take constant-speed geodesic $\sigma^*$ connecting $y,z$, the corresponding flux is
	\begin{equation}\label{eq:optimal_flux_jump_example}
		\nnu^{\sigma*}
		\coloneqq
		\delta_\alpha
		\otimes
		\left(
		\frac{z-y}{|z-y|}
		\,\Ha^1|_{[y,z]}
		\right),
	\end{equation}
	where $[y,z]$ denotes the straight line connecting $y,z$. In this case, we have
	\begin{equation}
		\pr^I_\#|\nnu^{\sigma*}|
		=
		|z-y|\delta_\alpha
		=
		|\D\mu|.
	\end{equation}
    This flux realizes the variation measure of the BV-curve $(\mu_t)$. 
	As shown in \cref{cor:BVflux}, any other flux solving the CE gives an upper bound.
	Thus $\nnu^{\sigma*}$ is an optimal flux.

    \smallskip
	\noindent
	$\bullet$ \textit{A family of minimal derivations.}
    Similarly, the same curve $\sigma$ induces a derivation.
	Indeed, for $f\in\Lipb(\R^n)$, the function $f\circ\sigma$ is Lipschitz
	and hence differentiable $\mathcal L^1$-a.e. The right-hand side of \eqref{eq:jump_example_derivation_cond} can be written as
    	\begin{align}
		\xi(\alpha)\bigl(f(z)-f(y)\bigr)
		=
		\xi(\alpha)\int_0^1 (f\circ\sigma)'(r)\d r
		=
		\xi(\alpha)
		\int_{\sigma([0,1])}
		\frac{(f\circ\sigma)'\circ\sigma^{-1}(x)}
		{\ell(\sigma)}
		\d\Ha^1(x).
	\end{align}
	Accordingly,
	\begin{equation}\label{eq:derivation_sigma_jump_example}
		\cD^\sigma f
		\coloneqq
		\delta_\alpha\otimes
		\left(
		\frac{(f\circ\sigma)'\circ\sigma^{-1}}
		{\ell(\sigma)}
		\,\Ha^1|_{\sigma([0,1])}
		\right),
		\qquad
		f\in\Lipb(\R^n),
	\end{equation}
	satisfies \eqref{eq:jump_example_derivation_cond}. One can check that $\cD$ is linear and satisfies the Leibniz rule. Observe that
	\begin{equation}\label{eq:flux_derivation_relation_jump_example}
		\cD^\sigma f
		=
		\nabla f\cdot\nnu^\sigma,
		\qquad
		f\in C^1(\R^n).
	\end{equation}
    For a.e.\ $r\in(0,1)$, we have $
		|(f\circ\sigma)'(r)|\leq 
		\lip(f)(\sigma(r))\,\ell(\sigma), $
	and consequently one can obtain the mass measure $|\cD^\sigma| = \delta_\alpha\otimes \Ha^1|_{\sigma([0,1])}$.
	In particular,
	\begin{equation}\label{eq:mass_derivation_sigma_jump_example}
		\pr^I_\#|\cD^\sigma|
		=
		\ell(\sigma)\delta_\alpha
		\geq
		|z-y|\delta_\alpha
		=
		|\D\mu|.
	\end{equation}
    Similarly, by taking the constant-speed geodesic $\sigma^*$, one obtains a derivation
    \begin{equation}\label{eq:optimal_derivation_jump_example}
	\cD^{\sigma*} f
	\coloneqq
	\delta_\alpha\otimes
	\left(
	\frac{(f\circ\sigma^*)'\circ(\sigma^*)^{-1}}{|z-y|}
	\,\Ha^1|_{[y,z]}
	\right),
	\qquad
	f\in\Lipb(\R^n),
    \end{equation}
    whose mass measure realizes the variation measure of the BV-curve $(\mu_t)$. As shown in \cref{cor:derivation_CE}, any other derivation solving the CE provides an upper bound.
    
    \smallskip
	\noindent
	$\bullet$ \textit{A family of non-minimal fluxes/derivations.}	As discussed above, if $\sigma$ is not a
	geodesic, i.e., $\ell(\sigma)>|z-y|$, then $\nnu^\sigma$ is minimal but
	not optimal. On the other hand, let $\theta:[0,1]\to\R^n$ be a non-trivial simple
	closed $C^1$ curve, parametrized with constant speed, based at some
	$x_0\in\sigma([0,1])$, and whose image is otherwise disjoint from
	$\sigma([0,1])$. Define $\nnu^{\mathrm{loop}}$ as in \eqref{eq:flux_sigma_jump_example} with $\sigma$ replaced by $\theta$. 
	Since $\theta(0)=\theta(1)$,
	\begin{equation}
		\int_{I\times\R^n}\nabla\varphi(t,x)\cdot
		\d\nnu^{\mathrm{loop}}(t,x)
		=0,
		\qquad
		\forall\varphi\in C_c^1(I\times\R^n),
	\end{equation}
	so $\nnu^{\mathrm{loop}}$ is divergence-free. Therefore $ \nnu^{\sigma,\theta} \coloneqq \nnu^\sigma+\nnu^{\mathrm{loop}}$
	still solves the same CE. Since
	$\nnu^{\mathrm{loop}}\prec\nnu^{\sigma,\theta}$ and
	$\nnu^{\mathrm{loop}}\neq0$, the flux $\nnu^{\sigma,\theta}$ is not
	minimal in the sense of
	Definition~\ref{def:minimalmeasure}.
    \\
    Similarly, the loop $\theta$ yields a non-trivial divergence-free derivation $\cD^{\mathrm{loop}}$ defined as in \eqref{eq:derivation_sigma_jump_example} with $\sigma$ replaced by $\theta$. Hence $\cD^{\sigma,\theta}\coloneqq
	\cD^\sigma+\cD^{\mathrm{loop}}$ still solves the same CE. Since the images of $\sigma$ and $\theta$ are disjoint except at one point, their mass measures are mutually singular, and $\cD^{\mathrm{loop}}\prec\cD^{\sigma,\theta}$. Thus $\cD^{\sigma,\theta}$ contains a non-trivial spatial cycle as a subderivation and is not minimal in the sense of \cref{def:minimalderivation}.

    \begin{figure}[t]
    \centering
    \begin{tikzpicture}[>=Latex]
    
    \begin{scope}[xshift=0cm]
    	\draw (0,0) rectangle (3.8,2.8);
    	
    	\fill (0.5,0.5) circle (1.5pt);
    	\fill (3.2,2.2) circle (1.5pt);
    	\node at (0.3,0.25) {$y$};
    	\node at (3.45,2.45) {$z$};
    	
    	\draw[dashed, thick] (0.5,0.5) -- (3.2,2.2);
    	
    	\node[align=center] at (1.9,-0.6) {\footnotesize optimal (hence) minimal\\ \footnotesize flux $\nnu^{\sigma*}$/derivation $\cD^{\sigma*}$};
    \end{scope}
    
    \begin{scope}[xshift=5.4cm]
    	\draw (0,0) rectangle (3.8,2.8);
    	
    	\fill (0.5,0.5) circle (1.5pt);
    	\fill (3.2,2.2) circle (1.5pt);
    	\node at (0.3,0.25) {$y$};
    	\node at (3.45,2.45) {$z$};
    	
    	\draw[dashed, thick]
    	(0.5,0.5)
    	.. controls (1.2,0.7) and (1.2,1.5) ..
    	(1.8,1.7)
    	.. controls (2.4,1.9) and (2.6,1.6) ..
    	(3.2,2.2);
    	
    	\node at (1.65,1.15) {$\sigma$};
    	\node[align=center] at (1.9,-0.6) {\footnotesize minimal but non-optimal \\ \footnotesize flux $\nnu^\sigma$/derivation $\cD^\sigma$};
    \end{scope}
    
    \begin{scope}[xshift=10.8cm]
    	\draw (0,0) rectangle (3.8,2.8);
    	
    	\fill (0.5,0.5) circle (1.5pt);
    	\fill (3.2,2.2) circle (1.5pt);
    	\node at (0.3,0.25) {$y$};
    	\node at (3.45,2.45) {$z$};
    	
    	\draw[dashed, thick]
    	(0.5,0.5)
    	.. controls (1.2,0.7) and (1.2,1.5) ..
    	(1.8,1.7)
    	.. controls (2.4,1.9) and (2.6,1.6) ..
    	(3.2,2.2);
    	
    	\node at (1.55,1.05) {$\sigma$};
    	
    	\draw[dashed, thick] (1.8,2.0) circle[radius=0.3];
    
        \node at (1.25,2) {$\theta$};
    	\node[align=center] at (1.9,-0.6) {\footnotesize non-minimal (hence) non-optimal\\ \footnotesize flux $\nnu^{\sigma,\theta}$/derivation $\cD^{\sigma,\theta}$};
    \end{scope}
    \end{tikzpicture}
    \captionsetup{font=footnotesize}
    \caption{Spatial supports of the fluxes $\nnu$ and the mass measures $|\cD|$ of the derivations $\cD$ discussed in Example~\ref{ex:jump_between_two_points} at the time slice $t=\alpha$ where the jump occurs.}
    \label{fig:three_fluxes}
    \end{figure}
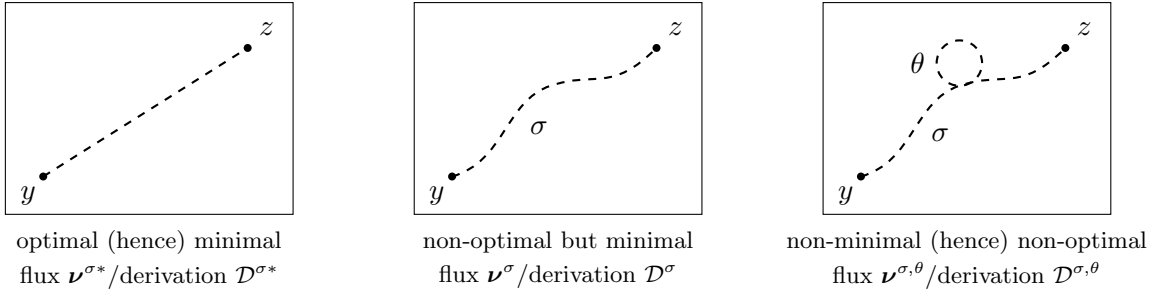

    \smallskip
	\noindent
    The example above thus illustrates the three cases: (1) optimal (hence) minimal, (2) minimal but non-optimal, and (3) non-minimal (hence) non-optimal flux measures, as shown in Fig.~\ref{fig:three_fluxes}.
\end{example}

\begin{example}[Teleportation between two points]\label{ex:teleportation_flux}
	Let $y,z\in\R^n$ with $y\neq z$, and consider
	\begin{equation}
		\mu_t\coloneqq(1-t)\delta_y+t\delta_z,
		\qquad
		t\in I\coloneqq(0,1).
	\end{equation}
	Then $(\mu_t)$ is an absolutely continuous curve in $(\P(\R^n),W_1)$ with
	\begin{equation}
		|\D\mu|
		=
		|z-y| \, \mathcal{L}^1|_I.
	\end{equation}
	For each $\alpha\in I$, denote by $t \mapsto \gamma_t^{(\alpha)}$ the curve defined in \eqref{eq:jump_curve}, and consider the path measure
	\begin{equation}
		\pi^*
		\coloneqq
		(\gamma^{(\cdot)})_\#
		\mathcal L^1|_I .
	\end{equation}
	Then $(e_t)_\#\pi^*=\mu_t$ for every $t\in I$, and moreover $\pi^*$ is an optimal lift of $(\mu_t)$; as discussed in \cite[Example 1.1]{AbediLiSchultz2024}.
    Let $\sigma^*$ denote again the constant-speed geodesic connecting $y,z$.
    By superposing the optimal fluxes
    \eqref{eq:optimal_flux_jump_example} associated with the individual curves $\gamma^{(\alpha)}$, here denoted by $\nnu_{\alpha}^{\sigma*}$, we obtain $\nnu^{\sigma*}\coloneqq\int_{(0,1)}\nnu_{\alpha}^{\sigma*}\d\alpha$, which is given by,
    \begin{equation}\label{eq:optimal_flux_teleportation}
    	\nnu^{\sigma*}
    	=
    	\mathcal L^1\llcorner I
    	\otimes
    	\left(
    	\frac{z-y}{|z-y|}
    	\,\Ha^1\llcorner[y,z]
    	\right),
    \end{equation}
    as also derived in \cite[Example~7.1]{AlmiRossiSavare2025}.
    This flux is optimal, since
    $\pr^I_\#|\nnu^{\sigma*}|=|\D\mu|$. Similarly, superposing the corresponding derivations from
    \cref{ex:jump_between_two_points}, we obtain the derivation
    \begin{equation}
    	\cD^{\sigma*} f
    	\coloneqq
    	\mathcal L^1\llcorner I
    	\otimes
    	\left(
    	\frac{(f\circ\sigma^*)'\circ(\sigma^*)^{-1}}{|z-y|}
    	\,\Ha^1\llcorner[y,z]
    	\right),
    	\qquad
    	f\in\Lipb(\R^n),
    \end{equation}
    whose mass measure satisfies
    $\pr^I_\#|\cD^{\sigma*}|=|\D\mu|$.
    Thus both $\nnu^{\sigma*}$ and $\cD^{\sigma*}$ realize the variation
    measure of the Wasserstein curve $(\mu_t)$.
\end{example}

Finally, beyond the simple illustrative examples above, let us highlight that curves of probability measures on metric spaces with jumps also arise in applications. For example, in filtering with point-process observations, the arrival of an event can instantaneously change the conditional law
of the manifold-valued hidden state \cite[Eq.~(19)]{SuraceKutschireiterPfister2020}.
Another example is the law of a diffusing particle in models of resetting, where the particle's position is instantaneously
contracted towards the origin at prescribed times
\cite[Eqs.~(26)--(28)]{TalFriedmanRoichmanReuveni2022}.

\subsubsection*{Organization of the paper} The rest of the paper is structured as follows. In \textbf{Section~\ref{sec:preliminaries}}, we collect basic notions and notation; in particular, we give several definitions of BV-curves in extended metric spaces and prove their equivalence. In \textbf{Section~\ref{sec:derivations}}, we introduce measure-valued derivations, study their properties, define a notion of minimality, formulate the continuity equation, and establish its consistency with the formulation via fluxes on $\R^n$. 
In \textbf{Section~\ref{sec:characterization}}, we establish probabilistic representations and construct from them flux measures and measure-valued derivations, yielding the main characterization of BV-Wasserstein curves.
Finally, in \textbf{Appendix}, we provide a proof of the completeness of the extended 1-Wasserstein space together with the separability of each of its metric components, as well as measurability results needed in this paper.

\subsubsection*{Acknowledgments} E.A. acknowledges the partial funding by Deutsche Forschungsgemeinschaft (DFG) – Project-ID 318763901 – SFB 1294.

\subsubsection*{Declarations} 
After posting the first version of this manuscript on arXiv, we asked an AI tool, namely, GPT-6 Astra with Medium reasoning effort, to review the manuscript for potential errors. The review identified one incorrect statement in Theorem~\ref{thm:correspondence_minimal_solutions}, where we had overlooked the fact that strict convexity of the norm is needed for the full correspondence between subderivations we defined and submeasures in the sense of Almi--Rossi--Savar\'{e} \cite{AlmiRossiSavare2025}.
The statements are now revised and corrected accordingly. We also corrected several minor issues and typos identified
in the review.

\section{Preliminaries}\label{sec:preliminaries}
\subsection{Basic notions and notations}\label{subsec:prel_basic_notions}
\subsubsection*{(Vector-valued) Measures}
We denote by $\M(\R^m;\R^n)$ the space of Borel $\R^n$-valued measures $\boldsymbol{\mu}$ with finite total variation $\|\boldsymbol{\mu}\|_{{\rm TV}}\coloneqq|\boldsymbol{\mu}|(\R^m)<+\infty$, where for every $B\in \mathcal{B}(\R^m)$ (the latter denotes the Borel $\sigma$-algebra of $\R^m$), 
\begin{equation}
	|\boldsymbol{\mu}|(B)\coloneqq \sup\left\{ \sum^{\infty}_{i=0} |\boldsymbol{\mu}(B_i)|: B_i \in\mathcal{B}(\R^m), \text{\ $B_i$ pairwise disjoint, }  B=\bigcup_{i=0}^\infty B_i  \right\},
\end{equation}
where $|\,{\cdot}\,|$ is a norm on $\R^n$ (not necessarily the Euclidean norm). 
Although $|\boldsymbol{\mu}|$ and
$\|\boldsymbol{\mu}\|_{{\rm TV}}$ depend on the chosen norm on $\R^n$, the space
$\M(\R^m;\R^n)$ does not, since all norms on $\R^n$ are equivalent. We use the standard dot product convention $\boldsymbol{x}\cdot \boldsymbol{y}\coloneqq \sum^n_{i=1}x_iy_i$ for $\boldsymbol{x},\boldsymbol{y}\in \R^n$, and denote by $|\,{\cdot}\,|_*$ the dual norm of $|\,{\cdot}\,|$, namely $
|\boldsymbol{y}|_*
\coloneqq
\sup \{
\boldsymbol{x}\cdot\boldsymbol{y}:
|\boldsymbol{x}|\leq 1
\}$.
Writing $\boldsymbol{\mu}=(\mu_1,\ldots,\mu_n)$ as a vector consisting of $n$ finite signed Borel measures on $\R^m$, we define the integral of every bounded Borel function $\boldsymbol{f}=(f_1,\cdots,f_n) : \R^m\to\R^n$ with respect to $\boldsymbol{\mu}$ as follows
\begin{equation}
	\int_{\R^m} \boldsymbol{f}\cdot \d \boldsymbol{\mu}\coloneqq \sum^n_{i=1}\int_{\R^m} f_i(x)\d\mu_i(x).
\end{equation}
We will also use the dual characterization of variation: for any open set $O\subset \R^m$, we have
\begin{equation}\label{eq:vardual}
    |\boldsymbol{\mu}|(O)=\sup\left\{\int_{\R^m}\boldsymbol{f}\cdot \d\boldsymbol{\mu}: \,\, \boldsymbol{f}\in C_c(\R^m;\R^n), \, \spt(\boldsymbol{f})\subset O, \, \sup_{x\in O}|\boldsymbol{f}(x)|_{*}\leq 1  \right\},
\end{equation}
where $\spt(\boldsymbol{f})$ is the support of $\boldsymbol{f}$.
Unless otherwise specified, a ``finite measure" refers to a non-negative finite Borel measure, and we denote the set of all such measures by $\M^+(\R^m)$. This convention also applies when working on a general metric space.

\subsubsection*{Metric spaces} Let $(X,d)$ be a metric space. 
We denote by $\mathcal{B}(X)$ the Borel $\sigma$-algebra of $X$.
For any $x\in X$ and $r>0$, we denote by $B(x,r)$ the open $d$-ball centered at $x$ with radius $r$.
We denote the set of all Lipschitz functions on $(X,d)$ by $\Lip(X)$, the set of bounded Lipschitz functions by $\Lipb(X)$, and the set of $1$-Lipschitz functions by $\Lip_1(X)\coloneqq\{f\in\Lip(X): \, \Lip(f)\leq 1\}$, where the following notation is used. 
For any $f\in\Lip(X)$, we denote its \emph{global Lipschitz constant} by
\begin{equation}
\Lip(f)\coloneqq
\sup_{y\neq z}\frac{|f(y)-f(z)|}{d(y,z)},
\end{equation}
and define its \emph{asymptotic Lipschitz constant} pointwise for $x\in X$ by
\begin{equation}\label{eq:def_lipfx}
	\lip_{\rm a}(f)(x)\coloneqq \lim_{r\to 0^+ }\sup_{\substack{y\neq z \\ y,z\in B(x,r)}} \frac{|f(y)-f(z)|}{d(y,z)}.
\end{equation}
Notice when $X=\R^n$ and it is endowed with the distance induced by a norm $|\,{\cdot}\,|$, then for every $f\in C^1(\R^n)$ and $x\in\R^n$, we have $\lip_{\rm a}(f) (x)=|\nabla f(x)|_{*}$, where $|\,{\cdot}\,|_*$ is the dual norm.

\subsubsection*{Probability measures and 1-Wasserstein distance}
Let $(X,d)$ be a complete separable metric space.
We denote by $\P(X)$ the set of all Borel probability measures on $X$, and by $\P_p(X)$ its subspace of measures with finite $p$-th moment, where $p \geq 1$. The \emph{extended} $1$-Wasserstein metric between any $\textup{m}_0,\textup{m}_1\in\P(X)$ is defined as 
\begin{equation}
	W_1(\textup{m}_0,\textup{m}_1)\coloneqq \inf \left\{ \int_{X \times X} d(x,y)\d\sigma(x,y) : \,\, \sigma \in \mathrm{Cpl} (\textup{m}_0,\textup{m}_1)  \right\} \in [0,+\infty],
\end{equation}
where $\mathrm{Cpl} (\textup{m}_0,\textup{m}_1) \subset \P (X \times X)$ denotes the set of couplings of $\textup{m}_0$ and $\textup{m}_1$, i.e., the set of measures on the product space whose marginals are $\textup{m}_0$ and $\textup{m}_1$. 
We will also use the dual characterization of the $1$-Wasserstein distance, namely, the Kantorovich--Rubinstein formula 
\begin{equation}\label{eq:Kduality}
	W_1(\textup{m}_0,\textup{m}_1)=\sup\left\{\int_X f\d\textup{m}_0-\int_X f\d\textup{m}_1: \,\, f\in \Lipb(X), \, \Lip(f)\leq 1\right\},
\end{equation}
for all $\textup{m}_0,\textup{m}_1\in\P(X)$, see e.g. \cite[Theorem 1.14]{Villani_topics}. We discuss further properties of the extended metric space $(\P(\X),W_1)$ in Appendix \ref{App:Ext_Wass_Space}. 

\subsubsection*{Spaces of curves}
Let $(X,d)$ be a metric space. 
Unless otherwise specified, in this paper, we denote by $I$ the open interval $(0,T)$ for some $T\in(0,\infty)$.
We say that a curve $\gamma\colon I\to X$ is càdlàg if it is right-continuous with left-limits.
The set $D(I;X)$ stands for the space of all càdlàg curves, endowed with the Skorokhod topology; see \cite{Billingsley} or \cite[Section 2.3]{AbediLiSchultz2024}.

Moreover, we denote by $L^0(I;X)$ the set of Borel curves, after identifying $\mathcal L^1$-a.e. equal curves, and denote by $L^1(I;X)$ the space of all $\gamma\in L^0(I;X)$ such that $\int_I d(\gamma_t,\bar{x})\d t<\infty$ for some (and thus any) $\bar{x}\in X$.

Finally, we say that $(\mu_t)_{t \in I} \subset \P (X)$ is a Borel family if $t \mapsto\mu_t (B)$ is a Borel map for any Borel set $B \subset \X$. 
In this case, the formula
\begin{equation}\label{eq:def_mu_from_mut}
    \int_{I \times X} f(t,x) \d \mu (t,x) \coloneqq \int_{I} \left( \int_\X f(t,x) \d \mu_t (x) \right) \d t, \quad \forall f \in \mathcal{B}_{\mathrm{b}} (I \times X;\R),
\end{equation}
defines a unique measure $\mu \in \M^+ (I\times X)$ which will be shortly denoted by $\mu(\mathrm{d} t,\mathrm{d} x)\coloneqq \mu_t(\mathrm{d} x) \mathrm{d} t$. In the above and throughout the paper, $\mathcal{B}_{\mathrm{b}}(Y;\R)$ denotes
the space of bounded Borel measurable functions on a topological space $Y$.

\subsection{BV-curves in \texorpdfstring{$\mathbb{R}^n$}{Rn}} Let $I \coloneqq (0,T) \subset \R$, and let $\R^n$ be endowed with the distance induced by a norm $|\,{\cdot}\,|$.

\begin{definition}[BV-curves in $\mathbb{R}^n$]\label{def_BV_curves_Euclidean}
    A curve $ \gamma \in L^1(I;\R^n)$ is called a BV-curve if its distributional derivative is given by a finite $\R^n$-valued measure, i.e., there exists a measure $\boldsymbol{\D\gamma} \in \M (I;\R^n)$ such that
\begin{equation}\label{eq:distributionalD}
	\int_I \frac{\d}{\d t}\xi(t)\cdot \gamma(t)\,\d t
	+ \int_I \xi(t)\cdot \d\boldsymbol{\D\gamma}(t) = 0
\end{equation}
holds for every $\xi\in C^1_c(I;\R^n)$. The space of BV-curves is denoted by $ BV (I;\R^n)$.
\end{definition}

The distributional derivative $\boldsymbol{\D\gamma}$ can be decomposed into the continuous part $\boldsymbol{\D\gamma}^{\mathrm{c}}$ (which itself contains absolutely continuous and  continuous singular parts) and the pure jump part $\boldsymbol{\D\gamma}^{\mathrm{j}}$:
\begin{equation}\label{eq:decomposition}
	\boldsymbol{\D\gamma} = 	\boldsymbol{\D\gamma}^{\mathrm{c}}+	\boldsymbol{\D\gamma}^{\mathrm{j}}.
\end{equation}
The jump part $\boldsymbol{\D\gamma}^{\mathrm{j}}$ can be written as
\begin{equation}
    \boldsymbol{\D\gamma}^{\mathrm{j}}=\sum_{t\in J_\gamma}( \gamma_{t+}-\gamma_{t-})\delta_t,
\end{equation}
where $J_\gamma \subset I$ denotes the set of jump points of $\gamma$, which for BV-curves is at most countable.

\begin{lemma}[Chain rule for distributional derivatives]\label{lemma:chainrule}
	Let $\varphi\in C_c^1( I\times \R^n; \R)$ and let $\gamma\in BV(I;\R^n)$ be  c\`adl\`ag.
	Denote $\varphi_\gamma(t)\coloneqq \varphi(t,\gamma_t)$ and $\tilde \gamma(t)\coloneqq (t,\gamma_t)$ for all $t\in I$.
	Then $\varphi_\gamma$ is a c\`adl\`ag BV-curve in $\R$, and its distributional derivative is given by 
	\begin{equation}
		\mathrm{D} \varphi_\gamma=\partial_t \varphi \circ \tilde\gamma \, \L^1 \llcorner I +\nabla \varphi \circ \tilde\gamma \cdot \boldsymbol{\mathrm{D} \gamma}^{\mathrm{c}}+\sum_{t\in J_\gamma}[\varphi(t,\gamma_{t+})-\varphi(t,\gamma_{t-})]\delta_t.
	\end{equation}
\end{lemma}
\begin{proof}
	The function $\varphi_\gamma \coloneqq \varphi \circ \tilde\gamma$ belongs to $BV (I;\R)$ since it is the composition of a global Lipschitz function and a BV-curve $\tilde\gamma$. It is also  c\`adl\`ag since it is the composition of a continuous function and a c\`adl\`ag curve $\tilde\gamma$. By applying Vol'pert's chain rule (see e.g. \cite[Eq.\,(0.1) and (0.2)]{Ambrosio-DalMaso90} and the original works \cite{Volpert1967,VolpertHudjaev1985}), using the fact that approximate limit (as defined in \cite[Eq.\,(1.4)]{Ambrosio-DalMaso90}) of a c\`adl\`ag curve $\tilde\gamma$ coincides with itself on $I \setminus J_{\tilde\gamma}$, and the obvious equality of jump sets $J_{\tilde \gamma}=J_\gamma$, we have 
	\begin{align}
		\mathrm{D} \varphi_\gamma
		& = \big( \partial_t \varphi, \nabla \varphi \big) \circ \tilde\gamma \cdot \big( \L^1 \llcorner I,  \boldsymbol{\mathrm{D} \gamma}^{\mathrm{c}}  \big) +\sum_{t\in J_{\tilde\gamma}}[\varphi(t+,\gamma_{t+}) -\varphi(t-,\gamma_{t-})]\delta_t \\
		& = \partial_t \varphi \circ \tilde\gamma \, \L^1 \llcorner I
  + \nabla \varphi \circ \tilde\gamma \cdot \boldsymbol{\mathrm{D} \gamma}^{\mathrm{c}}
  + \sum_{t\in J_\gamma}[\varphi(t,\gamma_{t+})-\varphi(t,\gamma_{t-})]\delta_t.
  \tag*{\qedhere}
  \end{align}
\end{proof}

\subsection{BV-curves in metric spaces} Let $(X,d)$ be a metric space, and $ I \subset \R$ be arbitrary. Given $\gamma: I \to X$, we recall that the \emph{pointwise variation} of $\gamma$ on any $A \subset I$ is defined as
    \begin{equation}\label{def:pointwise_variation}
        \Var (\gamma ; A)\coloneqq\sup\left\{\sum_{i=0}^k d(\gamma_{t_i},\gamma_{t_{i + 1}}) \big| \, \{ t_i\}_{ 0 \leq i \leq k + 1}  \subset A \text{ with } t_0 <\cdots < t_{k + 1} \right\},
    \end{equation}
    and we write $\Var (\gamma ) \coloneqq  \Var (\gamma; I ).$

Since pointwise variation is sensitive to modifications of the function’s value at single points, one considers equivalence classes
and defines essential variation that disregards such negligible changes. 

Now let  $I \coloneqq (0,T) \subset \R$. Given $\gamma: I \to X$, the essential variation of $\gamma$ on any $A \subset I$ is defined as
\begin{equation}\label{eq:def_essVar}
        \essVar (\gamma ; A) \coloneqq \inf \big\{ \Var (\tilde\gamma ; A) \big| \gamma = \tilde\gamma \text{ a.e. on } A \big\},
    \end{equation}
    and we write $ \essVar (\gamma ) \coloneqq \essVar (\gamma ; I)$. 

\begin{definition}[BV-curves on metric spaces]\label{def:BV_metric} 
    We define the space of BV-curves as follows
    \begin{equation}\label{eq:def_BV_metric}
       BV (I;\X) \coloneqq \big\{ \gamma \in L^1(I;X): \essVar (\gamma) < +  \infty \big\}.
    \end{equation}
    Furthermore, we denote by $\BV(I;X)$ the set of all càdlàg curves of bounded essential variation.
\end{definition}

For any $\gamma\in BV(I;X)$, we define the {total variation measure} of $\gamma$, denoted by $|\D \gamma|$, as the Borel measure\footnote{The existence of such measure can be seen by using \cite[Lemma 2.5]{AbediLiSchultz2024} on the metric completion of $X$.} on $I$ such that for all $(a,b)\subset I$, 
\begin{equation}\label{eq:goodrepresent}
    |\D \gamma|((a,b))=\essVar(\gamma; (a,b)).
\end{equation}
Note that the notions in the above definitions are consistent with those defined in \cref{def_BV_curves_Euclidean}.
Indeed, by \cite[Theorem 3.28]{Ambrosio-Fusco-Pallara2000} (the proof also applies to $\R^n$ with general norms), any $\gamma\in BV(I;\R^n)$ admits a c\`adl\`ag representative $\gamma^+$ such that for every interval $(a,b)\subseteq I$,
\begin{equation}
	\Var(\gamma^+;(a,b)) = |\boldsymbol{\D\gamma}|((a,b)).
\end{equation}
For c\`adl\`ag curve $\gamma^+$ it is known that $\Var(\gamma^+;(a,b))=\essVar(\gamma^+;(a,b))$; see e.g. \cite[Lemma 2.5]{AbediLiSchultz2024}.
Hence $|\D \gamma| = |\boldsymbol{\D\gamma}|$ as measures.
In other words, the total variation measure of a BV-curve on $\R^n$ coincides with the variation of its distributional derivative, and thus we will not distinguish between them throughout this paper.

\subsection{BV-curves in extended metric spaces}\label{subsec:BV_extended}
Let $(X,d)$ be an extended metric space.
Given $x\in X$, denote by $L^1_x(I;X)$ the set of equivalence classes of Borel curves $\gamma$ with
\begin{align}
    \int_I d(\gamma_t,x)\d t< + \infty,
\end{align}
where the equivalence is defined by equality $\L^1$-almost everywhere.
Moreover, write
\begin{align}
    L^1(I;X)=\bigcup_{x\in X}L^1_x(I;X).
\end{align}
We equip $L^1_x(I;X)$ and $L^1(I;X)$ with the obvious (extended) $L^1$-distances.
Denote also by $L^0(I;X)$ the set of Borel curves, again up to identifying almost everywhere equal curves.

\begin{definition}[BV-curves on extended metric spaces]\label{def:BV_extended}
    Let $(X,d)$ be an extended metric space.
    The space of BV-curves on $X$ is denoted by
      \begin{equation}\label{eq:def_BV_metric_extended}
       BV (I;\X) \coloneqq \big\{ \gamma \in L^0(I;X): \essVar (\gamma) < +  \infty \big\}.
    \end{equation}
\end{definition}
Note that this definition is well-posed. 
Firstly, when $(X,d)$ is a metric space, \eqref{eq:def_BV_metric_extended} gives the same space as the one defined in \cref{def:BV_metric}.
Secondly, it includes all curves with finite pointwise variation, which is a consequence of the following lemma.

\begin{lemma}
 Let $(X,d)$ be an extended metric space.
 Let $\gamma\colon I\to (X,d)$ be a curve with $\Var(\gamma)<+\infty$. 
 Then $\gamma$ is Borel measurable.   
\end{lemma}
\begin{proof}
    Since the pointwise variation is finite, and all metric components are open sets, we may assume that $(X,d)$ is a metric space.
    Since the $\sigma$-algebra generated by all Lipschitz functions coincides with the Borel $\sigma$-algebra of $X$, it suffices to prove that $f\circ \gamma$ is Borel measurable for all Lipschitz functions $f$.
    But this is a known fact, as $f\circ \gamma$ is a real-valued function of bounded pointwise variation.
\end{proof}

Let us introduce an auxiliary notion of variation.
This is motivated by the characterisation in \cref{thm:BVequiv}.
\begin{definition}
    Let $X$ be an extended metric space.
    Let $\gamma\colon I\to X$ be a Borel map.
    Define
    \begin{align}
        \essVar^*(\gamma)\coloneqq \inf \{\Var(\tilde\gamma): \tilde\gamma\colon \tilde I\to X,\ \L^1(I\setminus \tilde I)=0,\ \tilde\gamma_t=\gamma_t \forall\ t\in \tilde I\}.
    \end{align}
\end{definition}

\begin{remark}
    Since the pointwise variation of a curve is a monotone set function, the modified essential variation is a well-defined function on $L^0(I;X)$.
\end{remark}

\begin{lemma} \label{lma:L0toL1}
Let $(X,d)$ be an extended metric space, and let $\gamma\in L^0(I;X)$ be such that $\essVar^*(\gamma)<\infty$.
Then $\gamma\in L^1(I;X)$.
In particular, $\gamma\in L^1(I;X)$ for every BV-curve $\gamma\in L^0(I;X)$.
\end{lemma}
\begin{proof}
    Let $\gamma$ be a representative of the BV-curve with $\Var(\gamma;\tilde I)<\infty$ for some full-measure subset $\tilde I\subset I$.
    Let $x\coloneqq \gamma_{t_0}$ for some $t_0\in \tilde I$.
    Then $d(\gamma_t,x)=d(\gamma_t,\gamma_{t_0})\le \Var(\gamma;\tilde I)$ for all $t\in \tilde I$.
    In particular, $\int_{I} d(\gamma_t,x)\d t=\int_{\tilde I} d(\gamma_t,x)\d t\le T \Var(\gamma;\tilde I)<\infty$.
    Thus, $\gamma\in L_{x}^1$.
\end{proof}

\begin{lemma}\label{lma:esssepimage}
    Let $(X,d)$ be an extended metric space, and $\gamma\colon [0,1]\to X$ a Borel map.
    Then there exists a set $N\subset [0,1]$ with $\L^1(N)=0$ so that $\gamma([0,1]\setminus N)$ is separable.
\end{lemma}
\begin{proof}
Define $\tilde d(x,y)\coloneqq d(x,y)\wedge 1$.
Then $(X,\tilde d)$ is a metric space homeomorphic to $(X,d)$.
By \cite[451Q]{Fremlin}, since the Lebesgue measure is a Radon measure, there exists a closed separable set $X_0\subset X$ such that $\L^1(\gamma^{-1}(X\setminus X_0))=0$.
Thus, $N=\gamma^{-1}(X\setminus X_0)$ is the desired set.
\end{proof}

\begin{lemma}\label{lma:modessvar} Let $(X,d)$ be an extended metric space.
    \begin{enumerate}
        \item If $(X,d)$ is complete, then $\essVar=\essVar^*$.
        \item In general, $\essVar^*_{X}(\gamma)=\essVar_{ \bar X}(\gamma)$, where $\bar X$ is the completion of $X$.
    \end{enumerate}
\end{lemma}
\begin{proof}
    In general, $\essVar^*\le \essVar$.
    Let us show that $\essVar^*\ge \essVar$ when $X$ is complete.
    Let $\gamma\colon I\to X$ be a Borel curve, and let $\tilde\gamma\colon \tilde I\to X$, where $\L^1(\tilde I)=T$ and $\tilde\gamma_t=\gamma_t$ for all $t\in \tilde I$.
    Assume also that $\Var(\tilde\gamma;\tilde I)<\infty$.
    Then by \cite[Remark 2.6]{AbediLiSchultz2024}, since $X$ is complete, we can extend $\tilde \gamma$ to the whole interval $I$ so that $\Var(\tilde\gamma;I)=\Var(\tilde\gamma;\tilde I)$.
    In particular,
    \begin{align}
        \Var(\tilde\gamma;\tilde I)=\Var(\tilde\gamma;I)\ge \essVar(\gamma).
    \end{align}
    By taking infimum over all such curves $\tilde \gamma$, we get that
    \begin{align}
        \essVar^*(\gamma)\ge \essVar(\gamma).
    \end{align}
    Let us now show, in the case of the general metric space, $\essVar^*(\gamma)=\essVar_{ \bar X}(\gamma)$.
    By definition and by the first claim, we have that
    \begin{align}
        \essVar^*_{ X}(\gamma)\ge \essVar^*_{{ \bar X}}(\gamma)=\essVar_{ \bar X}(\gamma).
    \end{align}
    On the other hand, if $\hat \gamma\colon I\to \bar X$ is such that $\gamma=\hat \gamma$ almost everywhere, then there exists $\hat I\subset I$ with $\L^1(\hat I)=T$ such that $\hat\gamma_t=\gamma_t\in X$ for all $t\in \hat I$ and $\Var_{ \bar X}(\hat\gamma)\ge\Var_{ X}(\gamma\lvert_{\hat I})$.
    Hence,
    \begin{equation}
        \essVar^*_X(\gamma)\le \inf_{\hat\gamma}\Var_X(\gamma\lvert_{\hat I})\le \inf_{\hat\gamma}\Var_{\bar X}(\hat\gamma)=\essVar_{\bar X}(\gamma). 
        \tag*{\qedhere}
    \end{equation}
\end{proof}

 \begin{theorem}[Equivalent definitions of BV-curves in extended metric spaces]\label{thm:BVequiv}
     Let $(X,d)$ be an extended metric space, and $I \coloneqq (0,T) \subset \R$. Let $\gamma\in L^0(I;X)$.
     Then the following are equivalent:
     \begin{enumerate}[label=(\arabic*), font=\normalfont]
         \item $\essVar^*(\gamma)<+\infty$;
         \item $\sup_{0<h<T}\int_{[0,T-h]}
         \frac{d(\gamma_{t+h},\gamma_t)}{h}\d t<+\infty$;
         \item there exists a finite Borel measure $\mu\in \M^+(I)$ such that for all $\varphi\in \Lip_1(X,d)$, the function $\varphi\circ \gamma$ is a BV-function and
         \begin{align}
             |D(\varphi\circ\gamma)|\le \mu;
         \end{align}
         \item there exists a finite Borel measure $\mu\in \M^+(I)$ such that
         \begin{align}
             d(\gamma_s,\gamma_t)\le \mu([s,t])
         \end{align}
         for $\L^2$-almost every $(s,t)\in I^2$ with $s<t$;
     \end{enumerate}
     if, in addition $(X,d)$ is complete, these are also equivalent to: 
     \begin{enumerate}[label=(\arabic*), font=\normalfont,start=5]
     \item $\gamma$ is a BV-curve (in the sense of \Cref{def:BV_extended});
     \end{enumerate}
     and $\gamma$ admits a càdlàg representative $\tilde\gamma$ such that $\essVar (\gamma;(a,b)) =\Var\,(\tilde\gamma;(a,b))$ for $\forall (a,b) \subseteq I$.
 \end{theorem}
 
\begin{remark}[Equivalence of variation measures of BV-curves in extended metric spaces]\label{remark:variationmeasure}
    Each of the conditions $(1)-(4)$ induces a minimal measure with respect to it.
         Moreover, these minimal measures coincide.
         Namely, for a BV-curve $\gamma$, there exists a measure $|\D\gamma|$ satisfying conditions $(a)-(d)$ below such that if $\mu$ is another measure satisfying \emph{any} of them, then $|\D\gamma|\le \mu$.
         \begin{enumerate}[label=(\alph*)]
             \item $\essVar^*(\gamma;(a,b))\le \mu((a,b))$ for all $a,b\in I$;
             \item $\sup_{0<h<b-a}\int_a^{b-h}d(\gamma_{t+h},\gamma_t)/h\d t\le \mu((a,b))$ for all $a,b\in I$;
             \item $|\D(\varphi\circ \gamma)|\le\mu$ on $I$ for all $\varphi\in \Lip_1(X,d)$;
             \item $d(\gamma_s,\gamma_t)\le \mu([s,t])$ for $\L^2$-almost every $(s,t)\in I^2$ with $s<t$.
         \end{enumerate}
 \end{remark}
 
 \begin{proof}[Proof of \Cref{thm:BVequiv}]
 Let $\gamma\in L^0(I;X)$ be such that it satisfies any of the conditions $(1)-(4)$.
 We show first that $\gamma\in L^1(I;X)$.\medskip
\\ 
 The case (1) is the content of \cref{lma:L0toL1}.
    Both conditions (2) and (4) imply that
    \[\iint_{I^2}d(\gamma_s,\gamma_t)\d s\d t< +\infty, \]
    from which we deduce that $\gamma\in L^1_{\gamma_t}$ for almost every $t\in I$.
    \\
    Assume now that the condition $(3)$ is satisfied.
    Writing $\phi_{x,N}\coloneqq d(x,\cdot)\wedge N$ for $x\in X$ and $N\in \N$, we have $\phi_{x,N}\in \Lip_1(X)$.
    In particular, $\phi_{x,N}\circ \gamma$ is a BV-function, and (due to the equivalence of (3) and (4) on $\R$) 
    \[|\phi_{x,N}(\gamma_s)-\phi_{x,N}(\gamma_t)|\le \mu([s,t])\]
    for $\L^2$-almost every $(s,t)$ with $s<t$.
    On the other hand, for any $x,y\in X$, we have
    \begin{align}
        d(x,y)=\sup_{z\in X,N\in\N}|\phi_{z,N}(x)-\phi_{z,N}(y)|.
    \end{align}
    Hence, by the essential separability of the image of $\gamma$ (namely \cref{lma:esssepimage}), there exists $\{x_i\}_{i\in\N} \subset X$ so that 
    \begin{align}
        d(\gamma_s,\gamma_t)=\sup_{i,N\in\N}|\phi_{x_i,N}(\gamma_s)-\phi_{x_i,N}(\gamma_t)|
    \end{align}
    for almost every $s,t\in I$.
    Thus,
    \begin{align}
        d(\gamma_s,\gamma_t)\le \mu([s,t])
    \end{align}
    for almost every $s,t\in I$ with $s<t$, and by the previous argument we conclude that $\gamma\in L^1_{\gamma_t}$ for almost every $t\in I$.\medskip
\\    
 Now, due to the $L^1$-regularity of $\gamma$, we may assume that $\gamma\in L^1(I;(X_0,d))$, where $X_0$ is a metric component of $X$.
By \cref{lma:esssepimage} we may further assume without loss of generality that $ X_0$ is separable.
 Finally, we consider $\gamma$ as a curve in $L^1(I; \overline{X_0})$, where $\overline{X_0}$ is the completion of $X_0$.
 \\
 Observe that the curve $\gamma$ satisfies any condition of $(1)-(4)$ if and only if it satisfies the same condition viewed as a map to $\overline{X_0}$.
 It is clear that the items $(2)$ and $(4)$ do not differ on $X$, $X_0$ or $\overline{X_0}$.  
 For the condition (1) by \cref{lma:modessvar} $\essVar_{\bar{X_0}}(\gamma)=\essVar^*_{X_0}(\gamma)=\essVar^*_X(\gamma)$.
 For $(3)$, observe that for any $\varphi\in \Lip_1(\overline{X_0})$, there exists $\psi\in \Lip_1(X,d)$ with $\varphi\lvert_{X_0}=\psi\lvert_{X_0}$ due to McShane's extension theorem, and vice versa.
 \\
 Hence, applying \cite[Theorem 2.17]{AbediLiSchultz2024} we conclude that $\gamma$ satisfies all $(1)-(4)$ with respect to $\overline{X_0}$, which means it also satisfies all $(1)-(4)$ with respect to $X$.
 \end{proof}

\begin{example}
    Let $\gamma \colon I \to X\coloneqq\Q$ be a function equivalent to a Cantor function, that is, $\gamma(t)=\mathcal{C}(t)$ for almost every $t\in I$, where $\mathcal{C}\colon I\to \R$ is the standard Cantor function.
    Such a function exists, since the Cantor set $C$ has zero measure and $\mathcal{C}(I\setminus C)\subset \Q$.
    We claim that the essential variation of $\gamma$ is infinite.
    Denote by $C_0\subset C$ the set of points that are not endpoints of the removed intervals in the construction of the Cantor set.
    If $\gamma$ is continuous at some $x\in C_0$, then $\gamma(x)=\mathcal{C}(x)$ as they agree almost everywhere.
    Note that $\mathcal C$ is injective in $C_0$ and the image of $\gamma$ is countable. 
    Then the set of all $x$ in $C_0$ where $\gamma(x)=\mathcal{C}(x)$ is at most countable, which means that $\gamma$ is discontinuous at uncountably many points.
    Thus $\Var(\gamma)=\infty$, and since $\gamma$ was an arbitrary representative, $\essVar(\gamma)=\infty$.
    On the other hand, $\essVar^*(\mathcal C\lvert_{I\setminus C})\leq\Var (\mathcal C)<\infty$. 
    Therefore we see that in general $\essVar\neq \essVar^*$ nor do the corresponding BV-classes agree.
\end{example}

\section{Derivations and Continuity Equation}\label{sec:derivations}
Throughout this section, $(X,d)$ stands for a complete separable metric space.
\begin{definition}\label{def:derivation}
	Let $\lambda\in \M^+(I\times X)$.
	We say that a linear map $$V\colon \Lipb(X)\to L^\infty(I\times X, \lambda)$$ is a \emph{$\lambda$-derivation} if it satisfies the Leibniz rule, i.e., for all $f,g\in \Lipb(X)$
	\begin{equation}
		V(fg)=fV(g)+gV(f),\quad \lambda\text{-a.e. on $I\times X$}.
		\end{equation}
    \end{definition}

\subsection{Continuous derivations}\label{sec:contiDerivation}
	The space $\Lipb(X)$, equipped with the norm 
    \[
\|f\|_{\Lipb}\coloneqq \|f\|_{\infty} + \Lip(f)
    \]
 is not only a Banach space but also the dual of a Banach space, namely the Arens--Eells space; see \cite[Chapter 3]{Weaver_book}.
On norm-bounded subsets of $\Lipb(X)$, the weak$^*$ topology coincides with the topology of pointwise convergence.
Moreover, for any Borel measure $\lambda$ on $I\times X$, $L^\infty(\lambda)$ is the dual of $L^1(\lambda)$. 
Thus we may impose continuity conditions for $V$ in terms of the norm and weak$^*$ topologies as follows.

We say a $\lambda$-derivation $V$ is
\begin{itemize}
	\item \emph{continuous} if it is a bounded operator with respect to the norms $\|\cdot\|_{\Lipb}$ and $\|\cdot\|_{L^\infty(\lambda)}$, and we denote by $\|V \|_{{\rm op}}$ its operator norm;
	\item \emph{weak$^*$-continuous}\footnote{Strictly speaking, this is sequentially weak$^*$-continuous.
    However, since the predual of $\Lipb(X)$ is separable, a sequentially weak$^*$-continuous operator is also weak$^*$-continuous.
    This follows from the Banach--Dieudonn\'e--Krein--Smulian Theorem and the fact that the weak$^*$ topology on norm-bounded sets is metrizable when the predual is separable; see \cite[Theorem 3.28, Theorem 3.33]{Brezis_textbookFA}.} if for any weak$^*$-convergent sequence $f_n\overset{*}{\rightharpoonup }f$ in $\Lipb(X)$, the sequence $Vf_n$ converges to $Vf$ under the weak$^*$ topology of $L^\infty(\lambda)$.
	\end{itemize}
Specifically, as every weak$^*$-convergent sequence is norm-bounded by the Banach--Steinhaus theorem, a $\lambda$-derivation $V$ is weak$^*$-continuous if for any norm-bounded sequence $(f_n)\subset \Lipb(X)$ which converges pointwise to $f\in \Lipb(X)$, it holds that 
\begin{align}
    \int_{I\times X} g \cdot V(f_n)\d\lambda \overset{n\to\infty}{\longrightarrow} \int_{I\times X} g \cdot V(f)\d\lambda,\quad \forall g\in L^1(\lambda).
\end{align}
Given a $\lambda$-derivation $V$, we say $\theta\in L^0(\lambda)$ is a \emph{pointwise mass bound} of $V$ if 
for any $f\in \Lipb(X)$, 
\begin{equation}\label{eq:massbound}
	|Vf(t,x)|\leq \theta(t,x) \, \lip_{\rm a} (f) (x),\quad  \text{$\lambda$-a.e. on $I\times X$},
\end{equation}
where $\lip_{\rm a} (f) (x)$ denotes the asymptotic Lipschitz constant of $f$ at $x$; see \eqref{eq:def_lipfx}.

\begin{proposition}\label{prop:derivation}
    Let $V\colon \Lipb(X)\to L^\infty(\lambda)$ be a $\lambda$-derivation for $\lambda \in \M^+ (I \times X)$.
    Then
   \begin{enumerate}[label=(\arabic*), font=\normalfont]
    \item\label{item:weaklocality} (weak locality) for any $f\in \Lipb(X)$ which is constant on an open $U\subset X$, $Vf=0$ $\lambda$-a.e. on $I\times U$;
    \item\label{item:selfimprovement_massbound} (self-improvement) if there exists $\theta\in L^0(\lambda)$ such that 
    \begin{equation}\label{ineq:massbound_apriori}
    |Vf(t,x)|\leq \theta (t,x) \, \|f\|_{\Lipb},\quad \text{$\lambda$-a.e. on $I\times X$},
    \end{equation}
    then $\theta$ is a pointwise mass bound of $V$ in the sense of \eqref{eq:massbound}.
   \end{enumerate}
   If $V$ is continuous, then
   \begin{enumerate}[label=(\arabic*), font=\normalfont,start=3]
       \item\label{item:optimalmassbound} 
      there exists a unique $\theta_{V,\lambda}\in L^\infty(\lambda)$, called the \emph{optimal mass bound}, such that for all pointwise mass bounds $\theta$ of $V$, $\theta_{V,\lambda}\leq \theta$ $\lambda$-a.e.
   \end{enumerate}
    If $V$ is weak$^*$-continuous, then
    \begin{enumerate}[label=(\arabic*), font=\normalfont,start=4]
       \item\label{item:weakstartocontinuous} $V$ is continuous;
       \item\label{item:strong_locality} (strong locality) for any $f\in \Lipb(X)$, $Vf=0$ $\lambda$-a.e. on $I\times \{f=0\}$.
   \end{enumerate}
\end{proposition}
\begin{proof}
   Item \ref{item:weaklocality} is included in \cref{lemma:weaklocality}, following the definition of the notion of measure-valued derivation.
   We show now \ref{item:selfimprovement_massbound}.
Fix $f \in \Lipb(X)$ and $\varepsilon>0$.
Since $X$ is separable and $f$ is continuous, there exists a countable basis of $(U_n)_{n \in \N}$ of $X$ such that for each $n\in \N$, the oscillation of $f$ in $U_n$ is less than $\varepsilon$ i.e.
\[
M_n-m_n<\varepsilon,\quad  m_n\coloneqq \inf_{U_n}f, \ M_n\coloneqq \sup_{U_n}f.
\]
Then for each $n$, consider
\[
\bar f_n(x)\coloneqq \inf_{y\in U_n} (f(y)+L_nd(x,y))\wedge M_n\vee m_n,\quad L_n\coloneqq \Lip(f;U_n).
\]
Note that $\Lip(\bar f_n)\leq L_n$ and $\bar f_n=f$ on $U_n$.
Define $g_n \coloneqq \bar f_n - m_n$ on $X$. Then $g_n =  f - m_n$ on $U_n$, and $\|g_n\|_{\Lipb} <\varepsilon+\Lip(f;U_n)$.
Since $f-g_n = m_n$ is constant on $U_n$, weak locality gives $Vf = V g_n$ $\lambda$-a.e. on $I \times U_n$, which together with assumption yields
\begin{equation}
	|Vf(t,x)|=|V g_n(t,x)|\leq\theta(t,x)\|g_n\|_{\Lipb} < \theta(t,x)(\varepsilon+\Lip (f;U_n))
\end{equation}
for $\lambda$-a.e. $(t,x)$ on $I \times U_n$.
\\
Therefore, as $(U_n)_n$ forms a countable basis, 
\begin{equation}
	|Vf(t,x)|\leq \inf_n\{\theta(t,x)\cdot (\varepsilon+\Lip (f;U_n)):x\in U_n\}=\theta(t,x)\cdot (\varepsilon+\lip_{\rm a} f(x)).
\end{equation}
Taking $\varepsilon$ to $0$ establishes the pointwise mass bound.
\\
As for Item \ref{item:optimalmassbound}, when $V$ is norm bounded, the constant function $\|V\|_{\rm op}$ satisfies \eqref{ineq:massbound_apriori} and hence, by Item \ref{item:selfimprovement_massbound}, is a pointwise mass bound of $V$.
Thus, the set of pointwise mass bounds is nonempty, and since this set is stable under pointwise minima, by the standard essential-infimum argument, there exists a unique (up to $\lambda$-a.e. equality) optimal pointwise mass bound $\theta_{V,\lambda}$.
\\
The implication \ref{item:weakstartocontinuous} of norm boundedness from weak$^*$-continuity is standard due to the Banach--Steinhaus theorem.
The strong locality \ref{item:strong_locality} follows by the same argument in \cite[Lemma 10.34]{Weaver_book};  cf. \cite[Lemma 27]{Weaver00-JFA}.
\end{proof}

\subsubsection*{Measure-valued Derivations} 

\begin{definition}\label{def:measure_valued_derivation}
We say that a linear map
\begin{equation}\label{eq:MeasuredD}
\mathcal{D}\colon \Lipb(X) \to \M(I \times X; \R)
\end{equation}
is a \emph{measure-valued derivation} if it satisfies the Leibniz rule, i.e., for all $f,g\in \Lipb(X)$
\begin{equation}\label{eq:LeibnitzD}
		\mathcal{D}(fg) = f \mathcal{D}(g) + g \mathcal{D}(f).
	\end{equation}
\end{definition}

We say that a measure-valued derivation $\mathcal{D}$ is
\begin{itemize}
	\item \emph{continuous} if it is continuous with respect to the norm $\|\cdot\|_{\Lipb}$ on $\Lipb(X)$ and the total variation norm $\|\cdot\|_{\mathrm{TV}}$ on $\M(I \times X; \R)$;
	\item  \emph{weak$^*$-type continuous} if, whenever $f_n \wsconver f$ in $\Lipb(X)$, it holds that 
	\begin{equation}
		\int_{I \times X} g  \d\mathcal{D}(f_n) \overset{n\to\infty}{\longrightarrow} \int_{I \times X} g  \d\mathcal{D}(f), \quad \forall g \in C_b(I\times X).
	\end{equation}
	\end{itemize}

Here, the term ``weak$^*$-type'' is motivated by the fact that,
for a compact metric space $Y$, the topology of convergence of
measures on $\M(Y;\R)$ against functions in $C_b(Y) =C(Y)$ coincides
with the weak$^*$ topology on $\M(Y;\R)\cong C(Y)^*$,
via the Riesz representation theorem.
Nonetheless, for any element in $\M(I\times X;\R)$,
regarded as an element in the dual of $C_b(I\times X)$
equipped with the supremum norm, its dual norm coincides
with the total variation norm
(see e.g. \cite[Lemma 2.1.11]{Bogachev_book18}).
It then follows by the Banach--Steinhaus theorem that any weak$^*$-type continuous measure-valued derivation $\mathcal{D}$ is continuous.

Given a $\lambda$-derivation $V$, we associate with it a measure-valued derivation $\mathcal{D}$ by setting
	\begin{equation}\label{eq:measured-Derivation}
	\mathcal{D}f \coloneqq Vf\cdot \lambda,\quad  \forall f\in\Lipb(X).
\end{equation}
In this way, $\mathcal{D}$ is continuous if $V$ is continuous, and $\mathcal{D}$ is weak$^*$-type continuous if $V$ is weak$^*$-continuous.
Generally, we say a pair $(V,\lambda)$ is a \emph{representation} of $\mathcal{D}$ if \eqref{eq:measured-Derivation} holds.
However, given $\mathcal{D}$, the existence of a reference measure $\lambda$ is not clear.
To this aim, we consider an important class of derivations, in the spirit of Di Marino \cite{Dimarino2014}.
\begin{definition}\label{def:massD}
    We say that a measure-valued derivation $\mathcal{D}$ has \emph{finite mass} if there exists a finite measure $\lambda\in \M^+(I\times X)$ such that
    \begin{equation}\label{eq:finiteD}
        |\mathcal{D}(f)|\leq  \lip_{\rm a} (f) \, \lambda  ,\quad \forall f\in \Lipb(X).
    \end{equation}
    We call $\lambda_*$ the \emph{mass measure} of $\mathcal{D}$ if it is the minimal measure satisfying \eqref{eq:finiteD}, and we call a pair $(V,\lambda_*)$ a \emph{minimal representation} of $\cD$ if $V$ is a $\lambda_*$-derivation satisfying \eqref{eq:measured-Derivation}.
\end{definition}

It is obvious that any finite-mass measure-valued derivation is continuous.
Moreover, by duality (see Lemma~\ref{lemma:finitemassD}\,\ref{item:massM2} below) and Proposition~\ref{prop:derivation}\,\ref{item:selfimprovement_massbound}, a measure-valued derivation $\cD$ has finite mass if and only if there exist $C>0$ and $\lambda\in \M^+(I\times X)$ such that 
\[
|\mathcal{D}(f)|\leq C \|f\|_{\Lipb} \lambda,\quad f\in \Lipb(X).
\]
\begin{lemma}\label{lemma:finitemassD}
    Let $\mathcal{D}$ be a finite-mass measure-valued derivation. 
    Then 
    \begin{enumerate}[label=(\arabic*), font=\normalfont]
        \item\label{item:massM1} the mass measure $\lambda_*$ of $\cD$ uniquely exists;
         \item\label{item:massM2} there exists a continuous $\lambda_*$-derivation $V$ satisfying \eqref{eq:measured-Derivation}, which is weak$^*$-continuous if $\cD$ is weak$^*$-type continuous;
        \item\label{item:massM1.5} $\lambda_*$ is the supremum measure (see e.g. \cite{Ambrosio1990MetricSV} for this notion) of the family $\{|\cD f|, f\in \Lipb(X),\Lip(f)\leq 1\}$ i.e. the minimal measure among all $\lambda$ satisfying
        \begin{equation}\label{eq:min_family}
        |\cD (f)|\leq \lambda,\quad \forall f\in \Lipb(X), \, \Lip(f)\leq 1;
        \end{equation}
        \item\label{item:massM3} $(V,\lambda)$ is a minimal representation of $\cD$ (in particular $\lambda$ is the mass measure of $\cD$) if and only if $(V,\lambda)$ is a representation of $\cD$ and the optimal mass bound $\theta_{V,\lambda}=1$ $\lambda$-a.e. on $I\times X$.
    \end{enumerate}
\end{lemma}
\begin{remark}\label{remark:minimalrepresentation}
We can define an equivalence relation on pairs $(V,\lambda)$ by
$(V,\lambda)\sim(\tilde V,\tilde\lambda)$ whenever
\begin{equation}\label{eq:Lambda} 
V f \cdot \lambda = \tilde{V}f \cdot \tilde{\lambda},\quad \forall f \in \Lipb(X).
\end{equation}
By \cref{prop:derivation}\,\ref{item:optimalmassbound} and \cref{lemma:finitemassD}\,\ref{item:massM2}, finite-mass measure-valued derivations $\cD$ are in one-to-one correspondence with equivalence classes $[(V,\lambda)]$ of continuous derivations.
Among all representatives of a given equivalence class, a minimal representation by definition is one whose reference measure is minimal subject to the pointwise mass bound \eqref{eq:finiteD}. 
This minimality motivates the terminology.
\end{remark}
\begin{proof}[Proof of \cref{lemma:finitemassD}]
    \underline{Proof of \ref{item:massM1}}. This is based on the direct method in the calculus of variation.
    Denote by $\Lambda_{\cD}$ the set of finite measures on $I\times X$ satisfying \eqref{eq:finiteD}.
    By assumption, $\Lambda_\cD$ is not empty.
    Let $(\lambda_n)_{n\in \N}$ be a minimizing sequence in $\Lambda_{\cD}$ achieving
    \[
 \lim_{n\to \infty} \lambda_n(I\times X) =\inf_{\lambda\in \Lambda_\cD} \lambda (I\times X).
    \]
    Notice that if $\lambda_1,\lambda_2\in \Lambda_\cD$, then the minimum measure $\lambda_1\wedge\lambda_2\in \Lambda_\cD$.
    Here the minimum measure is defined by
    \begin{equation}
        \lambda_1\wedge \lambda_2\coloneqq \min \{\rho_1,\rho_2\} (\lambda_1+\lambda_2)
    \end{equation}
    where $\rho_i (\lambda_1+\lambda_2)=\lambda_i$ for $i=1,2$.
    \\
    Therefore, without loss of generality, we may assume $(\lambda_n)$ is monotone decreasing by replacing it with $\lambda_1 \wedge \cdots \wedge \lambda_n$.
	In particular, the family $(\lambda_n)$ is tight.  
	Hence, by Prokhorov’s theorem, up to passing to a subsequence, there exists $\lambda_* \in \M^+(I\times X)$ such that $\lambda_n \rightharpoonup \lambda_*$ weakly.
	\\
	For any $f \in \Lipb(X)$ and compact $K \subset I\times X$, the weak convergence $\lambda_n \rightharpoonup \lambda_*$ together with the upper semi-continuity of $\lip_{\rm a} f$ on $X$ yields
	\begin{align}
		\int \mathds{1}_K \d|\mathcal{D}(f)|
		&\leq \liminf_{n\to\infty} \int \mathds{1}_K \lip_{\rm a} f   \d\lambda_n\\
		&\leq \limsup_{n\to\infty} \int\mathds{1}_K \lip_{\rm a} f   \d\lambda_n
		\leq \int\mathds{1}_K \lip_{\rm a} f   \d\lambda_*.
	\end{align}
	That is, $|\mathcal{D}(f)|\leq\lip_{\rm a} f \cdot \lambda_*$ on compact sets.  
	By inner regularity of finite measures on Polish spaces, the inequality holds for all Borel sets.  
	In particular, $\lambda_* \in  \Lambda_{\mathcal{D}}$.  
	The weak convergence also ensures that $\lambda_n(I\times X) \to \lambda_*(I\times X)$.
    Finally, for any $\lambda \in  \Lambda_{\mathcal{D}}$, we have $\lambda_* \wedge \lambda \in  \Lambda_{\mathcal{D}}$.  
	Since $\lambda_*$ has the minimal total mass, $\lambda_*\wedge \lambda= \lambda_*$.
    This shows the uniqueness of the minimizer.\smallskip
    \\
    \underline{Proof of \ref{item:massM2}}. For any $f\in \Lipb(X)$, the mass bound \eqref{eq:finiteD} ensures that the map
    \[
    L^1(\lambda_*)\ni g\mapsto \int_{I\times X} g\d\cD (f)
    \]
    is a bounded linear functional. 
   By the duality $(L^1(\lambda_*))^* = L^\infty(\lambda_*)$, for each $f$, there exists a unique element in $ L^\infty(\lambda_*)$ denoted by $v_f$ such that $\int g \d \cD (f) = \int g v_f \d \lambda_*$ for all $g \in L^1(\lambda_*)$. 
   We then define a map $V: \Lipb(X) \to L^\infty (\lambda_*)$ by setting $Vf \coloneqq v_f$. 
   The map $V$ is linear and satisfies the Leibniz rule, which follows directly from that of $\cD$. 
   Thus, $V$ is a $\lambda_*$-derivation. It satisfies \eqref{eq:measured-Derivation} by the integral equality above, and admits a pointwise mass bound by $1$. 
    In particular, $V$ is norm-bounded with $\|V\|_{{\rm op}} \le 1$.
	\smallskip
    \\
    Now assume that $\mathcal{D}$ is weak$^*$-type continuous.
    We want to show that $V$ is weak$^*$-continuous.
    Let $f_n \wsconver f$ in $\Lipb(X)$ and let $g \in L^1(\lambda_*)$.  
	By the density of $C_b(I \times X)$ in $L^1(\lambda_*)$, take a sequence $(g_m)_{m}$ in $C_b(I \times X)$ converging to $g$ in the $L^1(\lambda_*)$ norm.  
	For each fixed $m$, the weak$^*$-type continuity of $\mathcal{D}$ gives
	\[
	\int g_m\d\mathcal{D}(f_n - f) \to 0
	\quad \text{as } n \to \infty.
	\]
	By the Banach--Steinhaus theorem, as a weak$^*$-convergent sequence, $(f_n)$ is bounded under the $\|\cdot\|_{\Lipb}$ norm.
	The continuity of $V$ yields that $(Vf_n)_n$ is bounded in $L^\infty(\lambda_*)$.
    Thus
	\begin{align}
		\left| \int (Vf - Vf_n) g\d\lambda_* \right|
		&\le \left| \int (Vf - Vf_n) g_m\d\lambda_* \right|
		+ \left| \int (Vf - Vf_n)(g - g_m)\d\lambda_* \right| \\
		&\le \left| \int g_m\d\mathcal{D}(f - f_n) \right|
		+ \|g - g_m\|_{L^1(\lambda_*)}\sup_n\|V(f_n) - V(f)\|_{L^\infty(\lambda_*)}.
	\end{align}
	Choosing first sufficiently large $m$ and then taking $n$ to $\infty$ we conclude $Vf_n \wsconver Vf$ in $L^\infty(\lambda_*)$.\smallskip
    \\
    \underline{Proof of \ref{item:massM1.5}.}
    Denote by $\lambda_\triangle$ the supremum measure of the family $\{|\cD f|, f\in \Lipb(X),\Lip(f)\leq 1\}$.
    By definition $\lambda_\triangle\leq \lambda_*$.
    To show the reverse direction,
     first note that by the definition of $\lambda_\triangle$, we have
    \[
    |\mathcal{D}(f)|\leq \Lip(f)\lambda_\triangle \quad \forall f\in \Lipb(X).
    \]
    Now let $V$ be the $\lambda_\triangle$-derivation, given by the duality argument similar to that in \ref{item:massM2}, which satisfies $D(f) = Vf \lambda_\triangle$ and
    \[
    |Vf|\leq \Lip(f)\leq \|f\|_{\Lipb},\quad \forall f\in \Lipb(X).
    \]
     Then \cref{prop:derivation}\,\ref{item:selfimprovement_massbound} yields that $|Vf|\leq \lip_{\rm a}f$.
     This means that $\lambda_\triangle$ satisfies \eqref{eq:finiteD}, which by definition implies that $\lambda_*\leq \lambda_\triangle$.\smallskip
     \\
    \underline{Proof of \ref{item:massM3}}.
   Let $(V,\lambda)$ be a representation of $\cD$ with the optimal mass bound $\theta_{V,\lambda}=1$ $\lambda$-a.e. To show that $(V,\lambda)$ is a minimal representation of $\cD$, we prove by contradiction. 
   Assume there exists $\lambda_*$ satisfying \eqref{eq:finiteD} such that $\lambda_*(\tilde A)<\lambda(\tilde A)$ for some Borel set $\tilde A$.
    We may assume $\lambda_*\leq \lambda$ as otherwise we consider $\lambda_*\wedge \lambda$.
    Denote by $\rho_*$ the Borel function that $\lambda_* = \rho_*\lambda$.
    Then $0\leq \rho_* \leq 1$ $\lambda$-a.e. and there exist $\delta>0$ and a Borel set $ A$ of positive $\lambda$-measure such that $\rho_*\leq 1-\delta$ on $A$.
    \\
    Take $\rho\coloneqq \mathds{1}_{I\times X\setminus A}+(1-\delta)\mathds{1}_{A}$.
    Then $\rho_* \leq \rho $ $\lambda$-a.e. and for any $f\in \Lipb(X)$
    \begin{equation}
        |Vf\cdot \lambda|=|\cD(f)|\leq \lip_{\rm a} (f)\, \lambda_* = \lip_{\rm a} (f)\, \rho_* \lambda \leq \lip_{\rm a} (f)\, \rho\lambda,
    \end{equation}
    which implies that $|V f|\leq \rho\, \lip_{\rm a} (f) $ $\lambda$-a.e.
    This contradicts the assumption that the optimal mass $\theta_{V,\lambda}=1$ $\lambda$-a.e. 
    The reverse direction can be argued similarly.
\end{proof}

\subsubsection*{Weak locality and extension of derivations} 
\begin{lemma}[Weak locality]\label{lemma:weaklocality}
Let $\cD$ be a measure-valued derivation.
Then for any $f\in \Lipb(X)$ which is constant on an open $U\subset X$, $\cD(f)=0$ on $I\times U$.
\end{lemma}
\begin{proof}
    By the Leibniz rule, $\cD(\mathds{1})=0$.
    Thus by linearity, it suffices to consider $f=0$ on $U$.
    For any compact set $C\subset I\times U$, denote $K\coloneqq {\rm Pr}^X(C)\subset U$ which is compact on $X$.
    There exists $\eta\in \Lipb(X)$ such that $\eta\equiv 1$ on $K$ and $\spt(\eta)\subset U$.
    By the Leibniz rule, with $f\cdot \eta\equiv 0$ on $X$ and $f=0$ on $U$,
    \[
0=\cD(f\eta)=\eta \cD(f)+f\cD(\eta)=\eta \cD(f), \text{  on $I\times U$}.
    \]
    Thus $\cD(f)=0$ on $I\times K\supset C$.
    By the inner regularity of $\cD(f)$, $\cD(f)=0$ on $I\times U$.
\end{proof}

Thanks to the weak locality, for every $f\in\Lip(X)$ and every bounded $A\subset I\times X$, we can set
\begin{equation}\label{eq:extension}
\cD(f)(A)\coloneqq \cD(f\wedge L\vee -L )(A)
\end{equation}
where $L$ is chosen that $L>\sup_{{\rm Pr}^X(A)}|f|$.
\cref{lemma:weaklocality} ensures that the quantity above does not depend on the choices of $L$.
Notice that $\cD(f)$ need not define a signed measure on $I\times X$, as its value on an arbitrary unbounded set $A$ may not be well defined.

\begin{proposition}[Extension, and chain rule]\label{prop:extension}
   Let $\cD$ be a continuous measure-valued derivation. 
   Then $\cD$ can be extended to a linear map on whole $\Lip(X)$,
  \begin{equation}
  \cD\colon\Lip(X)\to \M(I\times X;\R).
  \end{equation}
  Moreover:
  \begin{enumerate}[label=(\arabic*), font=\normalfont]
      \item\label{item:extension1} (chain rule) for every $f_1,\dots,f_n\in\Lip(X)$ and every $\Phi\in C^1(\R^n)\cap \Lip(\R^n)$, the following holds 
	\begin{equation}\label{eq:chainrule_D}
\cD(\Phi(f_1,\dots,f_n))=\sum_{i=1}^n\partial_i\Phi(f_1,\dots,f_n)\cD(f_i);
	\end{equation}
    \item\label{item:extension2} there exists $C>0$ such that $\|\cD(f)\|_{\rm TV}\leq C\,\Lip(f)$;
    \item\label{item:extension3} if $\mathcal{D}$ has finite mass measure $\lambda$, then the pointwise mass bound \eqref{eq:finiteD} holds for all $f\in \Lip(X)$.
  \end{enumerate}
\end{proposition}
\begin{proof}
  We first show the chain rule for $f_1,\cdots, f_n\in\Lipb(X)$.
    For simplicity, we prove only for $n=1$ as the multivariable case is completely analogous.
\\
    If $\Phi$ is affine, the conclusion follows from the linearity of $\cD$ and $\mathcal D(\mathds{1})=0$ (by the Leibniz rule). The general polynomial case then follows by induction using the Leibniz rule.
    \\
    Consider now $\Phi\in C^1(\R)$. 
    Let $f\in \Lipb(X)$ and let $L>\|f\|_\infty$.
    Take a sequence of polynomials $(p_n)_n$ on $\R$ converging to $\Phi$ in $C^1((-L,L))$.
    Then by continuity
    \[
    \cD(\Phi(f))=\lim_{n\to\infty} \cD(p_n(f))=\lim_{n\to\infty} p_n'(f)\cD(f)=\Phi'(f)\cD(f).
    \]
    By \cref{lemma:weaklocality} and the extension \eqref{eq:extension}, for general $f_1,\cdots, f_n\in\Lip(X)$ the identity \eqref{eq:chainrule_D} still holds when evaluating bounded sets in $I\times X$. 
    \\
     Now for $f\in \Lip(X)$, by the chain rule 
      \begin{align}\label{eq:chainrule_sincos}
        \cD(f)= (\cos^2(f)+\sin^2(f))\cD(f)= \cos(f) \cD(\sin(f))-\sin(f)\cD(\cos(f))
    \end{align} 
    holds over bounded sets.
    Since $\cD(\sin(f))$ and $\cD(\cos(f))$ are finite signed measures, $\cD(f)\in \M(I\times X;\R)$.
    Therefore, the identity \eqref{eq:chainrule_D} can be understood between finite measures.    
    The pointwise mass bound follows from the weak locality of $\cD$.
    \\
    It remains to show \ref{item:extension2}.
    If $\Lip(f)=0$, then $f$ is constant and the estimate follows from the weak locality.
    Assume that $L\coloneqq \Lip(f)>0$.
    Denote 
    \[
    C\coloneqq 2\sup \{\|\cD(f)\|_{\rm TV}: \|f\|_{\infty}\leq 1,\Lip(f)\leq 1\}
    \]
    which is finite due to the continuity of $\cD$.
    Applying \eqref{eq:chainrule_sincos} to $g\coloneqq f/L$ gives 
    \begin{align}
        \|\cD(g)\|_{\rm TV}&\leq \|\cD(\sin(g))\|_{\rm TV}+\|\cD(\cos g)\|_{\rm TV}\leq C.
    \end{align}
    The estimate \ref{item:extension2} for $f$ follows by scaling.
\end{proof}

\subsection{Continuity Equation (CE) via derivations}\label{subsec:CE_derivations} 
\begin{definition}[CE via derivations]\label{def:CE_mu_lambda_V}  
	Let $\mu \in \M^+ (I \times X)$, let $V\colon \Lipb(X)\to L^\infty(\lambda)$ be a $\lambda$-derivation for $\lambda \in \M^+ (I \times X)$ in the sense of \cref{def:derivation}, and let $\mathcal{D}$ be a measure-valued derivation in the sense of \cref{def:measure_valued_derivation}.
	We say that 
    \begin{itemize}
        \item the triple $(\mu,\lambda,V)$ solves the continuity equation, formally written as
\begin{equation}\label{eq:CE_derivation}\tag{$\mathrm{CE}_{V}$} 
        \partial_t \mu + \Div (V \lambda ) =0  \qquad \textrm{in }  I\times X,
    \end{equation}
    if for every $f\in \Lipb(X)$ and $\xi  \in C^1_c(I)$,
    \begin{equation}\label{eq:CE_derivation_weak_form}
    \int_{I \times X} \frac{\d}{\d t} \xi(t) f(x) \d \mu(t,x) + \int_{I \times X} \xi(t) Vf (t,x) \d \lambda (t,x) = 0.
\end{equation}
   \item the pair $(\mu,\mathcal{D})$ solves the continuity equation, formally written as
   \begin{equation}\label{eq:CE_D}\tag{$\mathrm{CE}_{\mathcal{D}}$}
       \partial_t \mu  + \Div (\mathcal{D}) = 0  \qquad \textrm{in }  I\times X,
   \end{equation}
   if for every $f\in \Lipb(X)$ and $\xi  \in C^1_c(I)$,
\begin{equation}\label{eq:CE_Mderivation}
        \int_{I \times X} \frac{\d}{\d t}\xi(t) f(x) \d \mu(t,x) + \int_{I \times X} \xi(t) \d\mathcal{D}f (t,x)= 0.
    \end{equation}
    \end{itemize}
     In particular, if $(\mu,\mathcal{D})$ solves the \eqref{eq:CE_D}, then any $(V,\lambda)$ which is a representation of $\cD$, the triple $(\mu, \lambda,V)$ solves \eqref{eq:CE_derivation}. Throughout this section, we exclude the trivial case $\mu=0$ and assume $\|\mu\|_{\mathrm{TV}}>0$.
\end{definition}

\begin{theorem}\label{thm:derivation_CE}
	Let $(\mu,\lambda,V)$ be a solution to \eqref{eq:CE_derivation} in the sense of Definition \ref{def:CE_mu_lambda_V}.
	Then
	\begin{enumerate}[label=(\arabic*), font=\normalfont]
		\item\label{item:CE_derivation2} $\pr^I_\# \mu=c \mathcal{L}^1\llcorner I$, where $c=T^{-1}\|\mu\|_{\mathrm{TV}}$, and there exists a Borel family $(\mu_t)_{t \in I} \subset \P (X)$  (uniquely determined for a.e. $t\in I$) such that $\mu(\mathrm{d} t,\mathrm{d} x)\coloneqq c \mu_t(\mathrm{d} x) \mathrm{d} t$;
        \item\label{item:CE_derivation3} for any $f\in \Lipb(X)$, $t\mapsto \int f\d\mu_t$ is in $BV (I;\R)$, whose distributional derivative satisfies
		\begin{equation}\label{eq:CE_ST}
			\D\left( \int_{X} f(x) \d \mu_{(\cdot)}(x)\right)=  c^{-1} \, \pr^I_{\#}(Vf\cdot \lambda);
		\end{equation}
		\end{enumerate}
        if additionally $V$ is a continuous $\lambda$-derivation, then
        \begin{enumerate}[label=(\arabic*), font=\normalfont,start=3]
    	\item\label{item:CE_derivation1} $V$ is weak$^*$-continuous;
		\item\label{item:CE_derivation4}  $(\mu_t)\in BV(I;(\P(X),W_1))$ with
        \begin{equation}\label{ineq:Varineq_derivation}
            |\D \mu| \leq  c^{-1} \, \pr^I_\# (\theta_{V,\lambda}\cdot \lambda),
        \end{equation}
    where $\theta_{V,\lambda}$ is the minimal pointwise mass bound given by Proposition \ref{prop:derivation}\,\ref{item:optimalmassbound}.
	\end{enumerate}
    \end{theorem}
    \begin{proof}
    \underline{Proof of \ref{item:CE_derivation2}.}
    By the Leibniz rule of $V$, we have $V\mathds{1}_X=0$ $\lambda$-a.e.
    Taking $f=\mathds{1}_X$ in \eqref{eq:CE_derivation_weak_form} yields for any $\xi\in C^1_c(I)$ that
    \begin{equation}
    	\int_{I\times X} \xi'(t)\d\mu(t,x)=\int_I \xi'(t)\d (\pr^I_\#\mu)(t)=0.
    \end{equation}
    As all divergence-free finite measures on $I$ are proportional to the Lebesgue measure (see, for a proof, e.g., \cite[Theorem 3.1.4]{Hormander_I}), we have  $\pr^I_\#\mu=c\mathcal{L}^1\llcorner (0,T)$ for some $c\geq 0$. In particular, we obtain $\|\mu \|_{\mathrm{TV}} = \mu (I \times X) = c T $. 
    \\
    As $c>0$, by the disintegration theorem, there exists a Borel family $(\mu_t)_{t \in I} \subset \P (X)$ (uniquely determined for a.e. $t \in I$) such that $\mu(\mathrm{d} t,\mathrm{d} x)\coloneqq c \mu_t(\mathrm{d} x) \mathrm{d} t$.
    \smallskip
    \\
    \underline{Proof of \ref{item:CE_derivation3}.}
    With the curve representation of $\mu$ shown above, \eqref{eq:CE_derivation_weak_form} can be written as
    
    \begin{align}
    	c\int_I \xi'(t) \left(\int_X  f(x)\d\mu_t(x) \right) \d t+ \int_I \xi(t)\d \pr^I_\# (Vf\cdot \lambda)(t) = 0.
    \end{align}
    As $|Vf\cdot \lambda|$ is a finite measure, this gives by Definition~\ref{def_BV_curves_Euclidean}, that $t\mapsto \int f\d\mu_t$ is BV with the distributional derivative given by \eqref{eq:CE_ST}.\smallskip
    \\
    \underline{Proof of \ref{item:CE_derivation1}.}
	  By the Banach--Steinhaus theorem, any weak$^*$-convergent sequence in $\Lipb(X)$ is bounded.  
	Hence it suffices to show that for any sequence $(f_n) \subset \Lipb(X)$ converging pointwise to $f\in \Lipb(X)$, with $\|f_n\|_\infty, \Lip(f_n) \leq L$ for some $L > 0$ for all $n$, we have
	\begin{equation}\label{eq:30/10-2}
		\int g\, V f_n \, d\lambda \to \int g\, V f \, d\lambda, 
		\qquad \forall g \in L^1(\lambda).
	\end{equation}
	For any $\varphi \in \Lipb(X)$ and $\xi \in C^1_c(I)$, set $g = \xi \varphi$.  
	Using the Leibniz rule and \eqref{eq:CE_derivation_weak_form}, we have 
	\begin{align}
		\int \xi\varphi Vf_n\d\lambda&=	\int\xi V(\varphi f_n)\d\lambda-\int \xi f_nV\varphi\d\lambda\\
		&=\int -\xi'(t) (\varphi f_n)(x)\d\mu(t,x)-\int \xi f_n V\varphi\d\lambda\\
		&\overset{n\to \infty}{\rightarrow}\int -\xi'(t) (\varphi f)(x)\d\mu(t,x)-\int \xi f V\varphi\d\lambda=\int \xi \varphi Vf\d\lambda
	\end{align}
	where the convergence follows from the pointwise convergence $f_n \to f$ and the dominated convergence theorem.
\\
    By linearity, the claim holds for $g\in \mathcal{F}$, where $\mathcal{F}$ is the set of all finite linear combinations of functions $\xi \varphi$ with $\xi\in C^1_c(I)$ and $\varphi \in \Lipb(X)$.
    Observe that the set $\mathcal{F}$ is dense in $L^1(\lambda)$.
	Take an arbitrary $g \in L^1(\lambda)$ and a sequence $(g_m) \subset \mathcal{F}$ with $g_m \to g$ in $L^1(\lambda)$.  
	Then for each $m \in \N$,
	\begin{align}
		\left| \int g (V f_n - V f) \d\lambda \right|
		&\leq \left| \int g_m (V f_n - V f) \d\lambda \right|
		+ \int |g - g_m|\, |V f_n - V f| \d\lambda \\
		&\leq \left| \int g_m (V f_n - V f) \d\lambda \right|
		+ \|g - g_m\|_{L^1(\lambda)}\sup_n\|V(f_n) - V(f)\|_{L^\infty(\lambda)}.
	\end{align}
    Letting first $n \to \infty$ and then $m\to \infty$,  we conclude that $Vf_n \wsconver Vf$ in $L^\infty(\lambda)$.\smallskip
    \\
    \underline{Proof of \ref{item:CE_derivation4}.}
    Denote by $\mathrm{Lip}_{\mathrm{b},1}(X)$ the set of all bounded $1$-Lipschitz functions on $X$.
    By Item~\ref{item:CE_derivation3}, for any $f\in \mathrm{Lip}_{\mathrm{b},1}(X)$, $t\mapsto \int f\d\mu_t$ is a BV-curve on $\R$. 
    Thus, there exists a Borel set $I_f\subset I$ of full measure such that for all $s,t\in I_f$ with $s<t$
    \begin{align}
    	\left|\int f(x)\d\mu_t(x)-\int f(x)\d\mu_s(x)\right|
    	& \leq c^{-1} \left| \pr^I_\# ( Vf \cdot \lambda) \right|([s,t]) \leq c^{-1} (\pr^I_\# (\theta_{V,\lambda}\lambda))([s,t]),\label{eq:30/10-1}
    \end{align}
    where in the last inequality, we used the pointwise mass bound, which follows from \cref{prop:derivation} \ref{item:optimalmassbound}.
    By weak$^*$-metrizability and separability on norm-bounded subsets, we can choose a countable family $(f_n)\subset\mathrm{Lip}_{\mathrm{b},1}(X)$ such that every $f\in\mathrm{Lip}_{\mathrm{b},1}(X)$ is the pointwise limit of a uniformly bounded sequence drawn from this family. For each $n\in\N$, there is $I_n\subset I$ of full measure such that \eqref{eq:30/10-1} holds for $f_n$ and all $s, t\in I_n$ with $s<t$.
    In particular, denoting $\tilde I\coloneqq \cap_n I_n$ we have for all $s,t\in\tilde I$ with $s<t$, \eqref{eq:30/10-1} holds for all $f_n$.
    By dominated convergence, the estimate
    $$
    \left|\int f\,\d\mu_t-\int f\,\d\mu_s\right|
    \leq c^{-1}\bigl(\pr^I_\#(\theta_{V,\lambda}\lambda)\bigr)([s,t])
    $$
    holds for every $f\in\mathrm{Lip}_{\mathrm{b},1}(X)$ and every $s,t\in\tilde I$ with $s<t$.
    \\
    Now, fix any $s,t\in\tilde I$.
    Taking supremum over $f \in \mathrm{Lip}_{\mathrm{b},1}(X)$ in the above and using the Kantorovich--Rubinstein duality \eqref{eq:Kduality}, we conclude for all $s,t\in\tilde I$ with $s<t$ that 
    \begin{equation}
    	W_1(\mu_s,\mu_t)\leq c^{-1} (\pr^I_\# (\theta_{V,\lambda}\lambda))([s,t]).
    \end{equation}
    This immediately implies that the curve $t \mapsto \mu_t$ from $\tilde I \to \P(X)$ is of bounded pointwise variation with respect to $W_1$, and that it stays in a metric component of $(\P (X),W_1 )$, which is complete by Proposition~\ref{prop:wassersteinspacecomplete}.
    By Lemma~\ref{lemma:cadlag_extension_dense_BV}, there exists a $W_1$-c\`adl\`ag curve $t \mapsto \mu_t$ from $I \to \P(X)$ which is of bounded pointwise variation. In particular, $(\mu_t)_{t \in I}$ is $W_1$-Borel. One can show that this extension agrees with
    the original family almost everywhere. All in all, we have $(\mu_t)\in BV(I;(\P(X),W_1))$ in the sense of Definition~\ref{def:BV_extended}, whose variation satisfies \eqref{ineq:Varineq_derivation} (recall Remark~\ref{remark:variationmeasure}).
    \end{proof}
    
    Repeating the proof of \cref{thm:derivation_CE}, we obtain the following analogous result for continuity equations via measure-valued derivations.
    \begin{corollary}\label{cor:derivation_CE}
    Let $(\mu,\cD)$ be a solution to \eqref{eq:CE_D} in the sense of Definition \ref{def:CE_mu_lambda_V}.
	Then
    \begin{enumerate}[label=(\arabic*), font=\normalfont]
        \item $\pr^I_\# \mu=c \mathcal{L}^1\llcorner I$, where $c=T^{-1}\|\mu\|_{\mathrm{TV}}$, and there exists a Borel family $(\mu_t)_{t \in I} \subset \P (X)$  (uniquely determined for a.e. $t\in I$) such that $\mu(\mathrm{d} t,\mathrm{d} x)\coloneqq c\mu_t(\mathrm{d} x) \mathrm{d} t$;
        \item for any $f\in \Lipb(X)$, $t\mapsto \int f\d\mu_t$ is in $BV (I;\R)$, whose distributional derivative satisfies
		\begin{equation}\label{eq:Derivative_alongf}
			\D\left( \int_{X} f(x) \d \mu_{(\cdot)}(x)\right)=c^{-1}\, \pr^I_{\#}(\cD(f));
		\end{equation}
    \end{enumerate}
    if additionally $\cD$ has finite mass, then
	\begin{enumerate}[label=(\arabic*), font=\normalfont,start=3]
    \item $\cD$ is weak$^*$-type continuous;
    \item\label{item:Varineq_Mderivation} $(\mu_t)\in BV(I;(\P(X),W_1))$ with
    \begin{equation}\label{ineq:Varineq_D}
        |\D \mu| \leq  c^{-1}\, \pr^I_\# |\cD|,
    \end{equation}
    where $|\cD| \in 
    \M^{+}(I \times X)$ is the mass measure of $\cD$ given by \cref{lemma:finitemassD}\,\ref{item:massM1}.
    \end{enumerate}
\end{corollary}

\subsection{Minimal derivations and Minimal solutions to CE}\label{sec:minimalderiavtion}
Throughout this section, all derivations are assumed to be defined with respect to the same underlying Polish metric space $X$ and interval $I$.

\begin{definition}[Subderivation]\label{def:subderivation}
Given two measure-valued derivations $\mathcal{D}_1$ and $\mathcal{D}_2$ with finite mass, we say that $\mathcal D_2$ is a \emph{subderivation} of $\mathcal D_1$ and write $\mathcal D_2\prec \mathcal D_1$ if
    \begin{align}\label{ineq:subderivation1}
        |\mathcal D_2|(I\times X)+|\mathcal D_1 - \mathcal D_2|(I\times X)\le|\mathcal D_1|(I\times X),
    \end{align}
    where $|\mathcal D|$ denotes the mass measure of a finite-mass measure-valued derivation $\mathcal D$.
\end{definition}

\begin{remark}\label{remark:subderivation} \,
\begin{enumerate}[label=(\arabic*), font=\normalfont]
    \item\label{item:rmksubderivation2}
    The definition remains equivalent if condition
    \eqref{ineq:subderivation1} is replaced by the following seemingly more restrictive one:
    \begin{equation}\label{ineq:subderivation2}
        |\mathcal D_2|+|\mathcal D_1-\mathcal D_2|
        =|\mathcal D_1|.
    \end{equation}
    This is because for any finite-mass derivations $\mathcal{D}_1$ and $\mathcal{D}_2$, we always have
    \begin{equation}\label{ineq:subderivation23}
         |\mathcal D_2|+|\mathcal D_1-\mathcal D_2| \geq |\mathcal D_1|.
    \end{equation}
    Indeed, by the pointwise mass bound for $\cD_2 $ and $\cD_1-\cD_2$, we have
    \[
        |\mathcal D_1(f)|
        \leq |\mathcal D_2(f)|
        +|(\mathcal D_1-\mathcal D_2)(f)|
        \leq \lip_{\rm a}f\,
        \bigl(|\mathcal D_2|+|\mathcal D_1-\mathcal D_2|\bigr),
        \quad \forall f\in\Lipb(X).
    \]
    Thus, by the definition of the mass measure, we obtain \eqref{ineq:subderivation23}. 
    So the inequality \eqref{ineq:subderivation1} forces
    the identity.

    \item\label{item:rmksubderivation1}
    The binary relation $\prec$ indeed defines a partial order.
    Reflexivity is obvious. To see antisymmetry, adding the defining inequalities for
    $\mathcal D_2\prec\mathcal D_1$ and
    $\mathcal D_1\prec\mathcal D_2$ gives
    $$
        2|\mathcal D_1-\mathcal D_2|(I\times X)\leq 0,
    $$
    hence $\mathcal D_1=\mathcal D_2$.
    To see transitivity, suppose
    $\mathcal D_3\prec\mathcal D_2$ and
    $\mathcal D_2\prec\mathcal D_1$.
    Similar to the first part of the remark, we get
    \begin{align}
        &|\mathcal D_3|(I\times X)
        +|\mathcal D_1-\mathcal D_3|(I\times X)\\
        &\quad\leq
        |\mathcal D_3|(I\times X)
        +|\mathcal D_2-\mathcal D_3|(I\times X)
        +|\mathcal D_1-\mathcal D_2|(I\times X)\\
        &\quad\leq
        |\mathcal D_2|(I\times X)
        +|\mathcal D_1-\mathcal D_2|(I\times X)
        \leq |\mathcal D_1|(I\times X),
    \end{align}
    so $\mathcal D_3\prec\mathcal D_1$ holds.
\end{enumerate}
\end{remark}
    
\begin{definition}[Minimal derivation]\label{def:minimalderivation}
We say a measure-valued derivation $\mathcal{D}$ is a \emph{spatial cycle} (or \emph{divergence free}), if
    \begin{align}
        \int_{I\times X} \xi(t)\d\mathcal{D}(f)=0,\quad\forall f\in\Lipb(X), \, \xi \in \mathcal{B}_{\mathrm{b}}(I).
    \end{align}    
We call a finite-mass derivation $\mathcal{D}$ a \emph{minimal derivation} (or \emph{acyclic}) if it does not have any nontrivial spatial cycle as a subderivation.
Moreover, we say a $\lambda$-derivation $V$ is minimal if its associated measure-valued derivation is minimal.
\end{definition}

One can check that a finite-mass derivation $\mathcal{D}$ is a minimal derivation if and only if it is a minimal element of the following set with respect to the partial order $\prec$ in \cref{def:subderivation}
	\[
	\left\{ \tilde{\mathcal{D}} \text{ finite-mass derivation}: \mathcal{D}-\tilde{\mathcal{D}} \text{ is divergence free} \right\}.
	\]
The terminology of minimal derivations is therefore completely analogous to that of minimal measures introduced in \cite{AlmiRossiSavare2025}; see also \cref{def:minimalmeasure}. 
Moreover, in the same spirit, one can define a minimal solution $(\mu,\mathcal D)$ to \eqref{eq:CE_D} by asking $\mathcal D$ to have a minimal singular part with respect to $\mu$.

Let $\mu\in \M^+(I\times X)$ be a reference measure. 
For any $\nu\in \M(I\times X;\R)$, we consider its Lebesgue decomposition $\nu=\nu^a+\nu^\perp$ with respect to $\mu$, where $\nu^a\ll \mu$ and $\nu^\perp \perp \mu$.
Analogously, we can define the Lebesgue decomposition of a measure-valued derivation $\mathcal{D}$ with respect to $\mu$ by setting
\begin{equation}
\mathcal{D}=\mathcal{D}^a+\mathcal{D}^\perp,\quad 	\mathcal{D}^a(f)\coloneqq \mathcal{D}(f)^a, \quad \mathcal{D}^\perp(f)\coloneqq\mathcal{D}(f)^\perp,
\end{equation}
so that $\mathcal{D}(f)^a$ and $\mathcal{D}(f)^\perp$ denote the Lebesgue decomposition of the signed measure $\mathcal{D}(f)$ with respect to $\mu$.
It is straightforward to check that both $\mathcal{D}^a$ and $\mathcal{D}^\perp$ are themselves measure-valued derivations.
Moreover, if $\mathcal{D}$ has finite mass, then both parts also
have finite mass, and
\begin{equation}\label{eq:decomposition_mass_derivation}
|\mathcal{D}^a| = |\mathcal{D}|^a,
\qquad
|\mathcal{D}^\perp| = |\mathcal{D}|^\perp.
\end{equation}
\begin{definition}[Minimal solutions to \eqref{eq:CE_D}]\label{def:minimalsolution_M}
We call a solution $(\mu,\mathcal{D})$ to \eqref{eq:CE_D} a \emph{minimal solution} if $\mathcal{D}^\perp$ is a minimal derivation in the sense of \cref{def:minimalderivation}, where $\mathcal{D}=\mathcal{D}^a+\mathcal{D}^\perp$ is the Lebesgue decomposition with respect to $\mu$.
\end{definition}

We conclude this subsection with the following observation. Recall from \cref{cor:derivation_CE} that whenever $(\mu,\mathcal{D})$ solves \eqref{eq:CE_D}, the mass measure $\lambda$ of $\mathcal D$ satisfies $|\D\mu|(I)\leq c^{-1} \lambda(I\times X)$. The following shows that whenever this lower bound is attained, i.e., $\mathcal D$ is optimal, the corresponding solution is necessarily minimal.
This criterion will be particularly useful in
Section \ref{subsec:lifts_to_derivation}.

\begin{proposition}\label{prop:Opt->Min}
Let $\mu\in \M^+(I\times X)$ and let $\mathcal{D}$ be a finite-mass measure-valued derivation such that $(\mu,\mathcal{D})$ solves the continuity equation \eqref{eq:CE_D}.
Assume
\begin{align}\label{eq:optimal_condition}
    |\D\mu|(I)=c^{-1}\lambda(I\times X).
\end{align}
where $c=T^{-1} \| \mu\|_{\mathrm{TV}}$ and $\lambda$ is the mass measure of $\cD$ in the sense of Definition~\ref{def:massD}. Then $(\mu,\mathcal{D})$ is a minimal solution.
\end{proposition}
\begin{proof}
Let $\tilde{\mathcal{D}}^\perp$ be a subderivation of $\mathcal{D}^\perp$ such that $\mathcal{D}^\perp - \tilde{\mathcal{D}}^\perp$ is divergence-free in the sense of \cref{def:minimalderivation}.
We need to show that $\mathcal{D}^\perp = \tilde{\mathcal{D}}^\perp$.
Denote by $\tilde \lambda^\perp$ the mass measure of $\tilde{\cD}^\perp$.
By Remark~\ref{remark:subderivation}~\ref{item:rmksubderivation2} and \eqref{eq:decomposition_mass_derivation}, we have $\tilde\lambda^\perp\leq \lambda^\perp$.
It is readily checked that $(\mu, \mathcal{D}^a + \tilde{\mathcal{D}}^\perp)$ also solves the continuity equation. We have
    \begin{equation}
          \lambda^a (I \times X) + \lambda^\perp (I \times X) = \lambda (I \times X) =c |\D \mu| (I)\leq \lambda^a (I \times X) +\tilde\lambda^\perp(I \times X) ,
    \end{equation}
    where we used assumption \eqref{eq:optimal_condition} and Corollary  \ref{cor:derivation_CE}.
    Since $\tilde\lambda^\perp\leq \lambda^\perp$, this forces $\lambda^\perp = \tilde\lambda^\perp$. Therefore, $ |\mathcal{D}^\perp - \tilde{\mathcal{D}}^\perp| = \lambda^\perp -  \tilde\lambda^\perp = 0 $ and thus $\mathcal{D}^\perp = \tilde{\mathcal{D}}^\perp$. 
    This proves that $\mathcal{D}^\perp$ is a minimal derivation, thereby concluding that $(\mu, \mathcal{D})$ is a minimal solution.
\end{proof}

\subsection{Derivations on \texorpdfstring{$\mathbb{R}^n$}{Rn}}\label{sec:derivationR^n}
In this subsection, we investigate the relationship between measure-valued derivations on $\R^n$ and $\R^n$-valued measures on $I\times\R^n$ via the notion of directional derivation.
These questions are of independent interest and will also play a crucial role in establishing the consistency between different notions of CE in \cref{sec:consistencyCE}.
\begin{definition}[Directional derivation]\label{def:directionalderivation}
    Let $\nnu \in \M(I\times\mathbb{R}^n;\R^n)$.
    We say a measure-valued derivation $\mathcal{D}$ is a \emph{$\nnu$-directional derivation} if $\mathcal{D}(f)=\nabla f\cdot\nnu$ for every $f\in C^1_c(\R^n)$.
\end{definition}
\begin{theorem}\label{thm:DtoV}
      Every continuous measure-valued derivation $\mathcal{D}$ is a $\nnu$-directional derivation for some $\nnu\in \M(I\times \R^n;\R^n)$.
      If in addition $\mathcal{D}$ is weak$^*$-type continuous, then
      \begin{enumerate}[label=(\arabic*), font=\normalfont]
          \item $\cD$ has finite mass;
          \item there exists a unique $\nnu\in \M(I\times \R^n;\R^n)$ such that $\cD$ is a $\nnu$-directional derivation, and the mass measure of $\cD$ equals $|\nnu|$.
      \end{enumerate}
\end{theorem}
\begin{proof}
    For each $i=1,\cdots, n$ we denote by $x_i$ the coordinate function $(x_1,\cdots,x_n)\mapsto x_i$ on $\R^n$.
   By \cref{prop:extension}, $ \cD(x_i)\in \M(I\times \R^n;\R)$.
Denote $\nu_i\coloneqq \cD(x_i)$ and $\nnu\coloneqq (\nu_1,\cdots,\nu_n)\in\M(I\times\R^n;\R^n)$.
Again due to the chain rule in Proposition~\ref{prop:extension}~\ref{item:extension1}, we have for every $f\in C^1(\R^n)\cap \Lip(\R^n)$
\begin{align}
\cD(f)=\cD(f(x_1,\cdots,x_n))=\sum_{i=1}^n\partial_{x_i}f\cD(x_i)=\nabla f\cdot \nnu.
\end{align}
In particular, $\cD$ is a $\nnu$-directional derivation.
Moreover, denoting $\lambda\coloneqq |\nnu|$, it follows that
\begin{equation}\label{ineq:proof_mass}
    |\cD(f)|\leq |\nabla f|_*\cdot |\nnu|=\lip_{\rm a}f\cdot \lambda\leq \Lip(f)\lambda,\quad \forall f\in C^1(\R^n) \cap \Lip(\R^n).
\end{equation}\smallskip
Now assume that $\cD$ is weak$^*$-type continuous.
For each $f\in \Lipb(\R^n)$, by \cite[Lemma A.3]{Stepanov-Trevisan2017}, there exists a sequence $(f_n)_n\subset C^1(\R^n)\cap \Lipb(\R^n)$ 
\[
f_n\wsconver f,\quad \Lip(f_n)\leq \Lip(f),\ \forall n.
\]
By assumption, $\cD(f_n)\wsconver \cD(f)$ and thus passing to the limit in \eqref{ineq:proof_mass} yields $|\cD(f)|\leq \Lip(f)\lambda$.
This means, by duality, that there exists a continuous $\lambda$-derivation $V$ such that $(V,\lambda)$ is a representation of $\cD$.
By \cref{prop:derivation} \ref{item:selfimprovement_massbound}, one gets the improvement
\[
|\cD(f)|\leq \lip_{\rm a}f\cdot \lambda,\quad \forall f\in \Lipb(X).
\]
As $\cD$ has finite mass, the fact that $|\nnu|$ is the mass measure of $\cD$ follows from \cref{lemma:finitemassD} \ref{item:massM3} using the very definition of the $\nnu$-directional derivation.
\\
Finally, the uniqueness of $\nnu$ follows from the uniqueness of the mass measure.
Indeed, if $\cD$ is a direction derivation with respect to both $\nnu_1,\nnu_2\in \M(I\times \R^n;\R^n)$.
The preceding discussion demands that $|\nnu_1|=|\nnu_2|$.
Then to satisfy 
\[
\cD(f)=\nabla f\cdot \nnu_1=\nabla f\cdot \nnu_2,\quad \forall f\in C^1_c(\R^n),
\]
forces $\nnu_1=\nnu_2$.
\end{proof}

\cref{thm:DtoV} says that a weak$^*$-type continuous measure-valued derivation can canonically induce a ``flux measure".
Conversely, although we expect for every $\nnu\in\M(I\times \R^n;\R^n)$ there are (not necessarily uniquely) continuous $\nnu$-direction derivations, such derivations can not in general be weak$^*$-type continuous.
Indeed, for any Radon measure $\lambda$ that is concentrated on
a purely 1-unrectifiable set of Hausdorﬀ dimension (at most) one (e.g. Cantor sets), the only weak$^*$-continuous derivation from $\Lipb(\R^n)$ to $L^\infty(\R^n;\lambda)$ is $0$; see \cite[Lemma 3.7]{Gong12}.
It is therefore natural to ask for a characterization of all $\nnu\in \M(I\times \R^n;\R^n)$ that induces a weak$^*$-type
continuous derivation.
Nevertheless, we give a sufficient condition in the following, which is enough for the purpose of this work.

Let $\mu_1,\mu_2\in\M(I\times\R^n;\R)$.
We say that a vector-valued measure $\nnu\in\M(I\times\R^n;\R^n)$ \emph{has divergence} $\Div \nnu= \mu_1-\partial_t\mu_2$, if for all $\varphi\in C^1_c(I\times \R^n)$
\begin{equation}\label{eq:divergence}
		- \int_{I\times \R^n} \nabla \varphi\cdot \d\nnu=\int_{I\times \R^n} \varphi\d\mu_1+\int_{I\times \R^n} \partial_t\varphi\d\mu_2.
\end{equation}

\begin{proposition}\label{prop:measuretoderivation}
	Assume that $\nnu\in\M(I\times\R^n;\R^n)$ has divergence $\Div \nnu= \mu_1-\partial_t\mu_2$. 
    Then the map $C^1_c(\R^n)\ni f\mapsto \nabla f\cdot \nnu$ can be uniquely extended to a weak$^*$-type continuous derivation.
\end{proposition}
\begin{proof}
	Let $f \in \Lipb(\R^n)$, and let $(f_n)_n$ be a sequence of $C^1_c$-functions converging pointwise to $f$ with $\|f_n\|_{\Lipb} \leq L$ for some $L > 0$ and all $n \in \N$.
    By assumption, there exist $\mu_1,\mu_2\in \M(I\times \R^n;\R)$ satisfying the identity \eqref{eq:divergence} for all $\varphi\in C^1_c(I\times \R^n)$
	\\
	We claim for any $g \in L^1(|\nnu|)$ that the following sequence is Cauchy
	\[
	\int g(t,x) \nabla f_n(x) \cdot \d\nnu(t,x).
	\]
	Assume first that $g \in C^1_c(I \times \R^n)$. Plugging $\varphi=g\cdot f_n$ into \eqref{eq:divergence} obtains
	\begin{align}
		- \int g(t,x) \nabla f_n(x) \cdot \d\nnu(t,x)
		&= - \int \nabla (g f_n)(t,x) \cdot \d\nnu(t,x) + \int (f_n \nabla g)(t,x) \cdot \d\nnu(t,x) \\
		&=\int f_ng\d\mu_1+ \int f_n \partial_t g \d\mu_2 + \int f_n \nabla g \cdot \d\nnu.
	\end{align}
	By the dominated convergence theorem and the pointwise convergence of $f_n$, the integrals above converge as $n \to \infty$.
	\\
	Now for general $g \in L^1(|\nnu|)$, consider a sequence $(g_m) \subset C^1_c(I \times \R^n)$ which converges to $g$ in $L^1(|\nnu|)$. For any $\varepsilon > 0$, there exists $m_0 \in \N$ such that
	\begin{equation}\label{ineq:itisCauchy1}
		\int |(g_{m_0} - g) \nabla f_n|_{*} \d|\nnu| \leq \sup_n \Lip(f_n) \|g_{m_0} - g\|_{L^1(|\nnu|)} \leq \frac{\varepsilon}{2}.
	\end{equation}
	Since for any $m_0\in\N$, $\int g_{m_0}(t,x) \nabla f_n(x) \cdot \d\nnu(t,x)$ is a Cauchy sequence in $n$, there exists $n_0 \in \N$ such that for all $n_1, n_2 > n_0$,
	\begin{equation}\label{ineq:itisCauchy2}
		\int g_{m_0}(t,x) (\nabla f_{n_1} - \nabla f_{n_2})(x) \cdot \d\nnu(t,x)  \leq \frac{\varepsilon}{2}.
	\end{equation}
	Combining \eqref{ineq:itisCauchy1} and \eqref{ineq:itisCauchy2} shows the claim, which means the limit
	\[
	\lim_{n \to \infty} \int g(t,x) \nabla f_n(x) \cdot \d\nnu(t,x)
	\]
	exists. 
	One can check that this limit is independent of the choice of the sequence $(f_n)$.
	Therefore, for $f\in\Lipb(\R^n)$ one may define $\cD(f)$ by
	\[
	\int g(t,x) \d\cD(f)(t,x) \coloneqq \lim_{n \to \infty} \int g(t,x) \nabla f_n(x) \cdot \d\nnu(t,x), \quad \forall g \in L^1(|\nnu|),
	\]
    which by construction is weak$^*$-type continuous.
	The linearity and the Leibniz rule of $\cD$ are clear.
    The uniqueness follows from weak$^*$-density of $C^1_c(\R^n)$ and weak$^*$-type continuity of $\cD$.
\end{proof}

For $\nnu_1,\nnu_2\in \M(I\times \R^n;\R^n)$, we write $\Div \nnu_1=\Div \nnu_2$, if $\nnu_1-\nnu_2$ is divergence-free, that is
\begin{align}\label{eq:divergencefree}
   \int_{I\times \R^n} \nabla \varphi\cdot \d\nnu_1=\int_{I\times \R^n} \nabla \varphi\cdot \d\nnu_2, \quad \forall\varphi\in C^1_c(I\times \R^n).
\end{align}
\begin{remark}\label{remark:derivationfrommeasure}
	The proof of \cref{prop:measuretoderivation} implies also the following: if $\mathcal{D}_{\nnu}$ is weak$^*$-type continuous $\nnu$-directional derivation, then for any $\xxi\in\M(I\times \R^n;\R^n)$ with $\Div \xxi=\Div \nnu$, there exists a $\xxi$-directional derivation $\mathcal{D}_{\xxi}$ which is weak$^*$-type continuous.
\end{remark}

\begin{definition}[Submeasure]\label{def:submeasure}
Given $\nnu_1,\nnu_2\in \M(I\times \R^n;\R^n)$, we say that
 \begin{enumerate}[label=(\arabic*), font=\normalfont]
    \item\label{item:sub-measure-ARS} $\nnu_2$ is a \emph{submeasure} of $\nnu_1$ and write $\nnu_2 \precbar \nnu_1$ if there exists a Borel function $\rho\colon I\times \R^n\to [0,1]$ such that $\nnu_2=\rho \, \nnu_1$; 
    \item  $\nnu_2$ is a \emph{metric submeasure} of $\nnu_1$ and write  $\nnu_2 \prec\nnu_1$ if 
    $$|\nnu_2| + |\nnu_1 - \nnu_2| = |\nnu_1|.$$ 
\end{enumerate}
\end{definition}

The first definition is the one used in \cite{AlmiRossiSavare2025}, where it is denoted by $\prec$. 
We draw the reader's attention that our notation differs from theirs: here $\precbar$ denotes the partial order (1), while $\prec$ is
reserved for the metric submeasure relation (2). 
Our choice of symbols is meant to reflect that (1) is the more restrictive of the two relations i.e., $\nnu_2 \precbar\nnu_1$ implies $\nnu_2 \prec\nnu_1$.
If the norm $|{\cdot}|$ is strictly convex, the two notions coincide by \cite[Lemma~2.2]{AlmiRossiSavare2025}. 
We work with (2), as it is the notion compatible with our definition
of a subderivation.

 \begin{definition}[Minimal measures]\label{def:minimalmeasure}
    Given $\nnu_1\in \M(I\times \R^n;\R^n)$, we say that $\nnu_1$ is \emph{minimal} with respect to $\prec$ (resp. $\precbar$) if it does not have any non-trivial divergence-free submeasure, i.e., for any $\nnu_2\prec \nnu_1$ (resp. $\nnu_2 \precbar \nnu_1) $ with $\Div \nnu_2=0$, we have that $\nnu_2=0$.
    Equivalently, for any $\nnu_2\prec \nnu_1$ (resp. $\nnu_2\precbar\nnu_1)$ with $\Div \nnu_2=\Div \nnu_1$, we have that $\nnu_2=\nnu_1$.
  \end{definition}
    
\begin{theorem}[Consistency I, minimality]\label{thm:consistency}
	 Let $\mathcal{D}_{\nnu}$ be a weak$^*$-type continuous $\nnu$-directional derivation for $\nnu\in\M(I\times \R^n;\R^n)$. 
	 Then
	\begin{enumerate}[label=(\arabic*),font=\normalfont]
		\item for any weak$^*$-type continuous derivation $\mathcal{D}_1$, $\mathcal{D}_1\prec \mathcal{D}_{\nnu}$ if and only if $\mathcal{D}_1$ is a $\nnu_1$-directional derivation for a metric submeasure $\nnu_1\prec\nnu$;
		\item $\mathcal{D}_{\nnu}$ is a minimal derivation if and only if $\nnu$ is a minimal measure with respect to $\prec$.
	\end{enumerate}
\end{theorem}

In the theorem above, if the norm on $\R^n$ is strictly convex, then the order relation can also be replaced with $\precbar$.

\begin{proof}
	We divide the proof into two parts.\smallskip
	\\
    (1) Assume first that $\mathcal{D}_1\prec \mathcal{D}_{\nnu}$.
	By \cref{thm:DtoV}, we may assume that $\mathcal{D}_1$ and $\mathcal{D}_2\coloneqq\mathcal{D}_{\nnu}-\mathcal{D}_1$ are $\nnu_1$- and $\nnu_2$-directional derivation for $\nnu_1,\nnu_2\in\M(I\times \R^n;\R^n)$, respectively, where $\nnu_2=\nnu-\nnu_1$; and $|\nnu|,|\nnu_1|$ and $|\nnu_2|$ are the mass measures of the corresponding derivations.
	Since $\mathcal{D}_1\prec \mathcal{D}_{\nnu}$, we have that $|\nnu_1|+|\nnu_2|\leq |\nnu|$ (see \cref{remark:subderivation}).
    It follows that $\nnu_1\prec \nnu$.
	\\
	Now assume that $\mathcal{D}_1$ is a $\nnu_1$-directional derivation with $\nnu_1\prec\nnu$.
	Denote by $\nnu_2\coloneqq \nnu-\nnu_1$.
    Then $|\nnu_1|+|\nnu_2|\le |\nnu|$ due to the fact that $\nnu_1\prec \nnu$.
    Moreover, $|\nnu_2|$ is the mass measure of the $\nnu_2$-directional derivation $\cD_2\coloneqq \cD_{\nnu}-\cD_1$.
    Hence $\mathcal{D}_1\prec \cD_{\nnu}$.\medskip
	\\
    (2) For the consistency between two notions of minimality, we show first the ``only if" part. 
    Assume that $\mathcal{D}_{\nnu}$ is a minimal derivation.
    Let $\xxi\in\M(I\times\R^n;\R^n)$ be such that $\xxi\prec\nnu$ with $\Div \xxi=\Div\nnu$.
    By \cref{remark:derivationfrommeasure} there exists a weak$^*$-type continuous derivation $\mathcal{D}_{\xxi}$ which is a $\xxi$-directional derivation.
    Since $\xxi\prec \nnu$, by \cref{thm:DtoV} we have that $\cD_{\xxi}\prec \cD_{\nnu}$.
    \\
    We need to show that $\nnu=\xxi$.
   For any $f\in C_c^1(\R^n)$ and $g\in C^1_c(I)$
 \begin{align}
        \int g(t)\d (\mathcal{D}_{\nnu}-\mathcal{D}_{\xxi})(f)&=\int g(t)\nabla f(x)\cdot\d\nnu-\int g(t)\nabla f(x)\cdot\d \xxi=0
    \end{align}
    where the second equality is due to the assumption $\Div \xxi=\Div \nnu$.
    By the weak$^*$-continuity of $\mathcal{D}_{\nnu}-\mathcal{D}_{\xxi}$, the identity
    \begin{align}
     \int g(t)\d (\mathcal{D}_{\nnu}-\mathcal{D}_{\xxi})(f)=0
    \end{align}
    can be extended to all $f\in \Lipb(\R^n)$ and
    $g \in \mathcal{B}_{\mathrm{b}}(I)$.
    In particular, $\mathcal{D}_{\nnu}-\mathcal{D}_{\xxi}$ is a spatial cycle.
    On the other hand, we already know that $\cD_\nnu-\cD_\xxi\prec \cD_\nnu$.
    Since $\mathcal{D}_{\nnu}$ is a minimal derivation, we conclude that $\mathcal{D}_{\nnu}-\mathcal{D}_{\xxi}=0$ as a derivation.
    As $\mathcal{D}_{\nnu}-\mathcal{D}_{\xxi}$ is a $(\nnu-\xxi)$-directional derivation, we conclude $\xxi=\nnu$
\\
    Now for the ``if" part, assume that $\nnu$ is a minimal with respect to $\prec$.
    We want to show that $\mathcal{D}_{\nnu}$ is a minimal derivation.
    Let $\mathcal D_1$ be a subderivation of $\mathcal D_\nu$ such that $\mathcal{D}_{\nnu}-\mathcal{D}_1$ is a spatial cycle.
    Then $\mathcal{D}_1$ is weak$^*$-type continuous.
    By the first part of the theorem, there exists $\nnu_1\prec \nnu$ such that $\mathcal{D}_1$ is a $\nnu_1$-directional derivation.
    Then for any $f\in C^1_c(\R^n)$ and $g \in \mathcal{B}_{\mathrm{b}}(I)$ 
    \begin{align}
        \int_{I\times X} g(t)\d(\mathcal{D}_{\nnu}-\mathcal{D}_1)(f)=\int g(t)\nabla f\cdot \d (\nnu-\nnu_1)=0
    \end{align}
   Recalling by Nachbin's theorem (a Stone--Weierstrass type theorem for the space $C^1(I\times \R^n)$, see e.g. \cite[Theorem 3.3.1]{Bucur-Paltineanu20}) that any $\varphi\in C^1_c(I\times \R^n)$ can be approximated by a sequence of finite linear combinations of products of $C^1_c(I)$ and $C^1_c(\R^n)$ functions with respect to $C^1(U;\R)$-norm for a bounded open $U$ containing the support of $\varphi$.
   Therefore, $\Div (\nnu-\nnu_1)=0$ and the minimality of $\nnu$ demands that $\nnu_1=\nnu$.
   This implies $\mathcal{D}_{\nnu}-\mathcal{D}_1=0$.
\end{proof}

\subsection{Consistency between CE via flux and CE via derivation}\label{sec:consistencyCE}
Recall from \cref{def:ce} that, when the underlying space is $\R^n$, \cite{AlmiRossiSavare2025} introduced a notion of continuity equation in which the flux is given by a vector-valued measure, together with a corresponding notion of minimal solution; see \cref{def:ce_minimal_solution} below. 
This raises the natural question of how this formulation is related to the derivation-based approach. 
Utilizing the properties of directional derivations developed in \cref{sec:derivationR^n}, we establish a correspondence between (minimal) solutions to these two types of continuity equations.

\begin{theorem}[Consistency II, CE]\label{thm:nu_to_V}
Let $\mu\in\M^+(I\times \R^n)$. 
\begin{enumerate}[label=(\arabic*),font=\normalfont]
    \item\label{item:THM_CE1} Let $(\mu,\nnu)$ be a solution to the continuity equation \eqref{eq:CE}.
    Then there exists a unique weak$^*$-type continuous measure-valued derivation $\mathcal{D}_{\nnu}$, which is a $\nnu$-directional derivation, and the pair $(\mu,\mathcal{D}_{\nnu})$ solves the continuity equation \eqref{eq:CE_D}.
    \item\label{item:THM_CE2} Conversely, let $(\mu,\cD)$ be a solution to the continuity equation \eqref{eq:CE_D} for a finite-mass derivation $\cD$.
    Then $\cD$ is a $\nnu$-directional derivation for $\nnu\in\M(I\times \R^n;\R^n)$ such that $(\mu,\nnu)$ solves the continuity equation \eqref{eq:CE}.
\end{enumerate}
\end{theorem}
\begin{proof}
  As $(\mu,\nnu)$ solves the continuity equation \eqref{eq:CE}, $\nnu$ has divergence in the sense of \eqref{eq:divergence}.
  Thus by \cref{prop:measuretoderivation}, there exists a unique weak$^*$-type continuous $\nnu$-directional derivation $\mathcal{D}_{\nnu}$. 
For every $f\in C^1_c(\R^n)$ and $\xi\in C^1_c(I)$, taking $\varphi=\xi\cdot f$ in \eqref{eq:CE} with $f\in C^1_c(\R^n)$ gives
  \begin{align}
  	\int \xi(t) \d \cD(f)(t,x)&=\int \nabla(\xi(t) f(x))\cdot\d\nnu(t,x)=-\int \xi'(t)f(x)\d\mu(t,x).
  	\end{align}
  The above identity can be extended to all $f\in \Lipb(X)$ by the weak$^*$-type continuity of $\cD$, which verifies the continuity equation \eqref{eq:CE_Mderivation}.\medskip
  \\
  For the reverse direction, by \cref{cor:derivation_CE} $\cD$ is weak$^*$-type continuous, and by \cref{thm:DtoV} there exists $\nnu\in\M(I\times\R^n;\R^n)$ such that $\mathcal{D}$ is a $\nnu$-directional derivation. 
  \\
   Thus, for each $\xi\in C^1_c(I)$ and $f\in C^1_c(\R^n)$, 
  \begin{equation}\label{eq:10/5-1}
  	-\int \xi(t)\nabla f(x)\cdot\d \nnu(t,x)=-\int \xi\d \cD(f)=  \int \xi'(t)f(x)\d\mu(t,x).
  \end{equation}
By the same Nachbin-type argument used in \cref{thm:consistency}, the identity \eqref{eq:10/5-1} can be extended to all $\varphi\in C^1_c(I\times \R^n)$ such that
\[
-\int_{I\times\R^n} \nabla \varphi\cdot \d\nnu=\int_{I\times\R^n} \partial_t\varphi \d\mu
\] 
which means that $(\mu,\nnu)$ solves the continuity equation \eqref{eq:CE}.
\end{proof}

 \begin{definition}[Minimal solutions to \eqref{eq:CE}]\label{def:ce_minimal_solution}
	We call a solution $(\mu,\nnu) \in \M^+(I\times \R^n)\times \M(I\times \R^n;\R^n) $ to \eqref{eq:CE} a \emph{minimal solution} with respect to $\prec$ (resp. $\precbar$) if $\nnu^\perp$ is a minimal measure with respect to $\prec$ (resp. $\precbar$) in the sense of \cref{def:minimalmeasure}, where
	$\nnu=\nnu^a+\nnu^\perp$ is the Lebesgue decomposition of $\nnu$ with respect to $\mu$.
\end{definition}

\begin{theorem}[Consistency III, minimal solution]\label{thm:correspondence_minimal_solutions}
	Let $\mu\in \M^+(I\times \R^n)$ and $\mathcal{D}$ be a finite-mass measure-valued derivation. Let $(\mu,\cD)$ be a solution to \eqref{eq:CE_D}, and let $\nnu$ be the flux measure associated to $\cD$ given by Theorem~\ref{thm:nu_to_V}\,\ref{item:THM_CE2}. Then $(\mu,\mathcal{D})$ is a minimal solution to \eqref{eq:CE_D} if and only if $(\mu,\nnu)$ is a minimal solution to \eqref{eq:CE} with respect to $\prec$.
    In particular, when the norm is strictly convex, the notion of minimal solution agrees with the one in \cite{AlmiRossiSavare2025}.
\end{theorem}

\begin{proof}
	Since $(\mu,\mathcal{D})$ is a solution to \eqref{eq:CE_D},    
    then by \cref{cor:derivation_CE} and \cref{thm:nu_to_V}\,\ref{item:THM_CE2}, $\mathcal{D}$ is a weak$^*$-type continuous $\nnu$-directional derivation for $\nnu\in \M(I\times \R^n;\R^n)$ such that $(\mu,\nnu)$ solves the continuity equation \eqref{eq:CE}.
    With the Lebesgue decomposition $\nnu=\nnu^a+\nnu^\perp$, it is clear that $\mathcal{D}^\perp $ is a $\nnu^\perp$-directional derivation.
    Actually, $\cD^\perp$ is also weak$^*$-type continuous.
    \\
    Indeed, let $A,S\subset I\times \R^n$ be disjoint Borel sets such that $\nnu^a$ and $\nnu^\perp$ are concentrated on $A$ and $S$, respectively.
     Denote by $(V,\lambda)$ the minimal representation of $\mathcal{D}$.
     By \cref{thm:DtoV}, $\lambda=|\nnu|$. 
     Furthermore, 
     \begin{equation}
     \lambda=\lambda^a+\lambda^\perp,\quad 	 \lambda^a=|\nnu|^a=|\nnu|\mathds{1}_A,\quad \lambda^\perp=|\nnu|^\perp=|\nnu|\mathds{1}_S.
     	\end{equation}
   For any $f\in\Lipb(\R^n)$, define 
   \begin{equation}
   	V^a(f)\coloneqq V(f)\mathds{1}_{A},\quad V^\perp(f)\coloneqq V(f)\mathds{1}_{S}.
   \end{equation}
   Note that $V^a$ and $V^\perp$ are weak$^*$-continuous $\lambda$-derivations.  
   As $(V,\lambda)$ is representation of $\mathcal{D}$, we have for any $f\in\Lipb(\R^n)$ that
   \begin{align}
   	\mathcal{D}^\perp(f)=(\mathcal{D}(f))^\perp=(Vf\cdot\lambda)^\perp=Vf\cdot \lambda^\perp=V^\perp(f)\lambda.
   \end{align}
In other words, $(V^\perp,\lambda)$ is representation of $\mathcal{D}^\perp$ and hence $\mathcal{D}^\perp$ is weak$^*$-type continuous.
\\
If $(\mu,\mathcal{D})$ is a minimal solution, then by definition  $\mathcal{D}^\perp$ is a minimal derivation.
It follows from \cref{thm:consistency} that $\nnu^\perp$ is a minimal with respect to $\prec$ and thus $(\mu,\nnu)$ is a minimal solution with respect to $\prec$.
Conversely, if $(\mu,\nnu)$ solves \eqref{eq:CE}, then by \cref{prop:measuretoderivation} there exists a unique weak$^*$-type continuous measure-valued derivation $\mathcal{D}$, which is a $\nnu$-directional derivation, and $(\mu,\mathcal{D})$ solves the continuity equation \eqref{eq:CE_D}.
Therefore, the required implication follows again by \cref{thm:consistency} repeating the argument as before.
\end{proof}

As a direct consequence of \cref{thm:derivation_CE} and \cref{thm:nu_to_V}, we recover part (2) of \cite[Theorem~3.4]{AlmiRossiSavare2025} in the second item of the corollary below.
Moreover, as a consequence of \cref{prop:Opt->Min} and \cref{thm:correspondence_minimal_solutions}, we conclude the third item. It is worth mentioning that, thanks to the results established in \cref{subsec:BV_extended}, we do not require the $\P_1$-assumption imposed in \cite{AlmiRossiSavare2025}

\begin{corollary}\label{cor:BVflux}
Let $(\mu,\nnu)$ be a solution to the continuity equation \eqref{eq:CE}.
	Then 
	\begin{enumerate}[label=(\arabic*), font=\normalfont]
		\item $\pr^I_\# \mu=c\cdot\mathcal{L}^1\llcorner I$, where $c=T^{-1}\|\mu\|_{\mathrm{TV}}$, and there exists a Borel family $(\mu_t)_{t \in I} \subset \P (X)$  (uniquely determined for a.e. $t\in I$) such that $\mu(\mathrm{d} t,\mathrm{d} x)\coloneqq c \mu_t(\mathrm{d} x) \mathrm{d} t$;
	   \item $(\mu_t)\in BV (I;(\P(\R^n),W_1))$ with
	   \begin{equation}\label{ineq:optnu}
	   |\D \mu| \leq c^{-1}\, {\Pr}^I_{\#} |\nnu|.
	   	\end{equation}
        \item\label{item:CE_minimal} $(\mu,\nnu)$ is a minimal solution with respect to $\prec$ to \eqref{eq:CE} if the equality in \eqref{ineq:optnu} is attained.
	\end{enumerate}
\end{corollary}

\section{Characterization of BV-Wasserstein curves}\label{sec:characterization}

\subsection{From BV-Wasserstein curves to path measures: metric setting}\label{subsec:lift_to_mu}

The goal of this section is to represent BV-curves $(\mu_t)$ in the extended metric space $(\P(\X),W_1)$ by probability measures concentrated on BV-curves, extending our previous result \cite{AbediLiSchultz2024} for $(\P_1(\X),W_1)$.
This extension is useful for applications, where the measures $\mu_t$ under consideration need not have finite first moment. While the result is expected and is not the main result of the present work, its proof requires some technical care: unlike in the absolutely continuous setting, a single truncation is not sufficient for BV-curves, and a limiting procedure is needed. Readers interested in the subsequent results may skip the proofs in this section and use only the statements.

\subsubsection{From path measures to BV-Wasserstein curves}

Throughout Section~\ref{subsec:lift_to_mu}, $(X, d)$ is a complete separable metric space, and $I \coloneqq (0,T) \subset \mathbb{R}$ is an open time interval\footnote{If the open interval is replaced by a closed one, the statements in this section remain valid. In particular, this can be achieved for Theorem \ref{thm:optimal_lift} by considering the constant extension of the curve $(\mu_t)$ to a slightly larger open interval.}. 
We recall that $D(I;\X)$ denotes the space of c\`adl\`ag paths equipped with the Skorokhod topology.

\begin{definition}[Lift]\label{def:lift}
    Given $(\mu_t) \coloneqq (\mu_t)_{t \in I} \subset \P(\X)$, we call $\pi \in \P(D(I;\X))$ a \emph{lift} of $(\mu_t)$ if 
\begin{equation}
 (e_t)_{\#} \pi  = \mu_t \quad \text{ for all } t \in I,
\end{equation}
where $e_t: \gamma \in D(I;X) \mapsto \gamma_t \in X$ is the evaluation map.
\end{definition}

\begin{proposition}\label{prop:lift_to_mut}
    Let $(\X,d)$ be a complete separable metric space, and $I \coloneqq (0,T) \subset \mathbb{R}$. Let $\pi \in \P(D(I;\X))$ be concentrated on $\mathcal{BV}(I;\X) \subset D(I;\X)$ such that
    \begin{equation}\label{eq:Dmu_inequality}
        \int  | \D \gamma | (I) \d \pi (\gamma) < + \infty .
    \end{equation}
    Then the curve $t\mapsto\mu_t \coloneqq (e_t)_\# \pi$ is in $\mathcal{BV}(I;(\P(\X),W_1))$, and moreover,
    \begin{equation}\label{eq:easyineq}
        |\D \mu | \leq \int |\D \gamma| \d\pi (\gamma).
    \end{equation}
\end{proposition}

\begin{proof}
    The proof goes along the same steps as in \cite[Theorem 3.1]{AbediLiSchultz2024}, with the extension to the extended metric space $(\P(X), W_1)$.
    For $\pi$-a.e.\ $\gamma$ and all $s<t$ in $I$, we have $ d(\gamma_s,\gamma_t)\le |\D\gamma|((s,t])$
	(see e.g. \cite[Lemma 2.5 and Remark 2.20]{AbediLiSchultz2024}).
	Thus, for all $s<t$ in $I$, we have 
	\begin{align}
		W_1(\mu_s,\mu_t)
		&\leq \int_{D(I;X)} d(\gamma_s,\gamma_t)\,\d\pi(\gamma) \\
		&\leq \int_{D(I;X)} |\D\gamma|((s,t])\,\d\pi(\gamma) = \left(\int_{D(I;X)} |\D\gamma| \,\d\pi(\gamma)\right) ((s,t]). \label{eq:proof_lift_to_mu}
	\end{align}
    We show that $t \mapsto \mu_t$ is right-continuous in $(\P(X),W_1)$. Fix $t\in I$ and let $(t_n)_{n \in \N}\subset I$ be such that $t_n \downarrow t$. By the right continuity of curves and the dominated convergence theorem, we have 
	\begin{equation}
	   W_1(\mu_t,\mu_{t_n})\leq  \textstyle\int_{D(I;X)} d(\gamma_t,\gamma_{t_n})\,\d\pi(\gamma)  \rightarrow 0
	\qquad\text{as } n \to \infty.
	\end{equation}
    Next, we show that $t \mapsto \mu_t$ has left limits. Fix $t\in I$ and let $(t_n)_{n \in \N}\subset I$ be such that $t_n\uparrow t$. Then for $n<m$, we have by \eqref{eq:proof_lift_to_mu} that
	\begin{equation}
	   W_1(\mu_{t_n},\mu_{t_m}) \leq  ( \textstyle\int |\D\gamma| \,\d\pi ) ((t_n,t_m]) \leq  (\int |\D\gamma| \,\d\pi ) ((t_n,t)).
	\end{equation}
	Since $\int |\D\gamma| \,\d\pi$ is a finite non-negative measure on $I$, we have $(\int |\D\gamma| \,\d\pi)((t_n,t))\to 0$ as $n\to\infty$, and therefore $(\mu_{t_n})_n$ is a Cauchy sequence in $(\P(X),W_1)$. By completeness of $(\P(X),W_1)$ as an extended metric space (\cref{prop:wassersteinspacecomplete}), there exists $\bar\mu_t\in \P(X)$ such that $
	W_1(\mu_{t_n},\bar\mu_t)\to 0$.
	So we have shown $(\mu_t)\in D(I;(\P(X),W_1))$. In particular, it is Borel measurable, and therefore $(\mu_t) \in L^0 (I;(\P(X),W_1))$. Again, since $(P(X),W_1)$ is complete as an extended metric space, we can apply \Cref{thm:BVequiv} and with \eqref{eq:proof_lift_to_mu} conclude that $(\mu_t) \in BV (I;(\P(X),W_1))$. 
    Furthermore, by \Cref{remark:variationmeasure}, we have \eqref{eq:easyineq}. This completes the proof. 
\end{proof}

\subsubsection{From BV-Wasserstein curves to optimal lifts}\label{subsec:optimal_lift}
\begin{theorem}\label{thm:optimal_lift}
    Let $(X,d)$ be a complete separable metric space, and $I \coloneqq (0,T) \subset \mathbb{R}$. Let $(\mu_t)\in \mathcal{BV}(I;(\P(\X),W_1))$.
    Then there exists a probability measure $\pi \in 
    \P(D(I;\X))$ such that
    \begin{enumerate}[label=(\roman*), font=\normalfont]
        \item $\pi$ is concentrated on $\mathcal{BV}(I;\X)\subset D(I;\X)$;
        \item $(e_t)_\#\pi = \mu_t$ for all $t\in I$;
        \item\label{itm:totalvarian_overview} $\pi$ realizes the variation measure $|\D\mu|$ in the sense that
        \begin{equation}\label{eq:optimal_pi_identity}
        |\D\mu|= \int |\D\gamma|\d\pi(\gamma).
        \end{equation}
    \end{enumerate}    
\end{theorem}

\begin{proof}
     To prove this theorem, we use our previous result in \cite[Theorem 3.3]{AbediLiSchultz2024}, which was established for curves in $(\P_1(X),W_1)$. 
     To extend it to the extended metric space $(\P(X),W_1)$, we use a truncation argument: for each $n \in \N$, we truncate the distance $d$ by $n$, construct lifts for the truncated cases, and then pass to the limit $n \to \infty$. 
     In the absolutely continuous case proven in \cite[Corollary 1]{Lisini2007}, a single truncation is enough, since the metric derivative of curves in the base space remains unchanged under truncation of the distance. In contrast, the variation of curves changes, and therefore one needs to work with the full sequence and a limiting argument. For the latter, we draw on ideas from the proofs of \cite[Theorem 4.6 and Proposition 5.5]{AmbrosioErbarSavare2016}.
     
    \smallskip
    \noindent
    \textbf{Step 0} (Family $\{\pi_n\}_{n \in \N} \subset \P (D(I;X))$). For any $n \in \N$, we define the truncated distance 
    \begin{equation}
    	d_n (x,y) \coloneqq d(x,y) \wedge n, \qquad \forall x,y \in X.
    \end{equation}
    We denote the bounded metric space $(X,d_n)$ by $X_n$, and note that $d_n$ induces the same topology as $d$. We have $D(I;X_n) = D (I;X)$ and $BV (I;X_n) = BV (I;X) $.
    We denote by $\Var_n(u)$ and $ \essVar_n(u)$ the pointwise variation and  essential variation of a BV-curve $u$, respectively, computed with respect to the truncated distance $d_n$.
    \\
    We have $\P(X) = \P(X_n) = \P_1(X_n)$.
    Let $W_{1,n}$ be the 1-Wasserstein distance associated with $d_n$, and observe that
    \begin{equation}
    	W_{1,n} (\textup{m}_1,\textup{m}_2) \leq W_1(\textup{m}_1,\textup{m}_2) \wedge n, \qquad \forall \textup{m}_1,\textup{m}_2 \in \P(X). 
    \end{equation}
    We use analogous notation for the variations induced by $W_{1,n}$. 
    So we have
    \begin{equation}\label{eq:proof_unif_bound_essVar}
    	\essVar_n (\mu; I) \leq \essVar (\mu;I) < + \infty,
    \end{equation} 
    where the right-hand side denotes the essential variation of $(\mu_t)$ in the extended metric space $(\P(X),W_{1})$.
    Therefore, $(\mu_t) \in BV (I;\P_1(X_n))$. Moreover, since convergence in $W_1$ implies convergence in $W_{1,n}$, we have $(\mu_t) \in D (I;\P_1(X_n))$. Thus, we have $(\mu_t) \in  \BV (I;\P_1(X_n)).$ 
    Applying \cite[Theorem 3.3]{AbediLiSchultz2024}, we obtain $\pi_n \in \P (D(I;X_n))$ such that
     \begin{enumerate}[label=(\Roman*), font=\normalfont]
    	\item $\pi_n$ is concentrated on $\mathcal{BV}(I;X_n) \subset D(I;\X)$;
    	\item $(e_t)_\#\pi_n = \mu_t$ for all $t\in I$;
    	\item $\pi_n$ satisfies
        \begin{equation}\label{eq:proof:opt_pi_n}
            \int_{D(I;X)}  \essVar_n (\gamma;I) \d\pi_n(\gamma) = \essVar_n (\mu;I).
        \end{equation}
    \end{enumerate}
    Note that we can regard $\pi_n$ as an element of $\P(D(I;X))$, because $D(I;X_n)=D(I;X)$ as mentioned above and the fact that corresponding Skorokhod Borel $\sigma$-algebras coincide. 
    
    \smallskip
    \noindent
    \textbf{Step 1} (Construction of $\widetilde{\pi} \in \P \big(X^{\widetilde{I}}\big)$). Let $\widetilde{I} \subset I$ be a countable dense subset, and consider the restriction map 
    $$
    e_{\widetilde{I}} :D(I;X) \to X^{\widetilde{I}}, \qquad e_{\widetilde{I}}(\gamma) \coloneqq \gamma\big|_{\widetilde{I}},
    $$
    from $D(I;X)$ equipped with the Skorokhod topology to $X^{\widetilde{I}}$ equipped with the product topology.
    Since $(X,d)$ is a complete separable metric space and $\widetilde I$ is countable, the product space $X^{\widetilde I}$ equipped with the product topology is again Polish. 
    The restriction map $e_{\widetilde{I}} :D(I;X) \to X^{\widetilde{I}}$ is Borel measurable, because the evaluation maps $e_t:D(I;X)\to X$ are Borel measurable (see e.g. \cite[Proposition 2.15]{AbediLiSchultz2024}). 
    For each $n \in \N$, we define 
    \begin{equation}\label{eq:proof_def_tilde_pi}
    	\widetilde{\pi}_n \coloneqq (e_{\widetilde{I}})_{\#} \pi_n \in \P \big(X^{\widetilde{I}}\big). 
    \end{equation}
    We want to show that $\{\widetilde{\pi}_n\}_{n} \subset \P \big(X^{\widetilde{I}}\big)$ is tight.
    Note that for each fixed $ t \in \widetilde{I}$, we have 
    \begin{equation}\label{eq:proof_lift_et_tilde_pi}
    	(e_t)_{\#} \widetilde{\pi}_n = (e_t)_{\#} \pi_n = \mu_t,
    \end{equation}
    thus $\{(e_t)_{\#} \widetilde{\pi}_n \}_{n} \subset \P(X)$ is obviously tight. So by enumerating the points of $\widetilde{I}$ as $\widetilde{I} = \{t_m\}_{m \in \N}$, we have for any $l \in \N$ and $m \in N$ that there exists a compact set $K_{m,l} \subset X$ such that
    $$
    \sup_{n} \, (e_{t_m})_{\#} \widetilde{\pi}_n \big(X \setminus K_{m,l}\big) \leq 2^{-m-l}.
    $$
    By taking $K_l \coloneqq \cap_{m=1}^\infty \{ \widetilde{\gamma} \in X^{\widetilde{I}} : \widetilde{\gamma}_{t_m} \in K_{m,l} \} $, which is $\prod_{m=1}^\infty K_{m,l}$ after identifying $X^{\widetilde I}$ with $\prod_{m=1}^\infty X$, we have that  $K_l \subset X^{\widetilde I} $ is compact in the product topology and
    $$
    \sup_{n} \, \widetilde{\pi}_n \big(X^{\widetilde{I}} \setminus K_{l}\big) \leq \sum_{m=1}^{\infty} 2^{-m-l} = 2^{-l}.
    $$
    Hence, $\{\widetilde{\pi}_n\}_{n} \subset \P \big(X^{\widetilde{I}}\big)$ is tight. Then, by Prokhorov theorem on the Polish space $X^{\widetilde{I}}$, there exists a subsequence $\{\widetilde{\pi}_{n_k}\}_{k \in \N}$ such that $\widetilde{\pi}_{n_k} \to \widetilde{\pi} $ narrowly as $k \to \infty$ to a limit
    point $\widetilde{\pi} \in \P \big(X^{\widetilde{I}}\big)$.

    \smallskip
    \noindent
    \textbf{Step 2} (Properties of $\widetilde{\pi} \in \P \big(X^{\widetilde{I}}\big)$).
    We show that
     \begin{enumerate}[label=(\Roman*), font=\normalfont]
    	\item $\widetilde{\pi}$ is concentrated on $pBV(\widetilde{I};\X) \coloneqq \{ \gamma \in X^{\widetilde{I}} : \Var(\gamma;\widetilde{I})< + \infty \}$;
    	\item $(e_t)_\#\widetilde{\pi} = \mu_t$ for all $t\in \widetilde{I}$;
        \item $\widetilde{\pi}$ satisfies
        \begin{equation}\label{eq:proof_Var_tilde_pi}
           \int_{X^{\widetilde{I}}} \Var( \widetilde{\gamma}; \widetilde{I}) \d \widetilde{\pi} (\widetilde{\gamma}) \leq \essVar (\mu;I). 
        \end{equation}
    \end{enumerate}
    We begin with (II). Let $t\in \widetilde{I}$. For any $\varphi \in C_b(X)$, we have
    \begin{equation}
        \int_{X^{\widetilde{I}}} \varphi (\widetilde{\gamma}_t) \d \widetilde{\pi} (\widetilde{\gamma}) = \lim_{k \to \infty} \int_{X^{\widetilde{I}}} \varphi (\widetilde{\gamma}_t) \d \widetilde{\pi}_{n_k} (\widetilde{\gamma}) = \int_X \varphi (x) \d \mu_t,
    \end{equation}
   which follows from $\widetilde{\pi}_{n_k} \to \widetilde{\pi} $ narrowly on $\P(X^{\widetilde{I}})$, the fact that $\varphi \circ e_t \in C_b (X^{\widetilde{I}})$ (recalling that the evaluation map $e_t$ is continuous with respect to the product topology), and the identity \eqref{eq:proof_lift_et_tilde_pi}.
    \\
    Next, we prove (I) and (III). We have
    \begin{align}
        \int_{X^{\widetilde{I}}} \Var( \widetilde{\gamma}; \widetilde{I}) \d \widetilde{\pi} (\widetilde{\gamma})
        & = \lim_{n \to \infty} \int_{X^{\widetilde{I}}} \Var_n( \widetilde{\gamma}; \widetilde{I}) \d \widetilde{\pi} (\widetilde{\gamma}) \\
        & \leq \lim_{n \to \infty} \left( \liminf_{k \to \infty} \int_{X^{\widetilde{I}}} \Var_n( \widetilde{\gamma}; \widetilde{I}) \d \widetilde{\pi}_{n_k} (\widetilde{\gamma}) \right)\\ 
        & \leq \liminf_{k\to\infty}\int_{X^{\widetilde I}}\Var_{n_k}(\widetilde\gamma;\widetilde I)\,d\widetilde\pi_{n_k}(\widetilde\gamma) \\
        & = \liminf_{k \to \infty} \int_{D(I;X)} \Var_{n_k}( e_{\widetilde{I}}(\gamma); \widetilde{I}) \d {\pi}_{n_k} ({\gamma}) \\
        & = \liminf_{k \to \infty} \int_{D(I;X)} \Var_{n_k}( \gamma; I) \d {\pi}_{n_k} ({\gamma}) \\
        & = \liminf_{k \to \infty} \essVar_{n_k}(\mu;I) \leq \essVar (\mu;I) < + \infty. 
    \end{align}
    The first step above follows from the pointwise identity $ \Var (\widetilde{\gamma}; \widetilde{I}) = \lim_{n \to \infty} \Var_n (\widetilde{\gamma}; \widetilde{I})$ for any $\widetilde{\gamma} \in X^{\widetilde{I}}$ and the monotone convergence theorem.
    The second step follows from the narrow convergence $\widetilde{\pi}_{n_k} \to \widetilde{\pi} $ on $\P(X^{\widetilde{I}})$, the lower semi-continuity of the pointwise variation with respect to convergence in the product topology, and \Cref{lemma:lsc_pointwise_Var} (noting that convergence in the product topology implies pointwise convergence).
    For the third step, note that for every fixed $n$, we have that $n \leq n_k$ for sufficiently large $k$, and thus $\Var_n \leq \Var_{n_k}$. The fourth line follows from the push-forward \eqref{eq:proof_def_tilde_pi}. In the fifth step, we applied \Cref{lemma:Var_dense_subset}. Finally, the last line follows from \eqref{eq:proof:opt_pi_n}, \eqref{eq:proof_unif_bound_essVar}, and the fact that for c\`adl\`ag curves on open intervals the essential variation coincides with the pointwise variation (see \cite[Lemma 2.5]{AbediLiSchultz2024}).
    
    \smallskip
    \noindent
    \textbf{Step 3} (Construction of $\pi\in\P(D(I;X))$).
    By the previous Step, the measure $\widetilde\pi$ is concentrated on curves of bounded pointwise variation from $\widetilde I \to X$. By \Cref{lemma:cadlag_extension_dense_BV}, every such curve has a unique c\`adl\`ag reconstruction obtained by taking right limits. 
    Accordingly, we define the reconstruction map
    \begin{align}\label{eq:proof_reconstruction_formula}
    	R : pBV(\widetilde I;X) \to  D(I;X), \,\, \\
    	(R\widetilde\gamma)_t
    	\coloneqq
    	\lim_{\widetilde t\downarrow t,\ \widetilde t\in \widetilde I}\widetilde\gamma_{\widetilde t},
    	\qquad t\in I,
    \end{align}
    and we arbitrarily extend $R$ to all of $X^{\widetilde I}$ by  setting
    $(R\widetilde\gamma)_t\coloneqq \bar x, \, t\in I,$
    whenever $\widetilde\gamma\notin pBV(\widetilde I;X)$, where $\bar x\in X$ is an arbitrary fixed point.
    \begin{itemize}[itemsep=0.3em, leftmargin=1.5em] 
    	\item[] \textit{Claim.}
    	The map
    	$R:X^{\widetilde I}\to D(I;X) $
    	is Borel measurable, where $X^{\widetilde I}$ is equipped with the product topology and $D(I;X)$ with the Skorokhod topology.
    	\item[] \textit{Proof of Claim.}
    	First, we note that by \Cref{lemma:lsc_pointwise_Var}, the map $\widetilde\gamma\mapsto \Var(\widetilde\gamma;\widetilde I)\in [0,\infty]$	is lower semi-continuous on $X^{\widetilde I}$ with respect to the product topology, hence Borel measurable. In particular,
    	\begin{equation}
    		pBV(\widetilde I;X) = \{\widetilde\gamma\in X^{\widetilde I}:\Var(\widetilde\gamma;\widetilde I)<+\infty\}
    	\end{equation}
    	is a Borel subset of $X^{\widetilde I}$.
    	Second, to show that $R:X^{\widetilde I}\to D(I;X)$ is Borel measurable, it is enough to prove that $e_t\circ R:X^{\widetilde I}\to X$ is Borel measurable for every $t\in I$, because the evaluation maps $\{e_t:D(I;X)\to X\}_{t\in I}$ generate the Borel $\sigma$-algebra of $D(I;X)$.
    	Fix $t\in I$, and choose a sequence $(\widetilde t_m)_{m\in\N}\subset \widetilde I$ such that $\widetilde t_m\downarrow t$. For each $m\in\N$, the map $ e_{\widetilde t_m}:X^{\widetilde I}\to X $ 	is continuous with respect to the product topology, hence Borel measurable. By the definition of $R$, we have the pointwise limit
    	\begin{equation}
    		e_t\circ R=\lim_{m\to\infty} e_{\widetilde{t}_m}
    		\quad\text{ on } \, pBV(\widetilde I;X).
    	\end{equation}
    	Thus the restriction of $e_t\circ R$ to $pBV(\widetilde I;X)$ is Borel measurable as a pointwise limit of measurable maps.  	
    	On $X^{\widetilde I}\setminus pBV(\widetilde I;X)$, the map $e_t\circ R$ is constant equal to $\bar x$. Since $pBV(\widetilde I;X) \subset X^{\widetilde I}$ is a Borel subset, it follows that $e_t\circ R$ is Borel measurable on all of $X^{\widetilde I}$. This proves the claim.
        \end{itemize}

    \noindent
    Therefore, we define
    \begin{equation}\label{eq:proof_reconstruction_pi}
    	\pi\coloneqq (R)_{\#}\widetilde\pi\in \P(D(I;X)).
    \end{equation}

    \smallskip
    \noindent
    \textbf{Step 4} (Properties of $ \pi \in  \P (D(I;X))$). In the final step, we show that
    \begin{enumerate}[label=(\Roman*), font=\normalfont]
    	\item $\pi$ is concentrated on $\mathcal{BV}(I;\X) \subset D(I;\X)$;
        \item $(e_t)_\#\pi = \mu_t$ for all $t\in I$;
        \item $\pi$ satisfies 
        \begin{equation}\label{eq:proof:opt_pi}
            \int_{D(I;X)}  \essVar (\gamma;I) \d\pi(\gamma) = \essVar (\mu;I).
        \end{equation}
    \end{enumerate}
   We begin with (II). Let $t \in I$ be arbitrary. For any $\varphi \in C_b (X)$, we have 
    \begin{align}
     	\int_{D(I;X)}\varphi (\gamma_t) \d \pi (\gamma) 
     	& = \int_{X^{\widetilde{I}}} \varphi \big( (R\widetilde{\gamma})_t \big) \d \widetilde{\pi} (\widetilde{\gamma}) \\
     	& = \lim_{\widetilde t\downarrow t,\ \widetilde t\in \widetilde I} \int_{X^{\widetilde{I}}}   \varphi \big( \widetilde\gamma_{\widetilde t} \big) \d \widetilde{\pi} (\widetilde{\gamma}) \\
     	& = \lim_{\widetilde t\downarrow t,\ \widetilde t\in \widetilde I} \int_{X}   \varphi ( x) \d \mu_{\widetilde t} \, (x) = \int_{X}   \varphi ( x) \d \mu_{t} (x).
    \end{align}
    The first equality directly follows from the push-forward \eqref{eq:proof_reconstruction_pi}.
    For the second equality, we use the continuity of $\varphi$, the definition of $R$, and the dominated convergence theorem.
    The third equality follows from Step 2-(II). Since $(\mu_t) \in D (I;(\P(X),W_1))$, the map $t \mapsto \mu_t$ is in particular narrowly c\`adl\`ag, which yields the final equality. 
    \\
    Next, we prove (I) and (III). We have
    \begin{align}
    	\int_{D(I;X)} \essVar (\gamma;I) \d \pi (\gamma)
    	& =
    	\int_{D(I;X)} \Var (\gamma;I) \d \pi (\gamma) \\
    	& = \int_{X^{\widetilde{I}}} \Var (R\widetilde{\gamma};I) \d \widetilde{\pi} (\widetilde{\gamma}) \\
    	& \leq  \int_{X^{\widetilde{I}}} \Var (\widetilde{\gamma};\widetilde{I}) \d \widetilde{\pi} (\widetilde{\gamma}) \leq \essVar (\mu;I) < + \infty,
    \end{align}
    where we used the push-forward \eqref{eq:proof_reconstruction_pi}, \Cref{lemma:cadlag_extension_dense_BV}, and \eqref{eq:proof_Var_tilde_pi}.
    On the other hand, since $\pi$ is a lift of $(\mu_t)$, \Cref{prop:lift_to_mut} gives the reverse inequality. Thus, we have the equality \eqref{eq:proof:opt_pi}.
    This implies, again taking into account the observation in \Cref{prop:lift_to_mut}, that the equality \eqref{eq:optimal_pi_identity} holds as measures.
\end{proof}

Following Proposition \ref{prop:lift_to_mut} and Theorem \ref{thm:optimal_lift}, we define the following.
\begin{definition}[Optimal lift]\label{def:optlift}
Given $(\mu_t)\in \mathcal{BV}(I;(\P(\X),W_1))$, we call $\pi \in \P( D(I;X))$ concentrated on $\mathcal{BV}(I;\X)$ an \emph{optimal} lift of $(\mu_t)$, if it is a lift of $(\mu_t)$ (Definition~\ref{def:lift}) and it satisfies the optimality condition
\begin{equation} \label{eq:def_optlift}
	|\D\mu| (I) = \int |\D\gamma| (I)\,\d \pi(\gamma).
\end{equation}
\end{definition}

\begin{remark}\label{rmk:optimal_lift_measure_identity}
Note that \eqref{eq:def_optlift} implies $ |\D\mu| = \int |\D\gamma|\,\d\pi(\gamma)$
as measures on $I$, since Proposition \ref{prop:lift_to_mut} gives the corresponding measure inequality and \eqref{eq:def_optlift} gives equality of total masses.
\end{remark}

\subsection{From BV path measures to flux measures: \texorpdfstring{$\mathbb{R}^n$}{Rn} setting}\label{subsec:lifts_to_flux measures}

    This section provides an explicit construction of a flux measure from a path measure.
    To see where the construction ansatz comes from, we begin with a computation, temporarily setting aside measurability and integrability concerns; these will be addressed later in the proof of the next theorem.
	
    Let $I = (0,T) \subset \R$, and consider $\R^n$ endowed with a norm $|\,{\cdot}\,|$. Let $\pi \in \mathcal{P}(D (I;\R^n))$ be a path measure concentrated on $\BV (I;\R^n) \subset D(I;\R^n)$ such that
    \begin{equation}\label{eq:integrability_for_heuristic_computation}
        \int |\boldsymbol{ \mathrm{D}\gamma}|(I) \d \pi (\gamma) < + \infty.
    \end{equation}
    Denote $\mu_t\coloneqq (e_t)_{\#}\pi$ for all $t \in I$ (notice that 
    $t \mapsto \mu_t$ is narrowly c\`adl\`ag e.g. by \cite[Proposition 2.15]{AbediLiSchultz2024}) and define $\mu \in \M^+(I \times \R^n)$ via $\mu(\mathrm{d} t,\mathrm{d} x)\coloneqq \mu_t(\mathrm{d} x) \mathrm{d} t $.  
	For any test function $\varphi\in C^1_c(I\times \R^n)$, we have
	\begin{multline}
	\int_{I\times \R^n} \partial_t \varphi(t,x)\d \mu (t,x) = \int_{I} \int_{\R^n} \partial_t \varphi(t,x)\d \mu_t (x) \d t 
    = \int_I\int \partial_t \varphi(t,\gamma_t)\d\pi(\gamma)\d t \\
	=\int \bigg( \D\varphi_\gamma(I)- \int_I \nabla \varphi(t,\gamma_t)\cdot \d \boldsymbol{\D\gamma}^{\mathrm{c}}(t )-\sum_{t\in J_\gamma} \bigl[ \varphi(t,\gamma_{t+})-\varphi(t,\gamma_{t-})\bigr] \bigg) \d\pi(\gamma),
	\end{multline}
	where we applied the $\BV$-chain rule in \cref{lemma:chainrule}. Note that $\D\varphi_\gamma(I) = 0$ as $\varphi$ has compact support in the time variable. Therefore, if there exists a flux measure $\nnu$ such that $(\mu,\nnu)$ solves the continuity equation in the sense of \cref{def:ce}, then necessarily
	\begin{align}
		& \int_{I\times \R^n}  \nabla  \varphi(t,x)\cdot \d \nnu(t,x) \\
		&= \int \bigg( \int_I \nabla \varphi(t,\gamma_t)\cdot \d \boldsymbol{\D\gamma}^{\mathrm{c}}(t) +\sum_{t\in J_\gamma} \bigl[\varphi(t,\gamma_{t+})-\varphi(t,\gamma_{t-}) \bigr] \bigg) \d\pi(\gamma) \\
        &= \int \bigg( \int_{I} \nabla \varphi(t,\gamma_t)\cdot  \d \boldsymbol{\D\gamma}^{\mathrm{c}}(t) + \sum_{t\in J_\gamma} \int_{[0,1]} \nabla\varphi\bigl(t,(1-a)\gamma_{t-}+a \gamma_{t+}\bigr)\cdot(\gamma_{t+}-\gamma_{t-})\,\d a \bigg) \d\pi(\gamma) \\
		&= \int \bigg( \int_{I} \nabla \varphi(t,\gamma_t)\cdot  \d \boldsymbol{\D\gamma}^{\mathrm{c}}(t) +\sum_{t\in J_\gamma} \int_{[\gamma_{t-},\gamma_{t+}]} \nabla\varphi(t,x)\cdot \frac{(\gamma_{t+}-\gamma_{t-})}{|\gamma_{t+}-\gamma_{t-}|}\d \mathcal{H}^1(x) \bigg) \d\pi(\gamma) \\
		&=\int \bigg( \int_{I} \nabla \varphi(t,\gamma_t)\cdot  \d \boldsymbol{\D\gamma}^{\mathrm{c}}(t) + \int_{I} \int_{[\gamma_{t-},\gamma_{t+}]} \nabla\varphi(t,x)\cdot \frac{(\gamma_{t+}-\gamma_{t-})}{|\gamma_{t+}-\gamma_{t-}|^2}\d \mathcal{H}^1(x) \d |\boldsymbol{\mathrm{D}\gamma}^{\mathrm{j}}|(t) \bigg) \d\pi(\gamma),
	\end{align}
    where $\mathcal{H}^1$ is the 1-dimensional Hausdorff measure on $\R^n$ associated with the norm $|\cdot|$.
    This computation guides us to construct a flux measure $\nnu$ from $\pi$ in the next theorem via the formula
    \begin{multline}\label{eq:def_nu_inducedby_pi}
			\int_{I \times \R^n} F(t,x)  \cdot \d\nnu (t,x) \coloneqq \int \bigg( \int_{I}F(t,\gamma_t) \cdot  \d \boldsymbol{ \mathrm{D}\gamma}^{\mathrm{c}}(t) \,  + \\
               \int_{I} \int_{ [\gamma_{t-},\gamma_{t+}]} F(t,x)\cdot \frac{(\gamma_{t+}-\gamma_{t-})}{|\gamma_{t+}-\gamma_{t-}|^2}\d \Ha^1( x) \d |\boldsymbol{\mathrm{D}\gamma}^{\mathrm{j}}|(t) \bigg) \d \pi (\gamma), \quad \forall F\in C_c(I\times \R^n; \R^n).
		\end{multline}
    The key point here is that the flux measure $\nnu$ associated with the path measure $\pi$ arises as the superposition of the flux measures associated with individual paths, denoted by $\nnu_\gamma$, namely,
    \begin{equation}\label{eq:nnu_superposition_nnu_gamma}
        \nnu = \int \nnu_\gamma \d \pi (\gamma),
    \end{equation}
    in the sense of measures. 
    As shown below, all integrals above with respect to $\pi$ are well defined, and the integrability condition \eqref{eq:integrability_for_heuristic_computation} guarantees that they are finite.

    \begin{theorem}[From $\pi$ to $(\mu, \nnu)$]\label{thm:lifttoflux}
        Let $\pi \in \mathcal{P}(D (I;\R^n))$ be concentrated on $\BV (I;\R^n) \subset D(I;\R^n)  $ such that 
        \begin{equation}
         \int |\boldsymbol{ \mathrm{D}\gamma}|(I) \d \pi (\gamma) < + \infty,
        \end{equation}
        and set $\mu_t\coloneqq(e_t)_\#\pi$ for all $t\in I$, and define $\mu \in \M^+(I \times \R^n)$ via $\mu(\mathrm{d} t,\mathrm{d} x)\coloneqq \mu_t(\mathrm{d} x) \mathrm{d} t $. Then formula \eqref{eq:def_nu_inducedby_pi} defines a unique measure $\nnu \in \M(I\times \R^n;\R^n)$ such that the pair $(\mu,\nnu)$ solves \eqref{eq:CE}, and
        \begin{equation}\label{eq:lift_to_flux_inequality}
              |\D \mu| \leq  \pr^I_\#|\nnu| \leq \int |\boldsymbol{ \mathrm{D}\gamma}| \d \pi (\gamma).
        \end{equation}
        If $\pi$ is an optimal lift of $(\mu_t)$, then
        \begin{equation}\label{eq:lift_to_flux_equality}
              |\D \mu| =  \pr^I_\#|\nnu|=\int |\boldsymbol{ \mathrm{D}\gamma}| \d \pi (\gamma),
        \end{equation}
       and, in particular, $(\mu,\nnu)$ is a minimal solution with respect to $\prec$.
    \end{theorem}
    
    \begin{proof}
         We split the proof into a couple of intermediate steps. 

        \smallskip
        \noindent
        \textbf{Step 1} (Measurability)\textbf{.}
        We first prove the Borel measurability of the map 
        \begin{align}\label{eq:proof_measurability}
            \gamma \mapsto \int_{I}F(t,\gamma_t) \cdot  \d \boldsymbol{ \mathrm{D}\gamma}^{\mathrm{c}}(t)
             + \int_{I} \int_{ [\gamma_{t-},\gamma_{t+}]} F(t,x)\cdot \frac{(\gamma_{t+}-\gamma_{t-})}{|\gamma_{t+}-\gamma_{t-}|^2}\d \Ha^1( x) \d |\boldsymbol{\mathrm{D}\gamma}^{\mathrm{j}}|(t)
        \end{align}
        from $\big(\BV(I;\R^n), d_{Sk}\big) \to  {\R}$ for every bounded Borel function $F$. Note that for every fixed $\gamma \in \BV(I;\R^n)$, the right-hand side is finite. Indeed, by the duality inequality, the absolute value of the right-hand side is bounded by
        \begin{equation}\label{ineq:23/03-1}
			\|F\|_{\infty,*}|\boldsymbol{ \mathrm{D}\gamma}^{\mathrm{c}}|(I) +\|F\|_{\infty,*}\int_{I} \int_{ [\gamma_{t-},\gamma_{t+}]}\frac{1}{|\gamma_{t+}-\gamma_{t-}|}\d \Ha^1( x) \d |\boldsymbol{ \mathrm{D}\gamma}^{\mathrm{j}}|(t)
            = 
            \|F\|_{\infty,*}|\boldsymbol{ \mathrm{D}\gamma}|(I) < + \infty,
		\end{equation}
        where $\|F\|_{\infty,*} \coloneqq \sup_{(t,x) \in I \times \R^n} |F(t,x)|_{*}$ and $|\,{\cdot}\,|_*$ is the dual norm of $|\,{\cdot}\,|$.

        \smallskip
        \noindent
        For the measurability of the continuous part in \eqref{eq:proof_measurability}, by \cref{lma:evaluationBorel}, the evaluation map $e\colon (t,\gamma)\mapsto \gamma_t$ from $I\times D(I;\R^n) \to \R^n$ is Borel measurable. 
        Thus by \cref{lemma:Measurablity3}, 
        \begin{align}
            \gamma\mapsto \int F(t,\gamma_t)\cdot\d \boldsymbol{\D\gamma}^{\mathrm{c}}(t)=\int F\circ (\pr^I,e)(t,\gamma)\cdot\d \boldsymbol{\D\gamma}^{\mathrm{c}}(t),
        \end{align}
         is Borel measurable on $\BV(I;\R^n)$.

        \smallskip
        \noindent
         For the measurability of the jump part in \eqref{eq:proof_measurability}, by the standard reduction argument (see the proof of \cref{lemma:Measurablity3}), it suffices to show the Borel measurability of 
        \begin{align}
            \gamma \mapsto \int_{I} \int_{ [\gamma_{t-},\gamma_{t+}]} \mathds{1}_{B_1}(t) \mathds{1}_{B_2}(x)V\cdot \frac{(\gamma_{t+}-\gamma_{t-})}{|\gamma_{t+}-\gamma_{t-}|^2}\d \Ha^1( x) \d |\boldsymbol{\mathrm{D}\gamma}^{\mathrm{j}}|(t)
        \end{align}
       for all $V\in \R^n$, open sets $B_1\subset I$ and $B_2\subset \R^n$.
       \\
        Define the function $H: I\times D(I;\R^n) \to \R$ by
        \begin{align}
            H(t,\gamma)\coloneqq \int_{ [\gamma_{t-},\gamma_{t+}]} \mathds{1}_{B_2}(x) V\cdot \frac{(\gamma_{t+}-\gamma_{t-})}{|\gamma_{t+}-\gamma_{t-}|^2}\d \Ha^1( x),\quad t\in J_\gamma,
        \end{align}
        and $H(t,\gamma)\coloneqq0$ if $t\notin J_\gamma$. 
        By \cref{lemma:Measurablity3}, it is enough to show that $H$ is a bounded Borel function.
        Observe that
        \begin{align}
        H(t,\gamma)=
         V\cdot (G_1\circ H_1\cdot g_2\circ H_2)(t,\gamma),
        \end{align}
        where 
        \begin{align}
        H_1(t,\gamma)\coloneqq\gamma_{t+}-\gamma_{t-},& \quad  G_1\colon \R^n\ni x\mapsto \mathds{1}_{\{x\neq 0\}} \frac{x}{|x|}\in \R^n\\
        H_2(t,\gamma)\coloneqq (\gamma_{t-},\gamma_{t+}),&\quad g_2 \colon\R^{2n}\ni(x,y)\mapsto \int_{[0,1]}\mathds{1}_{B_2}((1-s)x+sy)\d s \in \R.
        \end{align}
        The maps $H_1$ and $H_2$ are Borel measurable by Lemma~\ref{lma:evaluationBorel}. Moreover, $g_2$ is lower semi-continuous. Indeed, for every $s\in[0,1]$, the function $(x,y)\mapsto \mathds{1}_{B_2}((1-s)x+sy)$ is lower semi-continuous as $B_2$ is open.
        Hence, by Fatou's lemma, $g_2$ is lower semi-continuous.
        All in all, the map $H$ is Borel measurable.

        \smallskip
        \noindent
        \textbf{Step 2} (Existence and uniqueness of $\nnu$ defined by \eqref{eq:def_nu_inducedby_pi})\textbf{.}
       We define the functional $\Lambda$ on $ C_b(I\times\R^n;\R^n)$ by
        \begin{equation}
            \Lambda (F) \coloneqq  \int \int_{I}F(t,\gamma_t) \cdot  \d \boldsymbol{ \mathrm{D}\gamma}^{\mathrm{c}}(t) \d \pi (\gamma) \,  + \int \int_{I} \int_{ [\gamma_{t-},\gamma_{t+}]} \hspace{-8pt} F(t,x)\cdot \frac{(\gamma_{t+}-\gamma_{t-})}{|\gamma_{t+}-\gamma_{t-}|^2}\d \Ha^1( x) \d |\boldsymbol{\mathrm{D}\gamma}^{\mathrm{j}}|(t) \d \pi (\gamma).
        \end{equation}
        It is clear that $\Lambda$ is linear.
        Moreover, with the Borel measurability of \eqref{eq:proof_measurability} obtained from the previous step, integrating over $\pi$, and using inequality \eqref{ineq:23/03-1}, we get
        \begin{equation}
            |\Lambda(F)|\leq \|F\|_{\infty,*} \int |\boldsymbol{ \mathrm{D}\gamma}|(I) \d\pi(\gamma)< + \infty,
        \end{equation}
		i.e., $\Lambda$ is a bounded functional on $(C_b,\|\cdot\|_{\infty,*})$. 
        In particular, it is a linear bounded functional on $(C_c,\|\cdot\|_{\infty,*})$. 
        Thus, by the (vector-valued) Riesz--Markov--Kakutani representation theorem, there exists a unique $\R^n$-valued Radon measure $\boldsymbol{\nu}$ on $ I \times \R^n$ such that
        \begin{equation}
        \int_{I\times\R^n} F\cdot \d\nnu =\Lambda (F)
        \qquad \forall\,F\in C_c(I\times\R^n;\R^n).
        \end{equation}  
        
        \smallskip
        \noindent
        \textbf{Step 3} (Variation estimates)\textbf{.}
        Fix an open interval $(a,b)\subseteq I$. 
        Using exactly the same arguments used for \eqref{ineq:23/03-1}, we have for any $F \in C_c((a,b)\times \R^n;\R^n)$ that
		\begin{align}\label{eq:proof:22:30}
			\bigg| \int_{(a,b)\times \R^n}  & F   \cdot \d\nnu  \bigg| \leq 
			\|F\|_{\infty,*} \int |\boldsymbol{ \mathrm{D}\gamma}|((a,b)) \d \pi (\gamma). 
		\end{align}
        In particular, taking the supremum over all such $F$ with $ \|F\|_{\infty,*} \leq 1$, we have by \eqref{eq:vardual}
        \begin{equation}
            |\nnu|((a,b)\times \R^n)\leq \int |\boldsymbol{ \mathrm{D}\gamma}|((a,b)) \d \pi (\gamma).
        \end{equation}
        In particular, $\nnu \in \M(I\times \R^n;\R^n)$. As a result of above, $\pr^I_\#|\nnu| \le \int |\boldsymbol{ \mathrm{D}\gamma}| \d \pi $ holds as measures.
       \\
       Next, it follows exactly from the computation at the beginning of \cref{subsec:lifts_to_flux measures} that $(\mu,\nnu)$ solves \eqref{eq:CE}, and thus, by \cref{cor:BVflux}, we have  
    		$
    		|\D\mu|\leq	\pr^I_\#|\nnu|.
    		$
        This completes the proof of the inequality \eqref{eq:lift_to_flux_inequality}. 
        Finally, if one starts from the optimal lift $\pi$ in the sense of Definition~\ref{def:optlift}, then inequalities in \eqref{eq:lift_to_flux_inequality} have to be equalities. This then implies, by \cref{cor:BVflux}~\ref{item:CE_minimal}, that $(\mu,\nnu)$ is a minimal solution with respect to $\prec$.
    \end{proof}

\subsection{From BV path measures to derivations: geodesic metric setting}\label{subsec:lifts_to_derivation}

In this section, we adapt the construction of the previous section to the metric setting, following a similar superposition strategy: we first associate a reference measure and a derivation with each individual BV path and then superpose them with respect to the path measure. The assumption that $(X,d)$ is geodesic is used to connect the jump points by a geodesic, so that no additional variation is introduced in the mass measure $|\cD|$ of the final measure-valued derivation $\cD$. Without this assumption, even starting from an optimal lift of $(\mu_t)$, the sharp identity $
|\D\mu|={\pr}^I_{\#}|\cD| $
cannot in general be attained; recall Example~\ref{ex:jump_between_two_points}.

\begin{theorem}[From $\pi$ to $(\mu,\cD)$]\label{thm:lifttoderivation}
    Let $(X,d)$ be a complete, separable, and geodesic metric space.
    Let $\pi \in \mathcal{P}(D (I;\X))$ be concentrated on $\BV (I;\X) \subset D(I;\X)$ such that 
    \begin{equation}\label{ineq:finiteEnergy}
     \int | \D\gamma|(I) \d \pi (\gamma) < + \infty,
    \end{equation}
    and set $\mu_t\coloneqq(e_t)_\#\pi$ for all $t\in I$, and define $\mu \in \M^+(I \times X)$ via $\mu(\mathrm{d} t,\mathrm{d} x)\coloneqq \mu_t(\mathrm{d} x) \mathrm{d} t $.
    Then there exists a finite-mass measure-valued derivation $\mathcal{D}\colon \Lipb(X) \to \M(I \times X; \R)$ with the mass measure $|\cD| \in \M^+(I \times X)$ such that the pair $(\mu,\cD)$ solves \eqref{eq:CE_D}, and 
        \begin{equation}\label{ineq:variation-mu-pi}
        |\D \mu| \leq  {\pr}^I_{\#}|\cD| \leq \int |\D \gamma| \d \pi (\gamma).
        \end{equation}
    	If $\pi$ is an optimal lift of $(\mu_t)$, then
    	 \begin{equation}
    		|\D \mu| = {\pr}^I_{\#}|\cD| = \int |\D \gamma| \d \pi (\gamma),
    	\end{equation}
        and, in particular, $(\mu,\cD)$ is a minimal solution in the sense of \cref{def:minimalsolution_M}.
\end{theorem}
\begin{remark}
    The construction underlying the theorem is based on a pathwise construction of $(V,\lambda)$ followed by superposition. More precisely, when $\pi$ is concentrated on a single curve, i.e., $\pi=\delta_\gamma$ for some $\gamma\in \BV(I;\X)$, so that $\mu_t = \delta_{\gamma_t}$, then the reference measure $\lambda_\gamma$ is given by
    \begin{equation}\label{def:lambda_gamma}
        \lambda_\gamma\coloneqq \tilde\gamma_\#|\D \gamma|^{\mathrm{c}}+\sum_{t\in J_\gamma}\delta_t\otimes \mathcal{H}^1|_{\Upsilon[\gamma_{t-},\gamma_{t+}]},
    \end{equation}
    where $\tilde \gamma(t)\coloneqq (t,\gamma_t)$ and $\Upsilon[\gamma_{t-},\gamma_{t+}]\colon [0,1]\to X$ is a constant speed geodesic from $\gamma_{t-}$ to $\gamma_{t+}$ for those $t\in J_\gamma$; and the associated $\lambda_\gamma$-derivation $V_\gamma$ has the expression (similar to \eqref{eq:def_nu_inducedby_pi} in $\R^n$):
    \begin{multline}
     \int_{I\times \X} g(t,x)V_\gamma f(t,x)\d\lambda_\gamma (t,x) \coloneqq \int_I g(t,\gamma_t) \d \D(f\circ \gamma)^{\mathrm{c}}(t)+\\
    +\int_I\int_{[0,1]} g(t,\Upsilon[\gamma_{t-},\gamma_{t+}](a)) \frac{(f\circ \Upsilon[\gamma_{t-},\gamma_{t+}])'(a)}{d(\gamma_{t-},\gamma_{t+})}  \d a \d |\D\gamma|^{\mathrm{j}}(t)
    \end{multline}
    for every $f\in\Lipb(X)$ and every $g\in L^1(\lambda_\gamma)$.
    For a general path measure $\pi$, we then construct $\lambda$ by superposition $$\lambda \coloneqq \int \lambda_{\gamma} \d \pi (\gamma)$$
     and the corresponding derivation $V$ as detailed in the proof below. The measure-valued derivation appearing in the theorem is finally obtained by $ \cD(f)\coloneqq Vf \lambda$ for all $f\in\Lipb(X)$.
\end{remark}

\begin{proof}
    We divide the proof into 4 steps. 
    \smallskip
    \\
	\textbf{Step 1} (Derivation for the single curve case)\textbf{.} 
	Assume first that $\pi=\delta_\gamma $ for some $\gamma\in\BV(I;X)$.
	For each jump point $t\in J_\gamma$, we denote by $\Upsilon[\gamma_{t-},\gamma_{t+}]\colon [0,1]\to (X,d)$ a constant speed geodesic from $\gamma_{t-}$ to $\gamma_{t+}$.  
	Define $\lambda_\gamma\in \M^+(I\times X)$ as in \eqref{def:lambda_gamma}.
\\
Consider for any $g\in C_b(I\times X)$ the expression:
\begin{equation}\label{eq:V_singlecurve}
	\int_I g(t,\gamma_t) \d \D(f\circ \gamma)^{\mathrm{c}}(t)+\int_I\left(\int_{[0,1]} \frac{g(t,\Upsilon[\gamma_{t-},\gamma_{t+}](a))}{d(\gamma_{t-},\gamma_{t+})} (f\circ \Upsilon[\gamma_{t-},\gamma_{t+}])'(a) \d a \right)\d |\D\gamma|^{\mathrm{j}}(t).
\end{equation}
Since $f$ is Lipschitz, we have
    \begin{align}
    	|\D (f\circ \gamma)^{\mathrm{c}}|&\leq (\lip_{\rm a} f\circ\gamma ) |\D \gamma|^{\mathrm{c}}\\
    	 |(f\circ \Upsilon[\gamma_{t-},\gamma_{t+}])'(a)|&\leq \lip_{\rm a} f(\Upsilon[\gamma_{t-},\gamma_{t+}](a))d(\gamma_{t-},\gamma_{t+}) ,\quad \text{$\mathcal{L}^1$-a.e. $a\in[0,1]$}.
    \end{align}
Then, together with \eqref{def:lambda_gamma}, above estimates yield
\begin{align}
	|\eqref{eq:V_singlecurve}|&\leq \int_I |g(t,\gamma_t) | \lip_{\rm a} f(\gamma_t) \d |\D\gamma|^{\mathrm{c}}(t)+\int_I \int^1_0 |g(t,\cdot)\lip_{\rm a} f|\left(\Upsilon[\gamma_{t-},\gamma_{t+}](a)\right)\d a\d |\D \gamma|^{\mathrm{j}}(t) \\
	&\leq \int|g\cdot \lip_{\rm a}f|\circ \tilde\gamma\d |\D\gamma|^{\mathrm{c}}
    +\sum_{t\in J_\gamma}\int|g\cdot \lip_{\rm a}f| \d \left(\delta_t\otimes\mathcal{H}^1|_{\Upsilon[\gamma_{t-},\gamma_{t+}]}\right)\\
    &=\int |g\cdot \lip_{\rm a}f|\d\lambda_\gamma
    \label{ineq:Vbound_singlecurve},
\end{align}
    Therefore, the expression \eqref{eq:V_singlecurve} extends to a bounded linear functional on $L^1(\lambda_\gamma)$ and by duality there exists a unique $L^\infty(\lambda_\gamma)$-function, denoted by $V_\gamma f$ such that
    \begin{equation}\label{eq:9/mar-1}
    		\int_{I\times \X} g(t,x)V_\gamma f(t,x)\d\lambda_\gamma (t,x)=\eqref{eq:V_singlecurve},\quad \forall g\in L^1(\lambda_\gamma).
    \end{equation}
	\textbf{Step 2} (Derivation for general path measures by superposition)\textbf{.}	
	Let $\pi \in \mathcal{P}(D (I;X))$.
    By \cite[Theorem 6.9.1]{Bogachev_book07}, there exists a Souslin measurable geodesic selection $\Upsilon\colon X^2\to \Geo(X)$ such that for every $x,y\in X$, $\Upsilon[x,y]$ is a constant speed geodesic on $X$ (we refer the reader to \cite{Bogachev_book07} for the relevant background).
\\    
For any $\gamma\in \spt(\pi)\cap\BV(I;X)$, let $\lambda_\gamma$ and $V_\gamma$ be the measure and the map explicitly constructed in the previous step.
Consider $\lambda \coloneqq \int \lambda_\gamma \d \pi(\gamma)$, which is a well-posed finite Borel measure on $I\times X$ by \cref{lemma:measurabilityGeo}.
Moreover, by Lemma~\ref{lemma:measurabilityGeo}, the function
    \begin{equation}
        D(I;X)\ni\gamma\mapsto \int_{I\times \X} g(t,x)V_\gamma f(t,x)\d\lambda_\gamma (t,x)
    \end{equation}
    is $\pi$-measurable.
	Then integrating the inequality \eqref{ineq:Vbound_singlecurve} over $\pi$ obtains
	\begin{align}
			\left|\iint_{I\times \X} gV_\gamma f\d\lambda_\gamma \d \pi (\gamma)\right|&\leq\int \int |g\cdot \lip_{\rm a}f|\d\lambda_\gamma\d\pi(\gamma)=\int |g(t,x)|\lip_{\rm a} f(x)\d\lambda(t,x)
	\end{align}
	By the same duality argument, we obtain a map $V: \Lipb (X) \to L^\infty(\lambda)$ such that
	\begin{equation}
	\int_{I\times \X} g(t,x)V  f(t,x)\d\lambda (t,x) =	\iint_{I\times \X} g(t,x)V_\gamma f(t,x)\d\lambda_\gamma (t,x) \d \pi (\gamma) .
	\end{equation}
	The Leibniz rule for $V$ follows from the Leibniz rule for classic/distribution derivatives, which is clear from the expression \eqref{eq:V_singlecurve}. 
    Thus, $V$ is a $\lambda$-derivation.
    \smallskip
    \\	
	\textbf{Step 3} (Verification of the CE)\textbf{.}
    For any $\xi\in C^1_c(I)$ and $\gamma\in \spt(\pi)\cap \BV(I;X)$, expanding the following identity
    \begin{equation}
    \D(  f\circ\gamma\cdot\xi ) (I)=0
    \end{equation}
    with the Leibniz rule and the chain rule 
    \begin{equation}
		\D (f\circ \gamma)=\D (f\circ \gamma)^{\mathrm{c}}+\frac{f(\gamma_{t+})-f(\gamma_{t-})}{d(\gamma_{t-},\gamma_{t+})}|\D \gamma|^{\mathrm{j}}(\d t)
	\end{equation}
gives that
\begin{align}
	-\int_I \xi'(t)f(\gamma_t)\d t&=\int_I \xi \d \D(f\circ \gamma)^{\mathrm{c}}+\int_I\xi(t) \frac{f(\gamma_{t+})-f(\gamma_{t-})}{d(\gamma_{t-},\gamma_{t+})}\d|\D \gamma|^{\mathrm{j}}(t)\\
	&=	\int_I \xi \d \D(f\circ \gamma)^{\mathrm{c}}+\int_I\int_{[0,1]} \frac{\xi(t)(f\circ \Upsilon[\gamma_{t-},\gamma_{t+}])'(a)}{d(\gamma_{t-},\gamma_{t+})} \d a \d |\D\gamma|^{\mathrm{j}}(t)\\
	&=	\int_{I\times \X} \xi(t)V_\gamma f(t,x)\d\lambda_\gamma (t,x).
\end{align}
Integrating the above over $\pi$ yields
\begin{align}
	0=&\iint \xi'(t)f(\gamma_t)\d t\d\pi(\gamma)+\iint_{I\times \X} \xi V_\gamma f\d\lambda_\gamma \d\pi(\gamma)\\
    =&\int_{I\times X}\xi'(t)f(x)\d \mu(t,x)+\int \xi Vf(t,x)\d\lambda(t,x).
	\end{align}
    Therefore, $(\mu,\lambda,V)$ solves \eqref{eq:CE_derivation}.
    \smallskip
    \\	
	\textbf{Step 4} (Variation estimates)\textbf{.}
    Define for each $f\in \Lipb(X)$, $\cD(f)\coloneqq Vf\cdot \lambda$.
    Then $\cD$ is a measure-valued derivation, and $(\mu,\cD)$ solves \eqref{eq:CE_Mderivation}.
    From Step 2, we know for any $f\in\Lipb(X)$
	\begin{align}
		&\left|\int_{I\times X}g(t,x)\d \cD(f)(t,x)\right| \leq\int |g(t,x)|\lip_{\rm a} f(x)\d\lambda(t,x),\quad \forall g\in L^1(\lambda),
	\end{align}
which shows that $|\cD(f)|\leq \lip_{\rm a} f\cdot\lambda$.
\\
Notice by the construction of $\lambda$, 
\begin{equation}\label{eq:projection_lambda}
    {\rm Pr}_\#^I \lambda=\int |\D \gamma|\d\pi(\gamma).
\end{equation}
By the definition of mass measure, $|\cD|\leq \lambda$.
Then applying \cref{cor:derivation_CE}~\ref{item:Varineq_Mderivation}, we get
\begin{align}
		|\D\mu|\leq {\rm Pr}_\#^I |\cD|\leq {\rm Pr}_\#^I \lambda= \int |\D \gamma|\d\pi(\gamma).
	\end{align}
When $\pi$ is optimal, all inequalities must be equalities.
In particular, $|\cD|=\lambda$, $|\cD|(I\times X)=|\D\mu|(I)$, and by \cref{prop:Opt->Min} $(\mu,\cD)$ is a minimal solution.
\end{proof}

\subsubsection*{A Benamou--Brenier formula via CEs with derivations}
As an immediate application of Corollary~\ref{cor:derivation_CE} and Theorem~\ref{thm:lifttoderivation}, we obtain a characterization of the $1$-Wasserstein distance via CEs with derivations, analogous to the Benamou--Brenier formula \cite{Benamou--Brenier}; see also \cite{DNS2009}.

\begin{corollary}[A Benamou--Brenier formula via $\mathrm{CE_\cD}$]\label{thm:Benmou-BrenierGeod}
	Let $(X,d)$ be a complete separable geodesic space.
	For any $\textup{m}_0,\textup{m}_1\in \P(X)$, we have
	\begin{align}
		W_1(\textup{m}_0,\textup{m}_1)&=\inf \Big\{|\mathcal{D}|((0,1)\times X): \quad \big((\mu_t),\cD\big) \in \mathcal{CE}_\cD(\textup{m}_0\to\textup{m}_1) \Big\},
	\end{align}
    where $\mathcal{CE}_\cD (\textup{m}_0 \to \textup{m}_1)$ denotes the set of all pairs $\big((\mu_t),\cD\big)$, consisting of a Borel family $(\mu_t)_{t\in(0,1)} \subset \P(X)$ and a measure-valued derivation $\mathcal{D}\colon \Lipb(X) \to \M(I \times X; \R)$, such that
    \begin{itemize}
        \item[\textbf{-}] $(\mu,\mathcal{D})$ solves \eqref{eq:CE_D}, where $\mu(\mathrm{d}t,\mathrm{d}x)\coloneqq \mu_t (\mathrm{d}x) \d t $; 
        \item[\textbf{-}]  $W_1(\mu_t,\textup{m}_0) \to 0$ as $t \downarrow 0$, and $W_1(\mu_t,\textup{m}_1) \to 0$ as $t \uparrow 1$;
    \end{itemize}
  and $|\cD|$ is the mass measure of $\cD$ given by \cref{lemma:finitemassD}\,\ref{item:massM1} with the convention that $|\cD|((0,1)\times X)=\infty$ if $\cD$ does not have finite mass.  
\end{corollary}

\begin{proof}
Recall from Corollary~\ref{cor:derivation_CE} that for any family $(\mu_t) \subset \P (X)$ admitting a finite-mass derivation solving the continuity equation is necessarily BV.
In particular, if $W_1(\textup{m}_0,\textup{m}_1)= + \infty$, then  the statement follows from the convention on $|\cD|$.
\\
Therefore, we assume that $W_1(\textup{m}_0,\textup{m}_1)<+\infty$.
The inequality ``$\leq$'' is an immediate consequence of \eqref{ineq:Varineq_D}. 
The equality is attained by choosing the derivation constructed as in Theorem~\ref{thm:lifttoderivation} from an optimal lift $\pi$ of a BV-geodesic $(\mu_t)_{t\in [0,1]}$ i.e. a BV-curve connecting $\mu_0$ and $\mu_1$ with $|\D \mu|([0,1])=W_1(\mu_0,\mu_1)$.
\end{proof}

\appendix
\section{Extended Wasserstein spaces}\label{App:Ext_Wass_Space}
We provide a proof of the following fundamental property of the extended Wasserstein space.
We expect this to be known, but did not find it in the literature.
Although the result is used in the present paper, it is also of independent interest. 

\begin{proposition}\label{prop:wassersteinspacecomplete}
    Let $(X,d)$ be a complete and separable metric space.
    Then $(\P(X),W_p)$ is a complete extended metric space for every $p\in[1,\infty)$.
    Moreover, each of its metric components is separable.
    \end{proposition}

\begin{lemma}\label{lma:cauchyistight}
    Let $(X,d)$ be a complete and separable metric space, and let $(\mu_n)_{n \in \mathbb{N}}$ be a Cauchy sequence in $(\P(X),W_p)$, for $p \geq 1$.
    Then $\{\mu_n\}_{n\in\N}$ is tight.
\end{lemma}
\begin{proof}
    Follows by the same argument as in \cite[Lemma 6.14]{Villani_oldnew}
\end{proof}
\begin{lemma}\label{lma:Ddense}
    Let $(X,d)$ be a complete and separable metric space, and $D=\{x_i\}_{i\in\N} \subset X$ a dense subset.
    Denote by $\mathcal{S}\subset \P(X)$ the set 
    \begin{align}
        \mathcal{S}\coloneqq \left\{\sum_{i\in\N}\lambda_i\delta_{x_i}\in \P(X): \lambda_i\in \mathbb{Q}\right\}.
    \end{align}
    Then $\mathcal{S}$ is dense in $(\P(X),W_p)$ for all $p\in[1,\infty)$.
\end{lemma}
\begin{proof}
    Let $\mu\in \P(X)$.
    Given $\varepsilon>0$, decompose the space into Borel sets $X_i$ so that $d(x,x_i)<\varepsilon$ for all $x\in X_i$.
    Define $\nu\coloneqq \sum_{i\in \N}\lambda_i\delta_{x_i}$, where $\lambda_i=\mu(X_i)$.
    Then $W_1(\mu,\nu)<\varepsilon$.
\\    
    We are left to approximate $\nu$ by $\mathcal{S}$. 
    We may assume that $\lambda_i>0$ for all $i\in \N$.
    Let us define $\eta=\sum_{i\in\N}\tilde\lambda_i\delta_{x_i}\in \mathcal{S}$ inductively.
    Denote by $\hat\lambda_1=\lambda_1$. 
    Let $\tilde\lambda_1\in \mathbb{Q}$, $0\le \tilde\lambda_1\le\hat\lambda_1$, be such that $(\hat\lambda_1-\tilde\lambda_1)d^p(x_1,x_2)<\varepsilon$.
    Let now $\hat\lambda_{i+1}=\hat \lambda_i-\tilde\lambda_i+\lambda_{i+1}$, and choose $\tilde\lambda_{i+1}\in \mathbb Q$, $0\le\tilde\lambda_{i+1}\le\hat\lambda_{i+1}$ such that $\hat\lambda_{i+1}-\tilde\lambda_{i+1},(\hat\lambda_{i+1}-\tilde\lambda_{i+1})d^p(x_{i+1},x_{i+2})<2^{-i}\varepsilon$.
     We also require that $\hat \lambda_{i+1}-\tilde \lambda_{i+1}\le \lambda_{i+1}$.
    Then
    \begin{align}
        1=\sum_{i\in\N} \lambda_i=\hat\lambda_1+\lim_{n\to\infty}\left[\hat\lambda_{n+1}-\hat\lambda_1+\sum_{i=1}^n\tilde\lambda_i\right]=\sum_{i\in \N} \tilde \lambda_i.
    \end{align}
    Moreover, by considering a coupling 
    \begin{align}
        \sigma\coloneqq\sum_{i\in\N} (\hat\lambda_i-\tilde\lambda_i)\delta_{(x_i,x_{i+1})}+\sum_{i\in\N}(\lambda_i-(\hat\lambda_i-\tilde\lambda_i))\delta_{(x_i,x_i)}
    \end{align}
    we see that
    \begin{align}
        W_p^p(\nu,\eta)\le \sum_{i\in\N}(\hat\lambda_i-\tilde\lambda_i)d^p(x_i,x_{i+1})\le 2\varepsilon.
    \end{align}
    Thus, $\mathcal S$ is dense in $\P(X)$.
\end{proof}

\begin{proof}[Proof of \cref{prop:wassersteinspacecomplete}]
    Let $\{\mu^n\}_{n \in \mathbb{N}}$ be a Cauchy sequence.
    By \cref{lma:cauchyistight}, the sequence is tight.
    Since $(X,d)$ is complete and separable, by Prokhorov's theorem, $\{\mu^n\}$ converges, up to a subsequence, narrowly (i.e., with respect to $C_b(X)$) to some $\mu^\infty\in \P(X)$.
    By the lower semi-continuity of the Wasserstein distance, we have for all $n$ that
    \begin{align}
        W_p(\mu^\infty,\mu^n)\le \liminf_{k\to\infty} W_p(\mu^k,\mu^n).
    \end{align}
    Therefore, since $\{\mu^n\}$ is a Cauchy sequence, it converges to $\mu^\infty$.
\\
    Next, we show the separability of the components of $\P(X)$.
    By \cref{lma:Ddense} it suffices to show the separability of the metric components of $\mathcal S$.
    Fix $\mu^*=\sum_{i\in \N}\mu^*_i\delta_{x_i}\in \mathcal{S}$.
    \\
    We will show that 
    \begin{align}
        \mathcal{F}\coloneqq \left\{\nu=\sum_i \nu_i\delta_{x_i}\in \mathcal S: \nu_i=\mu^*_i \mathrm{\ for\ all\ but\ finitely\ many\ }i\in\N \right\}.
    \end{align}
    is dense in the component of $\mu^*$.
    This set is countable: one first chooses the finite set of indices where $\nu_i\neq\mu_i^*$ and then chooses finitely many rational numbers.
    \\
    Let $\mu=\sum_{i\in\N}\mu_i\delta_{x_i}\in \mathcal{S}$ with $W_p(\mu,\mu^*)<\infty$ and let $\sigma=\sum_{i,j}\sigma_{ij}\delta_{(x_i,x_j)}$ be a coupling of $\mu$ and $\mu^*$ such that
\begin{equation}\label{eq:rationalchoose}
    \sum_{i,j\in \N} \sigma_{ij}d^p(x_i,x_j)<\infty,\quad \sigma_{ij}\in\mathbb{Q}.
\end{equation}
    Furthermore, let $N\in \N$ be such that
    \[\sum_{(i,j)\notin [0,N]^2}\sigma_{ij}d^p(x_i,x_j)<\varepsilon.\]
Define $\nu\coloneqq\sum_{j\in\N}\nu_j\delta_{x_j}\in \mathcal{F}$ (the rationality of $\nu_j$ follows from \eqref{eq:rationalchoose}), where
\[\nu_j\coloneqq 
\begin{cases}
    \mu^*_j= \sum_{i\in\N}\sigma_{ij},&  j>N\\
    \mu_j-\sum_{k>N}\sigma_{jk}+\sum_{i>N}\sigma_{ij},& j\leq N
\end{cases}.
\]
    We claim that $W_p^p(\mu,\nu)\le \varepsilon$.
    Consider $\tilde \sigma=\sum_{i,j\in\N}\tilde\sigma_{ij}\delta_{(x_i,x_j)}\in\P(X\times X)$ by
\[\tilde \sigma_{ij}\coloneqq
\begin{cases}
   \mu_i-\sum_{k>N}\sigma_{ik},& j=i\leq N\\
   \sigma_{ij},& i\in\N, j>N \text{ or } i>N,j\leq N\\
   0,& \text{else}
\end{cases}
\]
which is a coupling between $\mu$ and $\nu$.
    Then
    \begin{align}
        W_p^p(\mu,\nu)\le \sum_{i,j\in\N}\tilde\sigma_{ij}d^p(x_i,x_j)&=\sum_{i=1}^\infty\sum_{j={N+1}}^\infty\sigma_{ij}d^p(x_i,x_j)+\sum_{i={N+1}}^\infty\sum_{j=1}^N\sigma_{ij}d^p(x_i,x_j)
        \\&=\sum_{(i,j)\notin [0,N]^2}\sigma_{ij}d^p(x_i,x_j)\le\varepsilon.\tag*{\qedhere}
    \end{align}
\end{proof}

\section{Pointwise variation}\label{App:pointwise_variation}
    Here we collect some properties and results on pointwise variation which are used in the proof of \cref{thm:optimal_lift}. 
    
    Given a metric space $(X,d)$ and an arbitrary time interval $ I \subset \R$, we denote the set of curves of bounded \emph{pointwise} variation by 
    \begin{equation}\label{eq:def_pBV}
       pBV (I;X) \coloneqq \big\{ \gamma \in X^I : \, \Var(\gamma;I) < + \infty  \big\},
   \end{equation}
   where $\Var(\gamma;I)$ is defined in \eqref{def:pointwise_variation}. 

\begin{lemma}\label{lemma:lsc_pointwise_Var}
    The map $\gamma \mapsto \Var (\gamma;I)$ from $X^I \to [0,+\infty]$ is sequentially lower semi-continuous with respect to pointwise convergence. 
\end{lemma}
\begin{proof}
    Let $(\gamma^n)_{n \in \N} \subset X^I$ be such that $\gamma^n \to \gamma $ pointwise on $I$ for some $\gamma \in X^I$. Let $t_0 <\cdots < t_{k + 1}$ with $ \{ t_i\}_{ 0 \leq i \leq k + 1}  \subset I$ be arbitrary. Then 
    \begin{equation}
        \sum_{i=0}^k d(\gamma_{t_i},\gamma_{t_{i + 1}}) = \liminf_{n \to \infty} \sum_{i=0}^k d(\gamma^n_{t_i},\gamma^n_{t_{i + 1}}) \leq  \liminf_{n \to \infty} \Var (\gamma^n).
    \end{equation}
    Taking the supremum over all such partitions gives the result. 
\end{proof}

\begin{lemma}\label{lemma:Var_dense_subset}
	Let $(X,d)$ be a metric space, $I = (0,T ) \subset \mathbb R$, and $\widetilde I\subset I$ be dense. 
    For any $\gamma\in D(I;X)$, we have 
    \begin{equation}\label{eq:equality_Var_dense}
    \Var(\gamma;I)=\Var(\gamma;\widetilde I),
    \end{equation}
    where both sides may be infinite.
\end{lemma}
Note that if $I=[0,T]$, the equality \eqref{eq:equality_Var_dense} does not hold in general. It holds, however, if either $T\in \widetilde I$ or $\gamma$ is left-continuous at $T$.

\begin{proof}
	Since $\widetilde{I}\subset I$, the inequality
	$\Var(\gamma;\widetilde{I})\le \Var(\gamma;I)$
	is immediate. To prove the reverse inequality, fix a partition $\{t_i\}_{0\le i\le k+1}\subset I$ with $ t_0<\cdots<t_{k+1}$. 
	Since $\widetilde{I}$ is dense in $I$ and $\gamma$ is  c\`adl\`ag, for each $i=0,\dots,k+1$, there exists a sequence $(t_i^n)_n\subset \widetilde{I}$ such that
	$t_i^n \to t_i$ and $
	\gamma_{t_i^n}\to \gamma_{t_i}$.
	For $n$ sufficiently large, we  have
	$t_0^n<\cdots<t_{k+1}^n$,
	since the original partition is strictly increasing.
	Therefore, 
	\begin{equation}
	\sum_{i=0}^k d(\gamma_{t_i},\gamma_{t_{i+1}})
	=\lim_{n\to\infty}\sum_{i=0}^k d(\gamma_{t_i^n},\gamma_{t_{i+1}^n})
	\le \Var(\gamma;\widetilde{I}).
    \end{equation}
	Taking the supremum over all partitions of $I$ yields
	$\Var(\gamma;I)\le \Var(\gamma;\widetilde{I})$. 
\end{proof}

\begin{lemma}\label{lemma:cadlag_extension_dense_BV}
	Let $(X,d)$ be a complete metric space, $I=(0,T)\subset\mathbb R$, and $\widetilde I\subset I$ be dense.
	For every $\widetilde\gamma\in pBV(\widetilde I;X)$, there exists a unique curve $\gamma\in D(I;X)$
	such that
	\begin{equation}\label{eq:reconstruction_right_limit}
		\gamma_t=\lim_{\widetilde t\downarrow t,\ \widetilde t\in\widetilde I}\widetilde\gamma_{\widetilde t},
		\qquad t\in I.
	\end{equation}
	Moreover, we have 
	\begin{equation}\label{eq:Var_extension_dense}
		\Var(\gamma;I) \leq \Var(\widetilde\gamma;\widetilde I).
	\end{equation}
\end{lemma}

\begin{proof}
		We define the non-decreasing $V\colon (0,\infty)\to [0,+\infty)$ by $
		V(t)\coloneqq \Var(\widetilde\gamma;\widetilde I\cap (0,t])$. It is clear that for every $s<t$ with $s,t\in \widetilde I$, we have 
		\begin{equation}\label{eq:dense_BV_control}
			d(\widetilde\gamma_s,\widetilde\gamma_t)\le V(t)-V(s).
		\end{equation}
		We first verify that the limit in \eqref{eq:reconstruction_right_limit} exists.
		Let $(t_n)_n\subset \widetilde I$ be any sequence such that $t_n\downarrow t$.
		Then for $m<n$, by \eqref{eq:dense_BV_control}, we have $
		d(\widetilde\gamma_{t_n},\widetilde\gamma_{t_m})
		\le V(t_m)-V(t_n) $.
		Since $(V(t_n))_n$ is monotone and bounded, it is Cauchy. Thus $(\widetilde\gamma_{t_n})_n$ is a Cauchy sequence in $X$, and since $X$ is complete, it converges. It can be easily shown that the limit does not depend on the chosen sequence. Thus \eqref{eq:reconstruction_right_limit} defines a map $
		\gamma:I\to X$.
		\\
        Consider $V_+$ the c\`adl\`ag representative of $V$. 
        Then for any $s<t\in I$, by choosing arbitrary sequences $\widetilde I\ni s_n\searrow s$ and $\widetilde I\ni t_n\searrow t$, we see that
        \begin{equation}\label{eq:dense_BV_control_2}
            d(\gamma_s,\gamma_t)=\lim_{n\to \infty} d(\widetilde\gamma_{s_n},\widetilde\gamma_{t_n})\leq \lim_{n\to\infty}V(t_n)-V(s_n)=V_+(t)-V_+(s).
        \end{equation}
       This in particular implies that $\gamma$ is right-continuous.
       The existence of left limits follows from the same argument for $\widetilde\gamma$, with now the use of \eqref{eq:dense_BV_control_2}.
       Thus $\gamma$ is c\`adl\`ag.
		\\
        Finally, we show the inequality \eqref{eq:Var_extension_dense}. 
        Let $ \{ t_i\}_{ 0 \leq i \leq k + 1}  \subset I$ with  $t_0 <\cdots < t_{k + 1}$ be an arbitrary partition. 
        By \eqref{eq:dense_BV_control_2}
        \[
        \sum_{i=0}^k d(\gamma_{t_i},\gamma_{t_{i+1}})\leq \sum_{i=0}^k V_+(t_{i+1})-V_+(t_{i}) = V_+(t_{k+1})-V_+(t_{0}) \leq  V_+(T)=\Var(\widetilde\gamma;\widetilde I)
        \]
        Taking the supremum over all partitions of $I$ completes the proof. 
\end{proof}

\section{Measurability of maps}\label{App:Measurability}

\begin{lemma}\label{lma:evaluationBorel}
    Let $(X,d)$ be a metric space, and $I=(0,T)\subset\mathbb R$. The maps
    $$
    (t,\gamma)\mapsto\gamma_{t+} = \gamma_t, \qquad (t,\gamma)\mapsto\gamma_{t-}
    $$
    from $I\times D(I;X)\to X$ are Borel measurable. 
\end{lemma}
\begin{proof}
    For each $t\in I$, the evaluation map $e_t\colon\gamma\mapsto\gamma_t $ is Borel (e.g. by the same argument as in \cite[Proposition 2.15]{AbediLiSchultz2024}).
    For the joint measurability, we argue as follows. For simplicity, assume $I=(0,1)$, and set $I_j \coloneqq (\frac{1}{j},1-\frac{1}{j})$ for $j \geq 3$. For each $n\in\N$, let $q_{n,k}\coloneqq k2^{-n}, \, k \in \{ 0,\cdots, 2^n\}$. For $t \in I_j$, define 
    \begin{equation}
    r_n(t)\coloneqq\min\{q_{n,k}:q_{n,k}\geq t\},\qquad
    l_n(t)\coloneqq\max\{q_{n,k}:q_{n,k}<t\}.
\end{equation}
If $n$ is large enough that $2^{-n} \leq 1/j$, then $r_n(t),l_n(t) \in I$ for every $t \in I_j$.
For any open $U\subset X$, $\gamma_{r_n(t)}\in U$ if and only if $t\in (q_{n,k-1},q_{n,k}]$ for some $k$ and $\gamma_{q_{n,k}}\in U$, or in other words
\[
\Big\{(t,\gamma) \in I_j \times D(I ; X):\, \gamma_{r_n(t)}\in U\Big\}=\bigcup_{k=1}^{2^n-1}\left( I_j \cap (q_{n,k-1},q_{n,k}]\right)\times e_{q_{n,k}}^{-1}(U).
\]
This implies that $(t,\gamma)\mapsto \gamma_{r_n(t)}$ is Borel from $I_j \times D(I \times X) \to X$ for sufficiently large $n$, namely $2^{-n} \leq 1/j$, and a similar argument applies also to $(t,\gamma)\mapsto \gamma_{l_n(t)}$.
Then the right continuity and the existence of left limits yield
\begin{equation}
    \gamma_t=\lim_{n\to\infty}\gamma_{r_n(t)},\qquad
    \gamma_{t-}=\lim_{n\to\infty}\gamma_{l_n(t)}.
\end{equation}
Thus, both maps are Borel measurable on $I_j \times D(I;X)$ as pointwise limits of Borel maps. Since $I = \cup_{j\geq 3}I_j$, they are Borel measurable on $I \times D(I;X)$.
\end{proof}

\begin{lemma}\label{lemma:Measurablity3}
For any bounded Borel map $H \colon I\times D(I;\R^n)\to \R^n$ and $h\colon I\times D(I;\R^n)\to \R$, the functions
\begin{align}
     \gamma\mapsto \int_I H(t,\gamma)\cdot\d \boldsymbol{\D\gamma}^{\mathrm{c}}(t),\quad \gamma\mapsto \int_I h(t,\gamma)\d|\boldsymbol{\D\gamma}^{\mathrm{j}}|(t)
\end{align}
are Borel measurable on $\BV(I;\R^n)\subset D(I;\R^n)$.
\end{lemma}

\begin{proof}
	We first show that, for every $B\in\mathcal B(I\times D(I;\R^n))$ and $\bm V\in\R^n$, the maps
   \[
   \gamma\mapsto\int_I \mathds{1}_B(t,\gamma)\bm V\cdot\d\boldsymbol{\D\gamma}(t),
   \qquad
   \gamma\mapsto\int_I \mathds{1}_B(t,\gamma)\d|\boldsymbol{\D\gamma}|(t)
   \]
	are Borel measurable on $\BV(I;\R^n)$.
	\\
	For $B=B_1\times B_2$, where $B_1=(a,b)$ with $0<a<b<T$ and $B_2\in\mathcal B(D(I;\R^n))$, we have that 
	\begin{equation}\label{eq:map-to-be-Borel1}
			\gamma\mapsto \int_I \mathds{1}_B(t,\gamma)\bm V\cdot
		\d\boldsymbol{\D\gamma}(t)=\mathds{1}_{B_2}(\gamma)\bm V\cdot(\gamma_{b-}-\gamma_a)
	\end{equation}
is Borel measurable by \cref{lma:evaluationBorel}; and that
\begin{equation}\label{eq:map-to-be-Borel2}
		\gamma\mapsto\int_I \mathds{1}_B(t,\gamma)
		\d|\boldsymbol{\D\gamma}|(t)=\mathds{1}_{B_2}(\gamma)
		\Var(\gamma;(a,b)).
	\end{equation}
 is Borel measurable by \cite[Lemma 2.13]{AbediLiSchultz2024}.
	\\
	For each fixed $V$, by linearity and the dominated convergence theorem (with the fact that $|\boldsymbol{\D\gamma}|(I)<\infty$ for all $\gamma$ considered), the collection of sets $B$, for which maps in \eqref{eq:map-to-be-Borel1} and \eqref{eq:map-to-be-Borel2} are Borel measurable, is closed under complementation of subsets in supersets, and under countable increasing unions.
	The rectangles considered above form a $\pi$-system. 
	Thus, Dynkin's $\pi$--$\lambda$ theorem proves the claim for every $B\in\mathcal B(I\times D(I;\R^n))$.
	By approximation with uniformly bounded simple functions
	and the dominated convergence theorem, it follows that
	\[
	\gamma\mapsto\int_I H(t,\gamma)\cdot\d\boldsymbol{\D\gamma}(t),
	\qquad
	\gamma\mapsto\int_I h(t,\gamma)\d|\boldsymbol{\D\gamma}|(t)
	\]
	are Borel measurable for all bounded Borel maps $H$ and $h$ as in the statement.
	Now put
	\[
	\chi(t,\gamma)\coloneqq\mathds{1}_{\{\gamma_t\neq\gamma_{t-}\}},
	\]
	which is Borel measurable by \cref{lma:evaluationBorel}.
	For each $\gamma\in\BV(I;\R^n)$, we have $\chi(\cdot,\gamma)=\mathds{1}_{J_\gamma}$ and
	\[
	\boldsymbol{\D\gamma}^{\mathrm c}=\bigl(1-\chi(\cdot,\gamma)\bigr)\boldsymbol{\D\gamma},
	\qquad
	|\boldsymbol{\D\gamma}^{\mathrm j}|=\chi(\cdot,\gamma)|\boldsymbol{\D\gamma}|.
	\]
	Applying the preceding result to the bounded Borel maps $(1-\chi)H$ and $\chi h$ proves the assertion.
\end{proof}

 \begin{lemma}\label{lemma:measurability_variation}
		Let $(\X,d)$ be a complete and separable metric space.
        For any Lipschitz function $f$ on $X$ and any bounded Borel function $h\colon I\times D(I;X) \to \R$, the functions
        \begin{align}
            \gamma\mapsto \int_I h(t,\gamma)\d \D (f\circ \gamma)^c(t),\quad \gamma\mapsto \int_I h(t,\gamma)\d|{\D\gamma}|^{\rm j}(t),
        \end{align}
       are Borel measurable on $\BV(I;X)\subset D(I;X)$.
	\end{lemma}
    \begin{proof}
        We omit the proof, as it follows similarly to the previous lemmas in $\R^n$.
    \end{proof}

    \begin{lemma}\label{lemma:measurabilityGeo}
        Under the assumption of \cref{thm:lifttoderivation}, the integral $\int \lambda_\gamma\d\pi(\gamma)$ is a well-defined finite Borel measure on $I\times X$, which we denote by $\lambda$.
        \\
        Moreover, for any Borel $\lambda$-integrable function $g$, 
        function
    \begin{equation}
        D(I;X)\ni\gamma\mapsto \int_{I\times \X} g(t,x)V_\gamma f(t,x)\d\lambda_\gamma (t,x)
    \end{equation}
    is $\pi$-measurable/integrable. 
    Here $\lambda_\gamma$ and $V_\gamma$ are the objects explicitly constructed in the proof of \cref{thm:lifttoderivation}.
    \end{lemma}
\begin{proof}
First of all, notice by \cref{lma:evaluationBorel} and \cref{lemma:measurability_variation} that
\begin{equation}
    \gamma\mapsto \int_I g(t,\gamma_t) \d \D(f\circ \gamma)^{\mathrm{c}}(t)
\end{equation}
is Borel and hence $\pi$-measurable for any bounded Borel function $g$.
Thus, it suffices to consider the $\pi$-measurability of
\begin{equation}\label{eq:singlecurveexpression}
    \gamma\mapsto\int_I\int_{[0,1]} \frac{g(t,\Upsilon[\gamma_{t-},\gamma_{t+}](a))}{d(\gamma_{t-},\gamma_{t+})}\frac{\d }{\d a} (f\circ \Upsilon[\gamma_{t-},\gamma_{t+}])(a) \d a \d |\D\gamma|^{\mathrm{j}}(t).
\end{equation}
\\
   By \cref{lemma:measurability_variation}, 
   \[
\int |\D\gamma|^{\mathrm{j}}(\d t)\d\pi(\gamma)
   \]
is a well-posed Borel measure on $I\times D(I;X)$, and has finite total mass by \eqref{ineq:finiteEnergy}. 
\\
    Define the map 
    \begin{equation}
        \tilde e\colon I\times D(I;X)\to I\times X^2,\quad (t,\gamma)\mapsto (t,\gamma_{t-},\gamma_{t+}),
    \end{equation}
    which is Borel measurable by \cref{lma:evaluationBorel}.
    Then consider
    \begin{equation}
        \eta\coloneqq \tilde e_{\#} \left(\int |\D\gamma|^{\mathrm{j}}(\d t)\d\pi(\gamma)\right).
    \end{equation}
    We have $\eta$ a finite Borel measure on the Polish space $\bar I\times X^2$.
    In particular, $\eta$ is concentrated on those $(t,x,y)$ such that $x=\gamma_{t-}$ and $y=\gamma_{t+}$ for some $\gamma\in \spt(\pi)$ and $t\in J_\gamma$.\medskip
\\
  Now we will rewrite the integrands in \eqref{eq:singlecurveexpression} using the new coordinate of $\eta$.
  To this end, we define the following functions in $I\times X^2\times [0,1]$:
  \begin{align}
      &(t,x,y,a)\mapsto d(x,y)^{-1},\label{i}\tag{i}\\
      g_\eta\colon& (t,x,y,a)\mapsto g(t,\Upsilon[x,y](a)) \label{ii}\tag{ii}\\
      H\colon&(t,x,y,a)\mapsto \liminf_{n\to \infty} n\cdot \left( f\Big(\Upsilon[x,y](a+1/n)\Big)-f\Big(\Upsilon[x,y](a)\Big)\right)\label{iii}\tag{iii}
  \end{align}
It is clear that the function in \eqref{i} is Borel. 
We claim that the rest two functions are Souslin measurable i.e. measurable from the Souslin $\sigma$-algebra of the domain to the Borel $\sigma$-algebra of the codomain and therefore universally measurable; see e.g. \cite[Theorem 21.10]{Kechris95}.
\\
(ii) Recall that $\Upsilon$ is Souslin measurable from $X\times X$ to $\Geo(X)$.
Then the map $e_{\Upsilon}$, given as the composition of the following Souslin and Borel measurable maps, is Souslin
\begin{equation}
    e_{\Upsilon}\colon (x,y,a)\xmapsto{\text{Souslin}}  (\Upsilon[x,y],a)\xmapsto{\text{Borel}}  \Upsilon[x,y](a)\in X.
\end{equation}
Similarly, the function $g_\eta$ is Souslin as one can write
\begin{equation}
    g_\eta\colon (t,x,y,a)\xmapsto{(\mathrm{id},e_\Upsilon)}  (t,\Upsilon[x,y](a))\xmapsto{g} g(t,\Upsilon[x,y](a)).
\end{equation}
(iii) Denote by $D$ on $\Geo(X)\times [0,1]$ the function given as the liminf of the sequence $D_n$ by
\[
	D_n(\gamma,a)=
	\begin{cases}
	  n\cdot (f\bigl(\gamma(a+\tfrac1n)\bigr)-f\bigl(\gamma(a)\bigr)), & \text{if }a\in[0,1-\tfrac1n],\\[1.2em]
		0, & \text{otherwise}.
	\end{cases}
	\]
Since $f$ is continuous, $D_n$ is Borel on $\Geo(X)\times[0,1]$ for each $n$, and so is $D$.
Analogous to $g_\eta$, the function $H$ is Souslin as it is the composition of the following Souslin and Borel maps
\begin{equation}
    H\colon (t,x,y,a)\xmapsto{\text{Souslin}} (\Upsilon[x,y],a)\xmapsto{\text{Borel}} D(\Upsilon[x,y])(a). 
\end{equation}
As a consequence of the measurability of $g_{\eta}$ w.r.t.  $\eta\otimes \L^1|_{[0,1]}$, the following integral is well-posed and it satisfies
\begin{align}
    &\int\int^1_0 g_\eta(t,x,y,a)\d \L^1(a) \d\eta\eqqcolon \iiint g_\eta\circ \tilde e\d\L^1|_{[0,1]}\d |\D\gamma|^{\mathrm{j}}\d \pi\\
    =&\int\int_I\int_{[0,1]} g\Big(t,\Upsilon[\gamma_{t-},\gamma_{t+}](a)\Big)\d\L^1|_{[0,1]}(a)\d |\D\gamma|^{\mathrm{j}}(t)\d \pi(\gamma)\\
    =&\int \sum_{t\in J_\gamma} \int^1_0 g\Big(t,\Upsilon[\gamma_{t-},\gamma_{t+}](a)\Big) d(\gamma_{t-},\gamma_{t+})\d a\d\pi(\gamma) \\
    =&\int g \d \left(\sum_{t\in J_\gamma} \delta_t\otimes \mathcal{H}^1|_{\Upsilon[\gamma_{t-},\gamma_{t+}]}\right)\d\pi(\gamma).
\end{align}
This implies, together with the measurability of the continuous part, that $\int \lambda_\gamma\d\pi(\gamma)$ defines a finite Borel measure on $I\times X$, which we denote by $\lambda$. By monotone convergence, the integral identity extends to all
nonnegative Borel functions $g$, and hence to all
$g\in L^1(I\times X,\lambda)$\medskip
\\
Now, take $g$ an arbitrary $\lambda$-integrable Borel function.
We show that the function $d(x,y)^{-1}g_\eta\cdot H$ is $\eta\otimes \L^1$-integrable.
Indeed, as $|H(t,x,y,a)|\leq \Lip(f)\cdot d(x,y)$, we proceed as previously
\begin{align}
    &\int \int\frac{|g_\eta|(t,x,y,a)}{d(x,y)}|H|(t,x,y,a) \d a \d\eta \leq \Lip(f)\int\int^1_0 |g_\eta|\d \L^1\d \eta\\
    \eqqcolon & \Lip(f)\iiint |g_\eta\circ \tilde e|\d\mathcal{L}^1\d |\D\gamma|^{\mathrm{j}}\d \pi\leq \Lip(f)\|g\|_{L^1(\lambda)}<\infty.
\end{align}
By Fubini's theorem, the following integral as a function on $I\times X^2$
\begin{equation}
    (t,x,y)\mapsto\int\frac{g_\eta(t,x,y,a)}{d(x,y)}H(t,x,y,a) \d a 
\end{equation}
is $\eta$-measurable and integrable.
Finally, notice that $\frac{\d}{\d a} (f\circ \Upsilon[x,y])(a)=H(t,x,y,a)$ for $\L^1$-a.e. $a\in [0,1]$, we conclude that
\begin{equation}
     (t,x,y)\mapsto\int_{[0,1]} \frac{g(t,\Upsilon[x,y](a))}{d(x,y)}\frac{\d }{\d a} (f\circ \Upsilon[x,y])(a) \d a 
\end{equation}
is $\eta$-measurable and integrable, and thus, there exists a Borel version $G$ of it.
As a consequence,  \begin{align}
   G\circ \tilde e(t,\gamma)= \int_{[0,1]} \frac{g(t,\Upsilon[\gamma_{t-},\gamma_{t+}](a))}{d(\gamma_{t-},\gamma_{t+})}\frac{\d }{\d a} (f\circ \Upsilon[\gamma_{t-},\gamma_{t+}])(a) \d a
\end{align}
for $\int |\D\gamma|^{\mathrm{j}}(\d t)\d\pi(\gamma)$ -almost every $(t,\gamma)$.
In particular, it follows that for $\pi$-almost every $\gamma$, the above equality holds for $|\D\gamma|^{\mathrm{j}}$-almost every $t$.
Thus, 
\begin{align}
    \gamma\mapsto\int_I\int_{[0,1]} \frac{g(t,\Upsilon[\gamma_{t-},\gamma_{t+}](a))}{d(\gamma_{t-},\gamma_{t+})}\frac{\d }{\d a} (f\circ \Upsilon[\gamma_{t-},\gamma_{t+}])(a) \d a \d |\D\gamma|^{\mathrm{j}}(t)
\end{align}
agrees with the $\pi$-measurable and integrable function
\begin{align}
    \gamma\mapsto\int_IG\circ \tilde e(t,\gamma)\d |\D\gamma|^{\mathrm{j}}(t)
\end{align}
$\pi$-almost everywhere.
Hence, the desired measurability follows.
\end{proof}

\bibliographystyle{alpha}
\bibliography{ms}

\end{document}